\documentclass[11pt]{amsart}
\usepackage[margin=2.1cm]{geometry}
\usepackage[T2A]{fontenc}
\usepackage[utf8]{inputenc}

\usepackage{multirow}

\usepackage{float}

\usepackage{algorithm}
\usepackage{algpseudocode}

\usepackage{pgfplots}
\usepackage{subcaption} 
\pgfplotsset{compat=newest}

\pgfplotsset{
    myplotstyle/.style={
        width=\linewidth,
        height=5cm,
        xlabel={$\sigma$},
        ylabel={$\Delta(p, \sigma)$},
        grid=major,
        grid style={dotted, gray!50},
        tick label style={font=\footnotesize},
        label style={font=\small},
        title style={font=\small\bfseries}
    }
}

\usepackage{booktabs}
\usepackage{amsmath} 

\usepackage{macros}

\def \le {\leqslant}
\def \ge {\geqslant}

\usepackage{tikz}

\title{Dirichlet improvability in $L_p$-norms}

\begin{document}


 
\author[Nikolay Moshchevitin]{Nikolay Moshchevitin}
\address{Nikolay Moshchevitin,  Institute für diskrete Mathematik und Geometrie,  Technische Universität Wien, Wien 1040, Austria. }
\email{nikolai.moshchevitin@tuwien.ac.at, moshchevitin@gmail.com}

\author[Nikita Shulga]{Nikita Shulga}
\address{Nikita Shulga, Sydney Mathematical Research Institute, The University of Sydney, NSW, Australia and  Department of Mathematical and Physical Sciences,  La Trobe University, Bendigo, Australia. }
\email{nikita.shulga@sydney.edu.au}

\maketitle

\begin{abstract}
  For a norm $F$ on $\R^2$, we consider the set of $F$-Dirichlet improvable numbers $\mathbf{DI}_F$. In the most important case of $F$ being an $L_p$-norm with $p=\infty$, which is a supremum norm, it is well-known that $\mathbf{DI}_F = \mathbf{BA}\cup \Q$, where $\mathbf{BA}$ is the set of badly approximable numbers. It is also known that $\mathbf{BA}$ and each $\mathbf{DI}_F$ are of measure zero and of full Hausdorff dimension.

  Using the classification of critical lattices for unit balls in $L_p$, we provide a complete and effective characterization of $\mathbf{DI}_p:=\mathbf{DI}_{F^{[p]}}$ in terms of the occurrence of patterns in regular continued fraction expansions, where $F^{[p]}$ is an $L_p$-norm with $p\in[1,\infty)$. This yields several corollaries. In particular, we resolve two open questions by Kleinbock and Rao by showing that the set $\mathbf{DI}_{p}\setminus \mathbf{BA}$ is of full Hausdorff dimension, and prove several results about the size of the difference $\mathbf{DI}_{p_1}\setminus \mathbf{DI}_{p_2}$. To be precise, we show that both $\mathbf{DI}_{2}\setminus \mathbf{DI}_{1}$ and $\mathbf{DI}_{1}\setminus \mathbf{DI}_{2}$ have full Hausdorff dimension.
  We also find all values of $p$, for which the set $\mathbf{DI}_p^c\cap\mathbf{BA}$ has full Hausdorff dimension.

  Finally, our characterization result implies that the number $e$ satisfies $e\in \mathbf{DI}_p$ if and only if $p\in(1,2)\cup(p_0,\infty)$ for some special constant $p_0\approx2.57$.

\end{abstract}

 \section{Introduction}\label{sec:intro}


\subsection{Classical background} \label{bck}
 The simplest result in the theory of Diophantine approximation is Dirichlet's Approximation Theorem, which can be formulated as follows.

   {\bf Theorem A.}
   {\it Let $\alpha \in \mathbb{R}$. For any real $ t \ge 1$ there exists  positive integer $q$ such that 
   $$
   \begin{cases}
   ||q\alpha || <\frac{1}{t},
   \cr
   1\le q\le t.
   \end{cases}
   $$
   }
Here $||\cdot||$ is the distance to the nearest integer. It is easy to see that one cannot uniformly improve the constant $1$ on the right-hand side of the first inequality. 
However, for some particular numbers, Theorem A can be improved. A real number $\alpha$ is  called {\it Dirichlet improvable} if there exists a constant $c<1$, such that the system
$$
\begin{cases}
   ||q\alpha || <\frac{c}{t},
   \cr
   1\le q\le t.
   \end{cases}
   $$
can be solved  in $q\in\N$ for any large real number $t$. 
We denote the set of all real  Dirichlet improvable numbers by $\mathbf{DI}_\infty$.

Recall that 
$\alpha$ is called badly approximable if there exists $c>0$ for which  the inequality 
$$\left|\alpha-\frac{p}{q}\right| >\frac{c}{q^2}$$
holds for all rational $\frac{p}{q}$. We denote the set of all badly approximable real numbers by $\mathbf{BA}$.
Classical results state that both $\mathbf{DI}_\infty$ and $\mathbf{BA}$ have zero Lebesgue measure and are of full Hausdorff dimension.

Davenport and Schmidt in \cite{MR0272722} showed that these two sets are connected. One of their results is the following statement.

    {\bf Theorem B.} {\it
An irrational number $\alpha$ satisfies $\alpha\in\mathbf{DI}_\infty$ if and only if it is badly approximable.}

It is well-known that bad approximability is equivalent to the property of having bounded partial quotients in a regular continued fraction expansion, therefore, the last theorem can be restated as follows.

  {\bf Theorem B$'$}. {\it
An irrational number $\alpha=[a_0;a_1,\ldots,a_n,\ldots]$ is Dirichlet non-improvable if and only if there is a subsequence $\{n_k\}_{k\in\N}$, such that $a_{n_k}\to\infty$ as $k\to\infty$.}
 
\subsection{Dirichlet improvability  with respect to a norm}

Here we describe a setting formulated by Kleinbock and Rao  \cite{KR}
 which goes back to the paper \cite{AD} by Andersen and Duke. 
 Our exposition is a bit different from the one in \cite{AD,KR}, however, our objects are equivalent to the ones from these papers.
 
 We consider a norm $F(x,y)$ in $\mathbb{R}^2 (x,y)$ such that 
  $$  
  F(x,y) = F(|x|,|y|)\,\,\,\,
  \text{and}\,\,\,\,  F(1,0)= F(0,1) = 1.
  $$
  In \cite{AD} such a norm is called {\it strongly symmetric}.
Denote the unit disk with respect to the norm $F$ as
    $$ 
    \mathcal{B}_F  = \{ (x,y)\in \mathbb{R}^2: F(x,y) \le1\}.
    $$

         Dani's correspondence \cite{MR794799} allows to reformulate many diophantine statements in terms of trajectories in the space of two-dimensional lattices, and we use this principle below. For a fixed $\alpha \in \mathbb{R}$, we consider the family of lattices
  \begin{equation}\label{lata}
  \Lambda_\alpha (t) = G_t A_\alpha\mathbb{Z}^2,\,\,\,
  \text{where}\,\,\,
  G_t =\left(
  \begin{array}{cc}
  t^{-1} &0\cr
  0&t
  \end{array}
  \right),
  \,\,\,
  A_\alpha =
  \left(
   \begin{array}{cc}
  1 &0\cr
  -\alpha&1
  \end{array}
  \right)
 \end{equation}
 and successive minima
 \begin{equation}\label{manos}
 \lambda_i (t) =
 \lambda_i (\Lambda_\alpha (t), \mathcal{B}_F),\,\,\, i=1,2
 \end{equation}
 of the  lattice $\Lambda_\alpha (t)$ with respect to the convex body $\mathcal{B}_F$ as functions in $t$. 
 For every $t$, the fundamental volume of the lattice $\Lambda_\alpha (t)$  is equal to one.

  For the unit disk $\mathcal{B}_F$ we consider its {\it critical determinant}
  $$
  \Delta_F = \inf \{ {\rm det} \,\Lambda: \,\text{there are no non-zero points of $\Lambda$ inside }\, \mathcal{B}_F\}.
  $$
  By compactness, this infimum is attained on some lattice which is called a {\it critical lattice}.  
  The set of all critical lattices 
 for $\mathcal{B}_F$ forms a {\it critical locus} which is denoted by $ \frak{L}_F$.


 Consider the value defined by 
 \begin{equation}\label{Drlo}
  d_F  (\alpha) = \limsup_{t\to \infty}  \lambda_1 (t) =  
  \limsup_{t\to \infty} 
   \lambda_1(\Lambda_\alpha (t), \mathcal{B}_F)
 .
 \end{equation}
  We call $d_F(\alpha)$ a {\it Dirichlet constant of $\alpha$ for the norm} $F$. 
  In \cite{AD}, the authors dealt with the analogous values $\delta_F(\alpha)$.
  In our notation, we have
  $ \delta_F(\alpha) =\Delta_F  \cdot (   d_F  (\alpha))^2$,
  where $\Delta_F  $ is the {\it critical determinant} of $\mathcal{B}_F$.
  It is clear that  $ \delta_F(\alpha)\le 1$ (or $ d_F(\alpha)\le \frac{1}{\sqrt{\Delta_F}}$) for all $\alpha$. 

 A real number $\alpha$ is called {\it $F$-Dirichlet improvable} if $d_F(\alpha) < \frac{1}{\sqrt{\Delta_F}}$ or, equivalently, $\delta_F(\alpha) <1$.   Conversely, in the case 
  $d_F(\alpha) =\frac{1}{\sqrt{\Delta_F}}$   (or $\delta_F(\alpha) =1$) we say that $\alpha$ is {\it $F$-Dirichlet non-improvable}.

For a fixed norm $F$, we denote the set of $F$-Dirichlet improvable numbers by $\mathbf{DI}_F$.  For the $L_\infty$-norm, this definition corresponds to regular Dirichlet improvability discussed in Subsection \ref{bck} above.

In the present paper, we thoroughly study the Dirichlet improvability in the most important case of $L_p$-norms
  \begin{equation}\label{LP}
  F^{[p]}(x,y) =  (|x|^p+|y|^p)^{\frac{1}{p}}.
  \end{equation}
  Sup-norm, Euclidean norm, and taxicab norm
  $$
    F^{[\infty]}(x,y) =  \max (|x|, |y|),\,\,\, F^{[2]}(x,y) =  (|x|^2+|y|^2)^{\frac{1}{2}},\,\,\,  F^{[1]}(x,y) =  |x|+|y|
    $$
    are particular examples.
    
 For a fixed $p$ and norm $F^{[p]}$, we simply say that $\alpha$ is $p$-Dirichlet improvable (or is $p$-Dirichlet non-improvable) and denote the set of $p$-Dirichlet improvable numbers by $\mathbf{DI}_p$. The set of $p$-Dirichlet non-improvable numbers is denoted by $\mathbf{DI}_p^c$.  

In this paper, we tackle three problems. The main one is to deal with the following question.
 
{\bf Problem I.} {\it Can one classify Dirichlet (non-)improvable numbers for other norms similarly to Theorem B$'$? } 
  
It turns out that the answer to the latter is positive, and we provide a complete and effective classification result of $p$-Dirichlet improvable numbers, see Theorem \ref{thm:charachterization} below. This result deeply relies on the classification of critical lattices for unit balls $\mathcal{B}_p$ in $L_p$, completed by Glazunov, Golovanov, and Malyshev in \cite{maly}.

 Before formulating other problems, we recall an important metrical result that follows from Ergodic theory (see Theorem 1 from \cite{AD} by Andersen and Duke).

    {\bf Theorem C. }
    {\it  For every strongly symmetric norm $F$, almost all $\alpha$ in the sense of Lebesgue measure  are not $F$-Dirichlet improvable, that is, for almost all $\alpha$ we have the equality
    $
    d_F(\alpha) =      \frac{1}{\sqrt{\Delta_{F}}}
     $.}
               \vskip+0.3cm

                             One more crucial result was proven in \cite{KR}.
                             It has a general part that deals with {\it irreducible norms} and a corollary for the Euclidean norm.
 
      \vskip+0.3cm
 
    {\bf Theorem D.}  {\it If $F$ is an irreducible norm on $\mathbb{R}^2$ whose unit ball is not a parallelogram, then the set of all badly approximable $F$-Dirichlet non-improvable numbers $\mathbf{DI}^c_F\cap\mathbf{BA}$
    has full Hausdorff dimension. In particular,   the set of all badly approximable $2$-Dirichlet non-improvable numbers $\mathbf{DI}^c_2\cap\mathbf{BA}$
    has full Hausdorff dimension.}
       \vskip+0.3cm

       For the definition of {\it irreducible norm} see the original paper  \cite{mh} and the recent one  \cite{KKR}. We do not want to go into the details here.
    We should note that for $L_p$-norms, Theorem D gives a result only in the case $p=2$, because
      $L_p$-norms for $ p\neq 1,2$ are not irreducible (see Lemma 8 from \cite{mh}), while for $p=1$ the unit ball $\mathcal{B}_1$ is a parallelogram. 

         Also, in \cite{MR4410764} Kleinbock and Rao proved the following result.
         
 {\bf Theorem E.}
    {\it For each norm $F$ on $\R^2$, the set $\mathbf{DI}_F$ is of measure zero but winning in the sense of Schmidt. In particular, $\mathbf{DI}_F$ has full Hausdorff dimension.
}
  \vskip+0.3cm

As the set of badly approximable numbers $\mathbf{BA}$ is also winning, we get that for any norm $F$, the intersection $\mathbf{DI}_F \cap \mathbf{BA}$ is winning and therefore it has full Hausdorff dimension.

Moreover, in \cite{KR} Kleinbock and Rao proved that either 
$\mathbf{BA}\subset \mathbf{DI}_F$ or else $\mathbf{BA}\setminus \mathbf{DI}_F$ has full Hausdorff dimension.

This motivated them to formulate the following open question, which is stated in \cite{KR}. 


{\bf Problem II.} {\it Does the difference $\mathbf{DI}_2 \setminus \mathbf{BA}$ have full Hausdorff dimension?} 

In this paper, we resolve this question in a more general form for any $L_p$-norm with $p\in[1,\infty)$. Namely, we prove 

\begin{theorem}\label{theorem:diminusbad:dimension}
For any $p\in [1,\infty)$, the set $\mathbf{DI}_p \setminus \mathbf{BA}$ is of full Hausdorff dimension.
\end{theorem}

Note that for $p=\infty$, the statement of Theorem \ref{theorem:diminusbad:dimension} does not hold, as $\mathbf{DI}_\infty \setminus \mathbf{BA}=  \Q$.

In fact, Theorem \ref{theorem:diminusbad:dimension} is a consequence, among many other corollaries, of the aforementioned complete characterization of $p$-Dirichlet improvability of $\alpha$  in terms of regular continued fraction expansion (\ref{conti}) for any $p\in [1,\infty)$, that is Theorem \ref{thm:charachterization}.

 In \cite{KR}, the authors also formulated the following open problem about the sets of Dirichlet improvable numbers.
  
   {\bf Problem III.} {\it What can be said about the size of $\mathbf{DI}_{\nu_1}\setminus\mathbf{DI}_{\nu_2}$ for two different norms $\nu_1,\nu_2$? }

Our methods allow us to answer this question for two $L_p$-norms. In this introduction, we state the result for the two most interesting cases $\mathbf{DI}_{2}\setminus\mathbf{DI}_{1}$ and $\mathbf{DI}_{1}\setminus\mathbf{DI}_{2}$, for which we prove full dimension results. We use the notation $\hdim X $ for the Hausdorff dimension of a set $X$. 

\begin{theorem}\label{thm:di2minusd1:di1minusd2}
We have
$$
\hdim \mathbf{DI}_{2}\setminus\mathbf{DI}_{1}= \hdim \mathbf{DI}_{1}\setminus\mathbf{DI}_{2} =1.
$$
    
\end{theorem}

 Other results on this problem are listed in Theorem \ref{123} in Section \ref{sec:mainresults}.
 \vskip+0.3cm

\subsection{Structure of the paper} \label{str}
  Our paper is organized as follows. 
In Section \ref{axre}, we state basic facts and definitions concerning continued fractions, in particular, we recall different definitions of bad approximability
(end of Subsection \ref{sec:cf:and:bad}) and give interpretations of the values  $d^{[\infty]} (\alpha)$ and  $d^{[1]} (\alpha)$ (which are quantities \eqref{Drlo} for supremum norm and $L_1$ norm respectively)
in terms of irrationality measure function $\psi_\alpha(t)$ and  Minkowski's function $ \mu_\alpha (t)$ respectively
(Subsection \ref{sec:norms:lattices:dirichletconst}).
In Section \ref{sec:mainresults}, we formulate our main result - the classification 
Theorem \ref{thm:charachterization}, which gives a complete description of the sets  $\mathbf{DI}^c_p$ for all values of $p$ in terms of regular continued fractions of its elements.
Besides Theorems \ref{theorem:diminusbad:dimension} and \ref{thm:di2minusd1:di1minusd2} formulated above, our main result has many other corollaries, which are formulated in Section  \ref{sec:mainresults} as 
Corollary \ref{coroll:bad}, Theorem \ref{thm:bad:dirichlet:fulldim}, Theorem \ref{123} and Corollary \ref{coroll:setofparameters}.
We also formulate and prove non-metrical Corollary \ref{coroll:number:E} which deals with the number $e$.

The rest of the paper deals with the proofs of our results. In Section \ref{compaa} we recall the compactness argument related to the convergence in the space of lattices, which reduces the problem to consideration of critical lattices for the sets $\mathcal{B}_{p}$.
In Section \ref{sec:lpnorms:param}, based on the final solution of Minkowski conjecture \cite{maly}, we describe the complete structure of sets of critical lattices for $\mathcal{B}_{p}$ for all values of $p\ge 1$. 
In Section \ref{ffiii} we finalize the proof of Theorem \ref{thm:charachterization}.
In Section \ref{sec:basic:cf} we state all necessary facts concerning continued fractions, which we later use for the Hausdorff dimension bounds for the sets under consideration.
In Section \ref{sec:proof:bad:dirichlet:fulldim} we give a proof of Theorem \ref{thm:bad:dirichlet:fulldim} which is formulated in Section 
\ref{sec:mainresults}. In Section \ref{sec:di:minus:bad} we give proofs of
 Theorem \ref{theorem:diminusbad:dimension} and Theorem \ref{thm:di2minusd1:di1minusd2}.
In Section \ref{aadddd} we prove the main lemma, which allowed us to get the Hausdorff dimension in Theorems \ref{theorem:diminusbad:dimension} and \ref{thm:di2minusd1:di1minusd2}. 
 In Section \ref{sec:proof:setofparameters} we prove 
 Corollary \ref{coroll:setofparameters}, which is formulated at the end of Section \ref{sec:mainresults}. Finally, Appendix \ref{appendinx:1} is devoted to the independent verification of Minkowski's conjecture on critical lattices, with Appendix \ref{appendinx:2} containing pseudocode for some of the programs we used to verify this conjecture.

\subsection{Acknowledgements}

The authors thank the Number Theory group of La Trobe University, Bendigo for useful discussions. The authors thank the referee for many useful suggestions, in particular, proposing the creation of the Appendix with independent verification of the proof of Minkowski's conjecture. The authors are also grateful to the creators and contributors to IntervalArithmetic.jl Julia library for clarifying some specifics on the library directly to us. 

Nikolay Moshchevitin was supported by the Austrian Science Fund (FWF), Forschungsprojekt PAT1961524. Nikita Shulga was supported by a La Trobe University Graduate Research Scholarship and a La Trobe University Full Fee Research Scholarship.

   \vskip+0.3cm

 \section{Auxiliary remarks}\label{axre}

 In this section, we collect together some basic facts related to the application of continued fractions theory to Diophantine Approximation and discuss the relations with Dirichlet improvability.

 \subsection{Continued fractions and bad approximability}\label{sec:cf:and:bad}
  \vskip+0.3cm
  
 Let $\alpha$ be a real irrational number. We consider its continued fraction expansion
\begin{equation}\label{conti}
 \alpha = [a_0; a_1,a_2,\ldots, a_\nu, a_{\nu+1},\ldots],\,\,\,\,\,
 a_0 \in \mathbb{Z}, 
 \,\,
 a_j \in \mathbb{Z}_+, \, j =1,2,3,\ldots
 \end{equation}
 as well as the irrationality measure function
 \begin{equation}\label{irrational}
 \psi_\alpha (t ) = \min_{q\in \mathbb{Z}_+: \, q \le t} ||q\alpha ||,\,\,\,\,
 \text{where}\,\,\,\,
 ||x || = \min_{a\in \mathbb{Z}} |x - a|.
 \end{equation}
 Let us consider $\nu$-th convergent to $\alpha$ defined as
 $$
 \frac{p_\nu}{q_\nu} = [a_0; a_1,\ldots, a_\nu]
 $$
 and the values
 \begin{equation}\label{fromu}
   \xi_\nu = ||q_\nu\alpha|| = |p_\nu - q_\nu \alpha|,\,\,\,\,
  \alpha_\nu  = [ a_\nu; a_{\nu+1}, a_{\nu+2},\ldots] =\frac{\xi_{\nu-2}}{\xi_{\nu-1}},
  \,\,\,\,
  \alpha_\nu^* = \frac{q_{\nu-1}}{q_\nu} =  [0;  a_\nu , a_{\nu-1},\ldots, a_2, a_1].
  \end{equation}
    In this paper, we  also use the notation
 \begin{equation}\label{ze}
  \frak{z}_\nu = (q_\nu, p_\nu - q_\nu \alpha) \in \mathbb{R}^2 .
 \end{equation}
   In particular,  
   for the points defined in (\ref{ze}) and for the lattice defined in (\ref{lata}) 
   we have $\frak{z}_\nu \in \Lambda_\alpha (1)$.
 
For the values of the irrationality measure function, we have a nice equality
 \begin{equation}\label{psi}
  \psi_{\alpha}(t) = \xi_\nu   = \frac{1} {q_\nu (\alpha_{
  \nu+1} + \alpha_\nu^*)}
  =  \frac{1} {q_{\nu+1} ( 1+ \alpha_{
  \nu+2}^{-1}  \alpha_{\nu+1}^*)}
  \,\,\,
  \text{for}\,\,\, q_\nu \le t <q_{\nu+1},
  \end{equation}
  which is sometimes called Perron's formula, 
  and the upper bound
  $$
   \psi_{\alpha}(t)  < \frac{1}{t} \,\,\,\,\, \forall\, t \ge 1.
   $$
   The last inequality is equivalent to Dirichlet's theorem (Theorem \textbf{A}).

       \vskip+0.3cm

  For our purposes, the Lagrange constant
  $$
  \lambda(\alpha) = \liminf_{t\to\infty} t \psi_\alpha (t) =
   \liminf_{\nu\to\infty} q_\nu \xi_\nu=  \liminf_{\nu\to\infty} 
    \frac{1} {\alpha_{
  \nu+1} + \alpha_\nu^*}
  $$
  and  the Dirichlet constant
  \begin{equation}\label{diri}
  d(\alpha) = \limsup_{t\to\infty} t \psi_\alpha (t) =
   \limsup_{\nu\to\infty} q_{\nu+1} \xi_\nu =  \limsup_{\nu\to\infty} 
    \frac{1} { 1+\alpha_{
  \nu+1}^{-1} \cdot \alpha_\nu^*}
  \end{equation}
  are of major importance.
  The theorem of Hurwitz states that for all $\alpha$, one has
  $$
  0\le 
  \lambda(\alpha ) \le \frac{1}{\sqrt{5}},
  $$
  while Szekeres' theorem \cite{szekeres} claims that for all  irrational $\alpha$ one has
    $$
  \frac{1}{2} + \frac{1}{2\sqrt{5} } \le
  d(\alpha ) \le 1.
  $$
 There is an extended theory which deals with the study of the sets of possible values of $\lambda(\alpha)$ and $d (\alpha)$
 which are called  the {\it Lagrange spectrum} and the {\it Dirichlet spectrum} respectively.
 
   \vskip+0.3cm

  It is well-known that the following five  conditions for irrational $\alpha$ are equivalent:
    \vskip+0.3cm
    {\bf A.} $\alpha \in \mathbf{BA}$;
    
        \vskip+0.1cm
    
        {\bf B.}   $\inf_{q\in \mathbb{Z}_+} q ||q\alpha|| >0$;
        
             \vskip+0.1cm
        
             {\bf C.}   $\sup_{\nu \in \mathbb{Z}_+} a_\nu  <\infty$;
                  \vskip+0.1cm
             
                 {\bf D.}   $\lambda (\alpha) >0$;
                 
                      \vskip+0.1cm
                 
                    {\bf E.}   $d (\alpha) <1$.

        \vskip+0.3cm    
        In particular, 
  the equivalence {\bf A} \,$\Longleftrightarrow$\, {\bf E} is a reformulation of Theorem B, whereas the equivalence {\bf C} \,$\Longleftrightarrow$\, {\bf E} is practically a Theorem B$'$.

  Here we should note that Theorem B and its multidimensional generalization were proven by Davenport and Schmidt in \cite{MR0272722}.

  Beresnevich et al. in \cite{MR4395950} deepened the result of Davenport and Schmidt, by showing that the set
$$
\mathbf{FS}(n):=\mathbf{DI}(n) \setminus \Bigl(\mathbf{BA}(n) \cup \mathbf{Sing}(n) \Bigr)
$$
for $n\ge2$ has continuum many points. Here $\mathbf{DI}(n)$, $\mathbf{BA}(n)$ and $\mathbf{Sing}(n)$ are $n$-dimensional analogues (in a sense of Diophantine properties) of $\mathbf{DI}_\infty$, $\mathbf{BA}$ and $\Q$ respectively. 
As it was mentioned in \cite{MR4395950}, in the case $n=2$ this result also follows from a complete description of the two-dimensional Dirichlet spectrum
discovered by Akhunzhanov and Shatskov \cite{MR3284117}, see also \cite{MR4449836}.
Our Theorem \ref{theorem:diminusbad:dimension} shows that if we stay in dimension one, but change the norm from supremum norm to other $L_p$-norms, the corresponding set also has continuum many points. Moreover, it has full Hausdorff dimension. The same full dimension result is conjectured for $\mathbf{FS}(n)$ with $n\ge2$ in  \cite{MR4395950}.

       \vskip+0.3cm
      All the facts mentioned above can be found in the classical books 
      \cite{Cas,Cus,RockettSus,Sch}.  
      In particular, a lot is known about the Lagrange spectrum  
(see \cite{Cus,MR3815461}). As for the Dirichlet spectrum, one may 
consider  papers
  \cite{I1,I2,I3} and references therein.

  \subsection{Relations with different norms}\label{sec:norms:lattices:dirichletconst}
  
 For simplicity of notation, in the case of $L_p$-norms we use the convention and write $\mathcal{B}_{p}=\mathcal{B}_{F^{[p]}}$, $ \frak{L}_p= \frak{L}_{F^{[p]}}$. Also, for the $L_p$-norm $F^{[p]}$  we use the notation 
 $$ d^{[p]}(\alpha)
 = \limsup_{t\to \infty}   \lambda_1(t) .
 $$
 
 First, we note that for minima $\lambda_i(t)$ defined in \eqref{manos}, we get that the dilated body
 $
 \lambda_1 (t) \cdot \mathcal {B}_F
 $
 for any $t$ has a non-trivial  point of $ \Lambda_\alpha (t) $ on its boundary, while the body
  $
 \lambda_2 (t)\cdot \mathcal {B}_F
 $
 contains a pair of linearly independent points of $ \Lambda_\alpha (t) $.

Also note that if $\alpha \in \mathbb{Q}$, then $\lambda_1 (t) \to 0$ when $ t\to \infty$.
Thus, $d_F  (\alpha)=0$ for any rational $ \alpha $ and for any norm $F$, so every rational $\alpha$ is trivially $F$-Dirichlet improvable for any $F$.

Now consider the following two examples.

When our norm is a sup-norm $F^{[\infty]}$, one can easily see that 
 $$
 (\lambda_1(t_n))^2 = q_{n+1} ||q_n\alpha|| = \lim_{t \to q_{n+1}-} t\psi_\alpha (t)     $$
  for the irrationality measure function defined in (\ref{irrational}).
  In other words, taking into account the second equality from (\ref{psi}), one has
  $$
 (\lambda_1(t_n))^2 =  D\left(\frac{1}{\alpha_{n+2}}, \alpha_{n+1}^*\right),\,\,\,\,
  \text{where}\,\,\,\,
  D(x,y) = \frac{1}{1+xy}.
  $$
  Therefore, the definition of Dirichlet constant (\ref{diri}) may be rewritten as
  $$
  d(\alpha)  =   ( \limsup_{t\to \infty} \lambda_1(t))^2  =    (d^{[\infty]} (\alpha))^2.
  $$
  The value  $d^{[2]} (\alpha)$ was studied in \cite{KR}.  The value $d^{[1]} (\alpha)$ is related to the so-called Minkowski spectrum $\mathbb{M}$ (see \cite{momi}), associated with Minkowski diagonal continued fraction. In particular, in the notation of \cite{momi}, one has
  $$
   \frak{m}  (\alpha ) =   \limsup_{t\to \infty} t\cdot \mu_\alpha (t) =
    \frac{ (d^{[1]} (\alpha))^2}{4} .
  $$
\vskip+0.3cm

Also, we should note 
  that the set of values of $t$ for which
 $$
  \lambda_1 (t) = \lambda_2 (t)
  $$
  forms an infinite  strictly increasing sequence $ \{t_n\}_{n\in\N}$. Moreover,
  \begin{equation}\label{Drl}
  d_F  (\alpha) = \limsup_{n\to \infty}   \lambda_1(t_n)  =  \limsup_{n\to \infty}   \lambda_2(t_n)
 .
 \end{equation}
   A classical geometric construction which goes up to Hermite \cite{her} and  Minkowski \cite{Mi,Mi1}
   describes an algorithm of constructing the sequence $\{t_n\}_n$. It is related to certain generalizations of regular continued fractions.
  This construction was explained in detail in a recent paper by Andersen and Duke \cite{AD}. 
  
        \vskip+0.3cm
    \section{Main results and their corollaries}\label{sec:mainresults}

  \vskip+0.3cm

We begin this section with the formulation of our main classificational result.
  The proof of this result is based on the classification of critical lattices for $\mathcal{B}_p$, completed by Glazunov, Golovanov, and Malyshev in \cite{maly}. Davis \cite{D} observed that the structure of the set of critical lattices $\frak{L}_{p}$ differs depending on whether $p$ is greater or smaller than $p_0$, where $p_0\approx2.57$ is some constant. The detailed description of critical lattices is given in Section \ref{sec:lpnorms:param}.
        \vskip+0.3cm

    \begin{theorem}\label{thm:charachterization}

   {\rm  ({\bf a})}
 Let $ 2< p < p_0$. Then $\alpha\in\mathbf{DI}_p^c$
 if and only if in the continued fraction for $\alpha$ there are patterns of the type
         \begin{equation}\label{trip0}
         x,1,1,y\,\,\,\,\,\text{or}\,\,\,\,\, x,2,y
        \end{equation}
         with $\min(x,y)\to \infty$.

       {\rm  ({\bf b})}
     Let $ p \in (1, 2)\cup (p_0, \infty)$. Consider $\varsigma_p$ defined in (\ref{lat3}).

{\rm ({\bf b1})}  If $\varsigma_p \in \mathbb{Q}$, 
consider its regular  finite continued fraction expansion
  $$
  \varsigma_p = [0;s_1,s_2,\ldots,s_k], s_k \ge 2.
  $$
Then  $\alpha\in\mathbf{DI}_p^c$ if and only if in its continued fraction expansion (\ref{conti}) there
    occur patterns of at least one of the following four forms:
$$
   x, s_k,s_{k-1}, \ldots,s_2, s_1,
    1,s_1,s_2,\ldots , s_{k-1},s_k ,y;
$$
$$
    x, 1,s_k-1,s_{k-1}, \ldots,s_2, s_1,
    1,s_1,s_2,\ldots , s_{k-1},s_k ,y;
$$
     $$ 
    x,s_k,s_{k-1}, \ldots,s_2, s_1,
    1,s_1,s_2,\ldots , s_{k-1},s_k-1 ,1,y;
    $$
    $$
    x, 1,s_k-1,s_{k-1}, \ldots,s_2, s_1,
    1,s_1,s_2,\ldots , s_{k-1},s_k-1,1 ,y
$$ 
    with $\min (x,y) \to \infty$.

  {\rm ({\bf b2})} If $\varsigma_p \not\in \mathbb{Q}$, consider its regular continued fraction expansion
  $$
  \varsigma_p = [0;s_1,s_2,\ldots,s_\nu,\ldots].
  $$
   Then $\alpha\in\mathbf{DI}_p^c$ if and only if in its continued fraction expansion (\ref{conti}) there
    occur palindromic  patterns  of the form
    \begin{equation}\label{pat}
    s_\nu,s_{\nu-1}, \ldots,s_2, s_1,
    1,s_1,s_2,\ldots , s_{\nu-1},s_\nu
    \end{equation} 
    with arbitrarily large values of $\nu$.

       {\rm  ({\bf c})}
     Number $\alpha\in\mathbf{DI}_1^c$ if and only if there exists a  sequence of positive integers $\{b_n\}_{n\in\N}$, such that
     either the continued fraction expansion of $\alpha$ contains almost symmetric patterns
     \begin{equation}\label{asym}
     b_\nu, b_{\nu-1},\ldots,b_2, b_1,1,1, b_1+1, b_2,\ldots, b_{\nu-1},b_\nu\,\,\,\,\,
     \text{or}\,\,\,\,\,
          b_\nu,b_{\nu-1}, \ldots,b_2, b_1+1,1,1, b_1, b_2,\ldots, b_{\nu-1},b_\nu
          \end{equation}
          with arbitrarily large values of $\nu$ 
          \\  or  a sequence of patterns of at least one of  the following eight forms:
$$
 x, b_\nu, b_{\nu-1},\ldots,b_2, b_1,1,1, b_1+1, b_2,\ldots, b_{\nu-1},b_\nu,y;
$$
$$
  x, b_\nu, b_{\nu-1},\ldots,b_2, b_1,1,1, b_1+1, b_2,\ldots, b_{\nu-1},b_\nu-1,1,y;
$$
     $$  x,1, b_\nu-1, b_{\nu-1},\ldots,b_2, b_1,1,1, b_1+1, b_2,\ldots, b_{\nu-1},b_\nu,y;
    $$
    $$
  x,1, b_\nu-1, b_{\nu-1},\ldots,b_2, b_1,1,1, b_1+1, b_2,\ldots, b_{\nu-1},b_\nu-1,1,y;
$$ 
$$
  x, b_\nu,b_{\nu-1}, \ldots,b_2, b_1+1,1,1, b_1, b_2,\ldots, b_{\nu-1},b_\nu,y;
   $$
   $$
  x,1, b_\nu-1 ,b_{\nu-1}, \ldots,b_2, b_1+1,1,1, b_1, b_2,\ldots, b_{\nu-1},b_\nu,y;
   $$
      $$
  x, b_\nu,b_{\nu-1}, \ldots,b_2, b_1+1,1,1, b_1, b_2,\ldots, b_{\nu-1},b_\nu-1,1,y;
   $$
      $$
  x,1, b_\nu-1 ,b_{\nu-1}, \ldots,b_2, b_1+1,1,1, b_1, b_2,\ldots, b_{\nu-1},b_\nu-1,1,y
   $$
          with fixed $\nu , b_1,\ldots,b_\nu$ and $ \min(x,y) \to \infty$, 
          \\ or patterns $x,2,y$  with $ \min(x,y) \to \infty$,   or patterns $x,1,1,y$  with $ \min(x,y) \to \infty$.
          
          {\rm  ({\bf d})}    Number $\alpha\in\mathbf{DI}_2^c$ if and only if either in continued fraction for $\alpha$ there are patterns of the type
         $$
         x,1,1,y\,\,\,\,\,\text{or}\,\,\,\,\, x,2,y
         $$
         with $\min(x,y)\to \infty$
          or
          \\
          there exist two  irrational numbers 
                 \begin{equation}\label{aabb}
          \beta^* = [b_0^*; b_1^*,b_2^*,\ldots,b_{\nu-1}^*, b_\nu^*,\ldots],\,\,\,\,
          \beta = [b_0; b_1,b_2,\ldots, b_{\nu-1}, b_\nu,\ldots],\,\,\,\ b_0^*,b_0\ge 0
          \end{equation}
          satisfying the equation
          \begin{equation}\label{rararar}
          \beta\cdot\beta^*=3,
          \end{equation}
           such that in the continued fraction expansion of $\alpha$, there exist patterns
      \begin{equation}\label{asym1}
               b^*_{\nu}, \ldots,b^*_1, b^*_0+1,1, b_0+1, b_1,\ldots,b_\nu\,\,\,\,\,
     \text{or}\,\,\,\,\,
          b_\nu,\ldots,b_1, b_0+1,1, b^*_0+1, b^*_1,\ldots, b^*_{\nu}
          \end{equation}
          with arbitrarily large values of $\nu$
 or
 \\
 there exist two rational numbers $\beta^*$ and $\beta$ satisfying (\ref{rararar}) and in  the continued fraction expansion of $\alpha$  
 there exists  a  sequence  of one of the eight patterns 
  with $\min(x,y)\to \infty$
  constructed from $\beta^*=[b_0^*;b_1^*,\ldots,b_k^*], \beta=[b_0;b_1,\ldots,b_\nu]$ similarly to those from statement  {\rm  ({\bf c})}\footnote{ If $\beta^*, \beta$ are both rational, then their continued fractions are not necessarily of the same length. One should interpret the pattern structure in the following way.
  
  For example, one of the eight patterns is $x,b^*_{k},\ldots,b^*_1, b^*_0+1,1, b_0+1, b_1,\ldots, b_\nu,y$ with $\min(x,y)\to\infty$ and $b^*_{k},b_\nu\ge2$ and the rest are constructed from this one in the same way as in the case ${\bf (c)}$ by changing last partial quotient $b_\nu$ to $b_\nu-1,1$ and so on. Note that for $\beta=1$ there's no partial quotient $\ge2$, but it has still two representations $\beta=[0;1]=[1].$}.
\end{theorem}
 
    \vskip+0.3cm

  \begin{remark}\label{remark:p0} In the case $ p=p_0$  we have two non-congruent critical lattices for $ \mathcal{B}_p$ (see case {\bf 5.3} from Section \ref{sec:lpnorms:param}). Therefore,   $p_0$-Dirichlet non-improvability of $\alpha$ is equivalent to the existence in its continued fraction expansions 
    either patterns (\ref{trip0})  from the case ({\bf a}) or patterns from the case ({\bf b}).  However, it is not known whether $\varsigma_{p_0}$ is rational or not. Hence, now we cannot choose between subcases     ({\bf b1}),  ({\bf b2}).
 \end{remark}

         \vskip+0.3cm

  \begin{remark} In the case of an arbitrary strongly symmetric norm, the criterion for $F$-Dirichlet non-improvability can also be given in terms of patterns from the continued fraction expansion of $\alpha$.
    But in this general case critical patterns analogous to  \eqref{pat}, \eqref{asym}, \eqref{asym1}
    should be defined by subsets of the boundary of the singularization domain $S$ (see \cite{AD} and \cite{I}) which correspond to critical lattices and where the maximum of the corresponding determinant occurs. The related calculations in the general case are too complicated.
\end{remark}
 \vskip+0.3cm

   \begin{remark}
   An application of the celebrated Birkhoff's Ergodic Theorem yields that almost all reals (in the sense of Lebesgue measure) are continued fraction normal. In particular, this means that for almost all real numbers, all patterns will occur infinitely many times. Thus, Theorem \ref{thm:charachterization} together with the normality argument implies the Lebesgue measure result of Theorem C (Theorem 1 from \cite{AD}) for the case of $L_p$-norms.

   \end{remark}

From the characterization result of Theorem \ref{thm:charachterization} we can deduce several important corollaries. First, let us describe some properties of the set of values of $p$ for which there exist badly approximable numbers $\alpha$ which are simultaneously $p$-Dirichlet non-improvable.
 
The next result is a trivial consequence of Theorem \ref{thm:charachterization} ({\bf a}) and ({\bf b}), so we do not provide a proof.

   \begin{corollary}\label{coroll:bad}
           {\rm  ({\bf a})}
For $p \in  (2,p_0)$ the set of $\mathbf{DI}_p^c$ contains no badly approximable numbers. 
      \vskip+0.3cm
                 {\rm  ({\bf b})}
For $p\in (1,2)\cup (p_0,\infty)$ the set $\mathbf{DI}_p^c\cap\mathbf{BA}\neq\emptyset$ if and only if the number $\varsigma_p$ is badly approximable.

\end{corollary} 

          One more corollary of the characterization is the following result, which clarifies the size of the set $\mathbf{DI}_p^c\cap\mathbf{BA}$ in the case of $L_p$-norm 
   with $ p\in  [1,2]\cup (p_0,\infty)$.
   Also note that the result for $ p=1$ is not covered by Theorem D, because the unit ball $\mathcal{B}_1$ is just a parallelogram.

        \vskip+0.3cm
 
     \begin{theorem}\label{thm:bad:dirichlet:fulldim} Suppose that the value of $p$ satisfies either
   $ p=1, 2 $ or $ p\in  (1,2)\cup (p_0,\infty)$ with $ \varsigma_p \in \mathbf{BA}$.
    Then the set $\mathbf{DI}_p^c\cap\mathbf{BA}$ has full Hausdorff dimension.
    \end{theorem} 
    We prove this theorem in Section \ref{sec:proof:bad:dirichlet:fulldim}.


In the direction of {\bf Problem III}, in addition to Theorem \ref{thm:di2minusd1:di1minusd2}, we provide a simple corollary of Theorem \ref{thm:charachterization}, which can be seen just by looking at the restricted patterns. Thus, we do not provide a formal proof. 
  \begin{theorem}\label{123}
  For any $2<p_1,p_2<p_0$, we have
  
{\rm ({\bf a})} $\mathbf{DI}_{p_1}=\mathbf{DI}_{p_2}, $

{\rm ({\bf b})} $\mathbf{DI}_2 \subset \mathbf{DI}_{p_2}, $

{\rm ({\bf c})}  $\mathbf{DI}_1 \subset \mathbf{DI}_{p_2}.$

This, in particular, means that the difference $\mathbf{DI}_{p_1} \setminus \mathbf{DI}_{p_2}$ is empty for any $p_1\in\{1\}\cup[2,p_0)$ and any $p_2\in(2,p_0)$.
  \end{theorem}


In another direction, we formulate a statement about the set of values of $p$ for which $p$-Dirichlet  non-improvable badly approximable numbers exist.

\vskip+0.3cm

    \begin{corollary}\label{coroll:setofparameters}
        
   The set
$$
\frak{P} = \{ p\in [1,\infty):  \,\,\text{there exist  $p$-Dirichlet  non-improvable badly approximable numbers}\,\,
\alpha\}
$$
has zero Lebesgue measure, is dense in $ (1,2)\cup (p_0,\infty)$, is absolutely winning in any interval $[a,b]\subset (1,2)\cup (p_0,\infty)$, and hence has full Hausdorff dimension.
\end{corollary}

Note that from Corollary \ref{coroll:bad} and Theorem \ref{thm:bad:dirichlet:fulldim} for the set $\frak{P}$ from Corollary \ref{coroll:setofparameters} we have 
$$
\frak{P}=\left\{ p\in [1,\infty):  \,\,\hdim \mathbf{DI}_p^c\cap\mathbf{BA} = 1 \right\}.
$$

\vskip+0.3cm
Definition of {\it absolutely winning sets}, their properties, and the related discussion can be found in  \cite{mcm}. 

Corollary \ref{coroll:setofparameters} is proven in Section \ref{sec:proof:setofparameters}.

Lastly, we note that in all previous statements we applied Theorem \ref{thm:charachterization} to get metrical results. However, as it can be seen from the formulation of Theorem \ref{thm:charachterization}, it is even easier to deduce the Dirichlet (non-)improvability property for a particular number as long as one knows its continued fraction expansion. As an example, we provide the following interesting result.

\begin{corollary}\label{coroll:number:E}
For $p\in[1,\infty]$, the number $e=\sum_{n=0}^\infty \frac{1}{n!}$ satisfies $e\in \mathbf{DI}_p$ if and only if $p\in(1,2)\cup(p_0,\infty)$.
\end{corollary}

\begin{proof}
It is well-known that $e$ has a continued fraction
\begin{equation}\label{cf:E}
e=[2;1,2,1,1,4,\ldots,2k,1,1,2k+2,1,1,\ldots].
\end{equation}

One can see that it has infinitely many patterns $\,x,1,1,y\,$ with $\min(x,y)\to\infty$. By Theorem \ref{thm:charachterization} and Remark \ref{remark:p0}, the existence of those patterns guarantees that $e\notin \mathbf{DI}_p$ for $2<p<p_0$, $p=1$, $p=2$ and $p=p_0$. Also, $e\notin\mathbf{DI}_\infty$, because it is known that $\mathbf{DI}_\infty = \mathbf{BA}\cup\Q$.

On the other hand, if $p\in(1,2)\cup(p_0,\infty)$, then by Theorem \ref{thm:charachterization}, for the number $e$ to be $p$-Dirichlet non-improvable, it needs to have symmetrical patterns from ({\bf b1}) or ({\bf b2}), depending on the irrationality of $\varsigma_p$. As large partial quotients in \eqref{cf:E} form a strictly growing sequence, we cannot have symmetrical (in the case $\varsigma_p\notin\Q$) or almost symmetrical (in the case $\varsigma_p\in\Q$) patterns in the continued fraction of $e$, so we have $e\in\mathbf{DI}_p$.
    
\end{proof}

 \section{  Application of compactness argument}\label{sec:application:compactness}\label{compaa}

\vskip+0.3cm

  In this section, we recall well-known statements about critical lattices.
   
  All the definitions  concerning convergence in the space of lattices, compactness
  and the related properties 
  are discussed 
  in Chapter V of  the  book
  \cite{CaB} or in a modern survey by Kleinbock, Shash and Starkov \cite{MR1928528} from a dynamical point of view
  (see Subsection 1.3, example {\bf 3d} from \cite{MR1928528}). 
  We identify the space of unimodular  lattices with $\frak{X} ={\rm SL}_2 (\mathbb{R})/{\rm SL}_2 (\mathbb{Z})$.
   
    Here, we are interested in normalised critical lattices of the form
\begin{equation}\label{loci}
  \overline{\Lambda}  = \frac{1}{\sqrt{\Delta_F}} \cdot \Lambda ,\,\,\,\,\, \Lambda \in \frak{L}_F,\,\,\,\,
  {\rm det}\, \overline{\Lambda} = 1.
  \end{equation}
  We define the set of all the lattices of the form (\ref{loci}) as
  {\it normalised critical locus} and denote it simply by  $ \overline{\frak{L}}_F$.
  As for discussion concerning critical loci, one can see a recent paper \cite{KKR}.
  In particular, $ \overline{\frak{L}}_F$ is a compact set in the space of lattices $\frak{X}$.

    \vskip+0.3cm
    
    Recall  that  we always have the inequality
      $$
    d_F (\alpha) \le 
     \frac{1}{\sqrt{\Delta_F}}.
     $$

   For our purpose, we need the following statement, which is an immediate corollary  of 
   Mahler's theory (see Chapter V from   \cite{CaB}) and compactness of  $ \overline{\frak{L}}_F$.
  
    \vskip+0.3cm
    
  \begin{proposition}\label{proposition:1}
 Real number  $\alpha$   satisfies the equality 
         $
d_F(\alpha) =     \frac{1}{\sqrt{\Delta_F}}
     $
     (is $F$-Dirichlet non-improvable)
      if and only if  there exists a lattice $\Lambda \in \frak{L}_F$ 
   and a subsequence $\tau _k \to \infty, \, k \to \infty$ such that the sequence of unimodular  lattices
    $\Lambda_\alpha (\tau_{k})$
  converges to  $\overline{\Lambda}$ in the space of lattices.
   \end{proposition}
    
      \vskip+0.3cm

\begin{proof}
    
First, suppose that there is a sequence $\{\tau_k\}_{k\in \N}$ such that 
$\Lambda_\alpha(\tau_k) \to \overline{\Lambda} \in \overline{\frak{L}}_F$.
Then, by continuity of the first successive minimum, we have
$\lambda_1(\Lambda_\alpha(\tau_k), \mathcal{B}_F)\to \lambda_1(\overline{\Lambda}, \mathcal{B}_F)=\frac{1}{\sqrt{\Delta_F}}$.  This means that $d_F (\alpha) = \frac{1}{\sqrt{\Delta_F}}$ and so $\alpha$ is $F$-Dirichlet non-improvable.

On the other hand,
if $\alpha$ is $F$-Dirichlet non-improvable, then there exists a sequence $\{\tau_k\}_{k\in \mathbb{N}}$ with
$\tau_k \to \infty$, such that 
$$ \lambda_1  (\Lambda_\alpha(\tau_k), \mathcal{B}_F)\to  \frac{1}{\sqrt{\Delta_F}}$$
and
$ \lambda_1  (\Lambda_\alpha(\tau_k), \mathcal{B}_F)$ increases monotonically. By the Mahler criterion 
there exists a compact $ \frak{X}_1 \subset \frak{X}$ such that $\Lambda_\alpha(\tau_k)\in \frak{X}_1$ for every $ k\in \mathbb{N}$.
So we can choose a subsequence $\tau_{k_l}, l \to \infty$ such that 
$ \lambda_1  (\Lambda_\alpha(\tau_{k_l}), \mathcal{B}_F)\to \frac{1}{\sqrt{\Delta_F}}$ 
and $ \Lambda_\alpha(\tau_{k_l})$ converges to a certain $\overline{\Lambda}_\alpha \in \frak{X}_1$.
Function $ \overline{\Lambda} \mapsto  \lambda_1  (\overline{\Lambda}, \mathcal{B}_F)$ is continuous in $\frak{X}$.
So $\lambda_1  (\overline{\Lambda}_\alpha, \mathcal{B}_F) = \frac{1}{\sqrt{\Delta_F}}$. But this means that 
$ \Lambda_\alpha(\tau_{k_l})\to  \overline{\Lambda}_\alpha \in \overline{\frak{L}}_F$ and the proposition is proven.

\end{proof}
       
      \vskip+0.3cm 
      
     Proposition \ref{proposition:1} gives an ineffective criterion of $F$-Dirichlet improvability for general $F$.
      The main purpose of the present paper is to prove effective versions of this criterion for the case of $L_p$-norms in terms of regular continued fractions. We consider this special case in detail in the next section.

         \vskip+0.3cm
    \section{  $L_p$-norms and parametrization of critical lattices}\label{sec:lpnorms:param}
  \vskip+0.3cm
  
  In this section, we describe the well-known structure of the sets of critical lattices for
    $\mathcal{B}_p$
  for different values of $p\in [1,\infty)$.
  First, we consider the cases when the locus $\frak{L}_{p}$ consists of isolated points, and then we will deal with the cases $ p =1,2$ when the locus consists of continuous families of lattices.
  The notation in this section is not very standard, but nevertheless, it is very convenient for our purposes.

 In all the cases we should take into account the following characteristics.
 Let $\Lambda $ be a critical lattice  for
    $\mathcal{B}_p$. 
    Such a lattice has two  special points
    $\pmb{z}' = (a,b),\, \pmb{z}'' = (c,d) \in \Lambda$ on the boundary of $\mathcal{B}_p$
which are important for us. Besides these two points,
points      $-\pmb{z}',- \pmb{z}'' $ belong to the boundary of $\mathcal{B}_p$
as well as some of combinations       $\pm \pmb{z}'\pm \pmb{z}'' $ and even the combination 
$\pm ( \pmb{z}''+2\pmb{z}')$ in the special subcase of the case $p=1$.
In every case, we will explain which of these combinations belong to the boundary.
  We may assume that $ a,b,c,d\ge 0$, by symmetry. Then 
  $$ \Lambda = \Omega \cdot \mathbb{Z}^2 , \,\,\, 
     \Omega =
    \left(
    \begin{array}{cc}
  a& c 
  \cr
  b& d
  \end{array} \right)
  .$$
For each critical lattice $\Lambda$ from the considerations below, we will write out the matrix $\Omega$ and calculate the values
  \begin{equation}\label{alef}
    \alpha(\Omega) = -\frac{b}{d},\,\,\,\,\,
  \alpha^* (\Omega ) =  \frac{a}{c}
  \end{equation}
  (here if $c $ or $ d$ are equal to zero, we assume that the corresponding values  of
     $ \alpha(\Omega),
  \alpha^* (\Omega ) $ are equal to infinity).
  We should note that 
  if we consider a normalized critical lattice
  $
   \overline{\Lambda} = \overline{\Omega} \cdot \mathbb{Z}^2 
   $ defined in (\ref{loci}),
these values remain the same:     $ \alpha(\Omega) = 
  \alpha (\overline{\Omega} ) ,  \alpha^*(\Omega) = 
  \alpha^* (\overline{\Omega} ) .$

        \vskip+0.3cm
        While describing critical lattices, we consider several cases.
        
 The results for $ p\neq 1,2, \infty$ are related to a conjecture by Minkowski, which was studied by many authors (see references from \cite{AD} and \cite{maly}). The final solution was claimed in \cite{maly}, and it highly relies on computer calculations. However, in the original work \cite{maly}, and several prior manuscripts on this conjecture, no code was provided, making verification of the claim almost impossible. In Appendix \ref{appendinx:1} we provide an independent verification of the proof of Minkowski's conjecture. Appendix \ref{appendinx:2} contains three pseudocode for some of the programs we used for a better understanding of the strategy.
 
Davis \cite{D} showed that there exists $ p_0 \in (2.57, 2.58) $ such that the critical lattices in intervals
 $(2,p_0)$ and  $(1,2)\cup (p_0,\infty)$ are essentially different.  We describe the critical lattices in these cases in Subsections 5.1, 5.2 below.
 
We also note that the explicit formula for the value $p_0$ is yet unknown, as the proof started in \cite{D} and finalized in \cite{w1,w2} is ineffective. It is a good open question to provide a closed-form formula (if possible) for $p_0$.

 The results for the cases $ p= 1,2 $ are well-known. Nevertheless, we should mention that we formulate all these results below in a new form, which is convenient for our purposes, and thus it may differ from the very standard one.
   Cases $ p=1,2$ are considered in 
  Subsections 5.4, 5.5.

         \vskip+0.3cm
   {\bf 5.1. Case $ 2<p<p_0$.}
  In this case the only two (congruent) critical lattices for the ball $\mathcal{B}_p$ are
$$     \Lambda_1 =  \Omega_1 \cdot \mathbb{Z}^2\,\,\,\,
    \text{and}\,\,\,\,
       \Lambda_1' =   
        \Omega_1' \cdot \mathbb{Z}^2,\,\,\,\,
    $$
    {where}
     \begin{equation}\label{lat01}
    \Omega_1 =
      \left(
  \begin{array}{cc}
   \left(1- \frac{1}{2^{1/p}}\right)^{\frac{1}{p}}
& \left(1- \frac{1}{2^{1/p}}\right)^{\frac{1}{p}}
  \cr
  - \frac{1}{2} 
& \frac{1}{2} 
  \end{array}
  \right)    
   ,\,\,\,
   \Omega_1' =     
     \left(
    \begin{array}{cc}
   \frac{1}{2} 
& \frac{1}{2} 
  \cr
  -\left(1- \frac{1}{2^{1/p}}\right)^{\frac{1}{p}}
& \left(1- \frac{1}{2^{1/p}}\right)^{\frac{1}{p}}
  \end{array} \right) .
  \end{equation} 
  
  It is clear that 
    \begin{equation}\label{lat11}
   \alpha(\Omega_1) =
     \alpha(\Omega_1') =
        \alpha^*(\Omega_1) =
     \alpha^*(\Omega_1') =
1.
    \end{equation}
  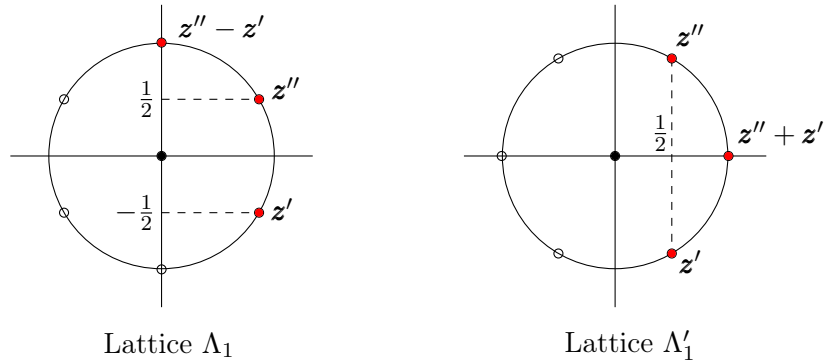
\begin{figure}[h] 
  \centering
  \begin{tikzpicture}

    \node[draw=black, circle,inner sep=1.2pt]  at (-3,2) {};
     \node[fill=black,circle,inner sep=1.1pt]   at (-3,2) {};

    \node[draw=black, circle,inner sep=1.2pt]  at (3,2) {};
     \node[fill=black,circle,inner sep=1.1pt]   at (3,2) {};

           \node[draw=black, circle,inner sep=30pt] at (-3,2) {};
             \node[draw=black, circle,inner sep=30pt] at (3,2) {};
           
             \draw (-5,2) -- (-1,2);
               \draw (1,2) -- (5,2);
                 \draw (3,0) -- (3,4);
                      \draw (-3,0) -- (-3,4);

        \draw[dashed] (3.75,3.29) -- (3.75,0.71);

        \draw[dashed] (-3,2.75) --  (-1.71,2.75);
        
         \draw[dashed] (-3,1.25) --  (-1.71,1.25);
         
          \draw (-3.2,2.75)  node {$\frac{1}{2}$}; 
               \draw (-3.35,1.25)  node {$-\frac{1}{2}$}; 
    
           \node[draw=black, circle,inner sep=1.2pt]  at (-3,3.5) {};
     \node[fill=red,circle,inner sep=1.1pt]   at (-3,3.5) {};
     
          \node[draw=black, circle,inner sep=1.2pt]  at (4.5,2) {};
     \node[fill=red,circle,inner sep=1.1pt]   at (4.5,2) {};
     
           \node[draw=black, circle,inner sep=1.2pt]  at (-3,0.5) {};
           
                     \node[draw=black, circle,inner sep=1.2pt]  at (1.5,2) {};

           \node[draw=black, circle,inner sep=1.2pt]  at (-1.71,2.75) {};
     \node[fill=red,circle,inner sep=1.1pt]   at (-1.71,2.75){};

            \node[draw=black, circle,inner sep=1.2pt]  at (-1.71,1.25) {};
     \node[fill=red,circle,inner sep=1.1pt]   at (-1.71,1.25){};

           \node[draw=black, circle,inner sep=1.2pt]  at (-4.29,2.75) {};
           
            \node[draw=black, circle,inner sep=1.2pt]  at (-4.29,1.25) {};

           \node[draw=black, circle,inner sep=1.2pt]  at (3.75,3.29) {};
     \node[fill=red,circle,inner sep=1.1pt]   at (3.75,3.29){};

            \node[draw=black, circle,inner sep=1.2pt]  at (3.75,0.71) {};
     \node[fill=red,circle,inner sep=1.1pt]   at (3.75,0.71){};

            \node[draw=black, circle,inner sep=1.2pt]  at (2.25,0.71) {};

            \node[draw=black, circle,inner sep=1.2pt]  at (2.25,3.29) {};

  \draw (-1.4,1.3)  node {$\pmb{z}'$}; 
  
  \draw (-1.35,2.9)  node {$\pmb{z}''$};

  \draw (-2.2,3.7)  node {$\pmb{z}''-\pmb{z}'$};

    \draw (4,0.6)  node {$\pmb{z}'$}; 
  
  \draw (4,3.6)  node {$\pmb{z}''$};

    \draw (3.6,2.3)  node {$\frac{1}{2}$}; 

  \draw (5.2,2.3)  node {$\pmb{z}''+\pmb{z}'$};

  \draw (-2.9,-0.5)  node {Lattice $\Lambda_1$}; 
    \draw (3.2,-0.5)  node {Lattice $\Lambda_1'$};

  \end{tikzpicture}
  
       \caption{ Critical lattices in the case  $ 2<p<p_0$.
       }
       \label{fff1}  
\end{figure}

For $\Lambda_1$ points $\pm ( \pmb{z}''-\pmb{z}')$ belong to the boundary of  $\mathcal{B}_p$. For $\Lambda_1'$ points $\pm ( \pmb{z}''+\pmb{z}')$ belong to the boundary of  $\mathcal{B}_p$.
Points $ \pmb{z}'$ and $\pmb{z}''$  for lattices $\Lambda_1,\Lambda_1'$  and all other lattice point on the boundary of $\mathcal{B}_p$ are visualised on Fig. \ref{fff1}. Here and in the sequel, the circles symbolize the unit disks      $\mathcal{B}_p$.

           \vskip+0.3cm
   {\bf 5.2. Case $ 1<p<2$ and $p> p_0$.}
  Here we should define $\varsigma_p \in \left(0,1 \right)$ as the unique root of the equation
   \begin{equation}\label{lat3}
\varsigma^p+(1+\varsigma)^p = 2
.\end{equation}
Then the only two (congruent) critical lattices  in this case are
  \begin{equation}\label{lat1}
     \Lambda_2^{\pm} =  \Omega_2^{\pm} \cdot \mathbb{Z}^2\,\,\,\,
     \text{where}
    \,\,\,\,
    \Omega_2^{\pm} =
    \left(
    \begin{array}{cc}
  \frac{\varsigma_p}{2^{{1}/{p}}}  &\frac{1}{2^{{1}/{p}}} 
  \cr
  \mp \frac{1+\varsigma_p}{2^{{1}/{p}}}  &\pm\frac{1}{2^{{1}/{p}}} 
\end{array} \right)       .
  \end{equation}
    We see  that 
    \begin{equation}\label{lat12}
   \alpha(\Omega_2^{\pm}) = 1+
   \varsigma_p >1>
        \alpha^*(\Omega_2^{\pm}) =  \varsigma_p
        .
    \end{equation}
    
      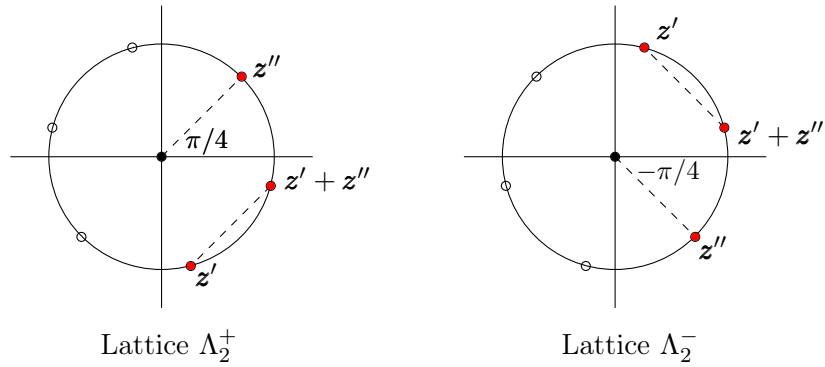
\begin{figure}[h]
  \centering
  \begin{tikzpicture}

    \node[draw=black, circle,inner sep=1.2pt]  at (-3,2) {};
     \node[fill=black,circle,inner sep=1.1pt]   at (-3,2) {};

    \node[draw=black, circle,inner sep=1.2pt]  at (3,2) {};
     \node[fill=black,circle,inner sep=1.1pt]   at (3,2) {};

           \node[draw=black, circle,inner sep=30pt] at (-3,2) {};
             \node[draw=black, circle,inner sep=30pt] at (3,2) {};
           
             \draw (-5,2) -- (-1,2);
               \draw (1,2) -- (5,2);
                 \draw (3,0) -- (3,4);
                      \draw (-3,0) -- (-3,4);

        \draw[dashed]  (-1.555,1.615)  -- (-2.613,0.555) ;

        \draw[dashed] (-3,2) --  (-1.94,3.06) ;

         \draw[dashed]  (4.445,2.385)-- (3.387,3.445)  ;

        \draw[dashed] (3,2) --  (4.06,0.94) ;

           \node[draw=black, circle,inner sep=1.2pt]  at (-1.94,3.06) {};
     \node[fill=red,circle,inner sep=1.1pt]   at (-1.94,3.06)  {};

           \node[draw=black, circle,inner sep=1.2pt]  at (-4.06,0.94) {};
     
          \node[draw=black, circle,inner sep=1.2pt]  at (-3.387,3.445) {};
          
              \node[draw=black, circle,inner sep=1.2pt]  at (-4.445,2.385) {};

            \node[draw=black, circle,inner sep=1.2pt]  at (-2.613,0.555) {};
     \node[fill=red,circle,inner sep=1.1pt]   at (-2.613,0.555) {};
     
           \node[draw=black, circle,inner sep=1.2pt]  at (-1.555,1.615) {};
     \node[fill=red,circle,inner sep=1.1pt]   at (-1.555,1.615) {};

       \node[draw=black, circle,inner sep=1.2pt]  at (4.06,0.94) {};
     \node[fill=red,circle,inner sep=1.1pt]   at (4.06,0.94)  {};

           \node[draw=black, circle,inner sep=1.2pt]  at (1.96,3.06) {};
     
          \node[draw=black, circle,inner sep=1.2pt]  at (2.613,0.555) {};
          
              \node[draw=black, circle,inner sep=1.2pt]  at (1.555,1.615) {};

            \node[draw=black, circle,inner sep=1.2pt]  at (3.387,3.445) {};
     \node[fill=red,circle,inner sep=1.1pt]   at (3.387,3.445) {};
     
           \node[draw=black, circle,inner sep=1.2pt]  at (4.445,2.385) {};
     \node[fill=red,circle,inner sep=1.1pt]   at (4.445,2.385) {};

  \draw(-2.4,0.44)  node {$\pmb{z}'$}; 
  
  \draw (-1.6,3.2)  node {$\pmb{z}''$}; 
    \draw (-2.4,2.2)  node {\begin{small}$\pi/4$\end{small}};

  \draw (-0.8,1.7)  node {$\pmb{z}'+\pmb{z}''$};

   \draw(4.3,0.8)  node {$\pmb{z}''$}; 
  
  \draw (3.7,3.7)  node {$\pmb{z}'$}; 
    \draw (-2.4,2.2)  node {\begin{small}$\pi/4$\end{small}}; 
  
    \draw (3.7,1.8)  node {\begin{small}$-\pi/4$\end{small}};

  \draw (5.2,2.3)  node {$\pmb{z}'+\pmb{z}''$};

  \draw (-2.9,-0.5)  node {Lattice $\Lambda_2^+$}; 
    \draw (3.2,-0.5)  node {Lattice $\Lambda_2^-$};

  \end{tikzpicture}
  
   \caption{ Critical lattices in  the case $ 1<p<2$ and $p> p_0$.} \label{fff2}
\end{figure}

For both lattices $\Lambda_2^{\pm}$ points $\pm ( \pmb{z}''+\pmb{z}')$ belong to the boundary of  $\mathcal{B}_p$.
 Points $ \pmb{z}'$ and $\pmb{z}''$  for lattices $\Lambda_2^{\pm}$ are shown on Fig. \ref{fff2}.

           \vskip+0.3cm
   {\bf 5.3. Case $ p=p_0$.}
In this case, we have four critical lattices $ \Lambda_1, \Lambda_1', \Lambda_2^{\pm}$.

    \vskip+0.3cm

               \vskip+0.3cm
   {\bf 5.4. Case $ p=1$.}
In this case, we have the following two one-parameter families of critical lattices,
both depending on a real parameter $ a\in \left[0,\frac{1}{2}\right)$:
\begin{equation}\label{lat119}
\Lambda_{3}^{\pm}(a)  =  \Omega_3^{\pm} (a) \cdot \mathbb{Z}^2,\,\,
 \,\,
     \text{where}
     \,\,\,\,
    \Omega_3^{\pm}(a) =
    \left(
    \begin{array}{cc}
  a  &\frac{1}{2} 
  \cr
\pm(a-1) &\pm \frac{1}{2} 
\end{array} \right)
    .
  \end{equation}
  In this case, one has
      \begin{equation}\label{lat192}
   \alpha(\Omega_3^{\pm}(a)) = {2-2a}
   \ge 1 \ge 
   \alpha^*(\Omega_3^{\pm}(a)) 
   =
2{a} 
    \end{equation}
and 
      \begin{equation}\label{lat1920}
   \alpha(\Omega_3^{\pm}(a)) +
   \alpha^*(\Omega_3^{\pm}(a)) 
   =
2.
    \end{equation}
    It is important for us that for a special value of parameter  $a=0$, critical lattices 
    $\Lambda_{3}^{+}(0)
    =\Lambda_{3}^{-}(0)
    $
    coincide and have
    eight points on the boundary  of $\mathcal{B}_1$,
    while for $a\neq0$ we have six points. All the possible situations are visualized in Fig. \ref{fff3}.

  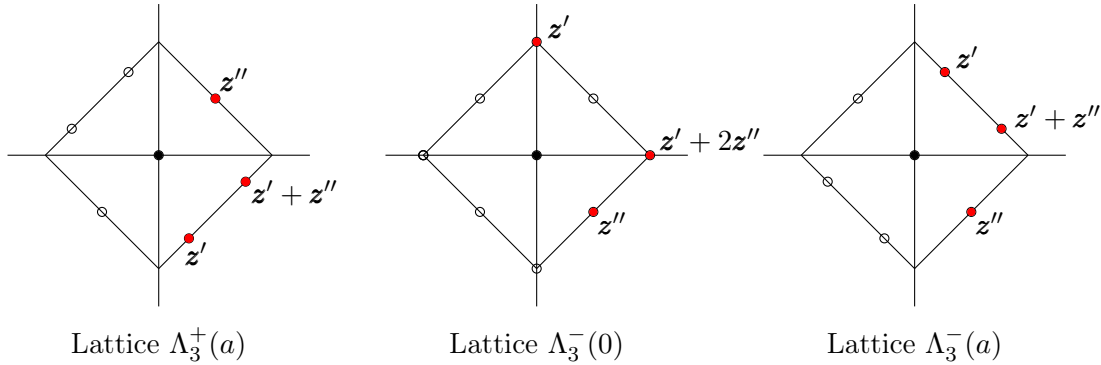
\begin{figure}[h]
  \centering
  \begin{tikzpicture}

             \node[draw=black, circle,inner sep=1.2pt]  at (0,2) {};
     \node[fill=black,circle,inner sep=1.1pt]   at (0,2) {};

             \draw (-2,2) -- (2,2);
               \draw (0,0) -- (0,4);
                 \draw (0,3.5) -- (1.5,2);
                   \draw (0,3.5) -- (-1.5,2);
                   \draw (0,0.5) -- (1.5,2);
                   \draw (0,0.5) -- (-1.5,2);

             \node[draw=black, circle,inner sep=1.2pt]  at (0,3.5) {};
     \node[fill=red,circle,inner sep=1.1pt]   at (0,3.5) {};
     \node[draw=black, circle,inner sep=1.2pt]  at (1.5,2) {};
     \node[fill=red,circle,inner sep=1.1pt]   at (1.5,2) {};
         \node[draw=black, circle,inner sep=1.2pt]  at (0.75,2.75) {};
     
             \node[draw=black, circle,inner sep=1.2pt]  at (0.75,1.25) {};
     \node[fill=red,circle,inner sep=1.1pt]   at (0.75,1.25) {};
     
     \node[draw=black, circle,inner sep=1.2pt]  at (0,0.5) {};
     
         \node[draw=black, circle,inner sep=1.2pt]  at (-0.75,1.25) {};
          \node[draw=black, circle,inner sep=1.2pt]  at (-0.75,2.75) {};
         
               \node[draw=black, circle,inner sep=1.2pt]  at (-1.5,2) {};
 
   \draw (0,-0.5)  node {Lattice $\Lambda_3^{-}(0)$};

             \node[draw=black, circle,inner sep=1.2pt]  at (5,2) {};
     \node[fill=black,circle,inner sep=1.1pt]   at (5,2) {};

             \draw (3,2) -- (7,2);
               \draw (5,0) -- (5,4);
                 \draw (5,3.5) -- (6.5,2);
                   \draw (5,3.5) -- (3.5,2);
                   \draw (5,0.5) -- (6.5,2);
                   \draw (5,0.5) -- (3.5,2);

             \node[draw=black, circle,inner sep=1.2pt]  at (5.4,3.1) {};
     \node[fill=red,circle,inner sep=1.1pt]   at (5.4,3.1) {};
  
         \node[draw=black, circle,inner sep=1.2pt]  at (6.15,2.35) {};
          \node[fill=red,circle,inner sep=1.1pt]   at (6.15,2.35){};
     
             \node[draw=black, circle,inner sep=1.2pt]  at (5.75,1.25) {};
     \node[fill=red,circle,inner sep=1.1pt]   at (5.75,1.25) {};
     
     \node[draw=black, circle,inner sep=1.2pt]  at (4.6,0.9) {};
     
         \node[draw=black, circle,inner sep=1.2pt]  at (3.85,1.65) {};
          \node[draw=black, circle,inner sep=1.2pt]  at (4.25,2.75) {};
         
               \node[draw=black, circle,inner sep=1.2pt]  at (-1.5,2) {};

      \draw (5,-0.5)  node {Lattice $\Lambda_3^{-}(a)$};

     \draw(5.7,3.3)  node {$\pmb{z}'$}; 
     
       \draw(6,1.1)  node {$\pmb{z}''$};

             \node[draw=black, circle,inner sep=1.2pt]  at (-5,2) {};
     \node[fill=black,circle,inner sep=1.1pt]   at (-5,2) {};

             \draw (-3,2) -- (-7,2);
               \draw (-5,0) -- (-5,4);
                 \draw (-5,3.5) -- (-6.5,2);
                   \draw (-5,3.5) -- (-3.5,2);
                   \draw (-5,0.5) -- (-6.5,2);
                   \draw (-5,0.5) -- (-3.5,2);

             \node[draw=black, circle,inner sep=1.2pt]  at (-5.4,3.1) {};

         \node[draw=black, circle,inner sep=1.2pt]  at (-6.15,2.35) {};

             \node[draw=black, circle,inner sep=1.2pt]  at (-5.75,1.25) {};

     \node[draw=black, circle,inner sep=1.2pt]  at (-4.6,0.9) {};
     \node[fill=red, circle,inner sep=1.2pt]  at (-4.6,0.9) {};
     
         \node[draw=black, circle,inner sep=1.2pt]  at (-3.85,1.65) {};
           \node[fill=red, circle,inner sep=1.2pt]  at (-3.85,1.65) {};
         
          \node[draw=black, circle,inner sep=1.2pt]  at (-4.25,2.75) {};
              \node[fill=red, circle,inner sep=1.2pt]  at (-4.25,2.75) {};
         
               \node[draw=black, circle,inner sep=1.2pt]  at (-1.5,2) {};

      \draw (-5,-0.5)  node {Lattice $\Lambda_3^{+}(a)$};

     \draw(-4,3)  node {$\pmb{z}''$}; 
     
       \draw(-4.5,0.7)  node {$\pmb{z}'$};

    \draw(-3.2,1.5)  node {$\pmb{z}'+ \pmb{z}''$};

    \draw(6.9,2.5)  node {$\pmb{z}'+ \pmb{z}''$}; 
    
        \draw(2.3,2.2)  node {$\pmb{z}'+2 \pmb{z}''$};

          \draw(0.3,3.7)  node {$\pmb{z}'$};

     \draw(1,1.1)  node {$\pmb{z}''$}; 
   
  \end{tikzpicture}
  
   \caption{Critical lattices for $p=1$ } \label{fff3}
\end{figure}
    
    For  both lattices $\Lambda_3^{\pm}(a)$ with $a\neq 0$,  points $\pm ( \pmb{z}''+\pmb{z}')$ belong to the boundary of  $\mathcal{B}_p$.
    For the lattices $\Lambda_3^{\pm}(0)$  not only points 
    $\pm ( \pmb{z}''+\pmb{z}')$ but also points $\pm ( \pmb{z}''+2\pmb{z}') = (\pm 1, 0)$
    belong to the boundary.  The situation is visualized in Fig. \ref{fff3}.

               \vskip+0.3cm
   {\bf 5.5. Case $ p=2$.}
In this case, we describe the critical locus $\frak{L}_{2}$ by the following parametrization of its elements. For 
$\varphi \in \left[0, \frac{\pi}{6}\right] $
and 
$
u = \sin\varphi  \in \left[0,\frac{1}{2}\right]$, consider the lattices
\begin{equation}\label{lat20}
     \Lambda_{4}^{\pm}(\varphi ) =  \Omega_4^{\pm} (\varphi )\! \cdot\! \mathbb{Z}^2, \,\,
    \Omega_4^{\pm}(\varphi ) =
    \left( 
    \begin{array}{cc}
  \sin\varphi&\cos \left(\frac{\pi}{6}+ \varphi\right)
  \cr
\pm \cos \varphi &\mp \sin \left(\frac{\pi}{6}+ \varphi\right) 
\end{array} \right)
=
 \left( 
    \begin{array}{cc}
  u&\frac{\sqrt{3-3u^2}-u}{2}
  \cr
\pm\sqrt{1-u^2}&
\mp\frac{\sqrt{1-u^2}+u\sqrt{3}}{2}
 \end{array} \right)
.
\end{equation}
Now put
$$
\omega
= \frac{u\sqrt{3}}{\sqrt{1-u^2}} = \sqrt{3}\,\tan \varphi \in [0,1].
$$ 
Then 
\begin{equation}\label{lat199}
2\ge \alpha (\Omega_4^{\pm}(\varphi )) =\frac{2}{1+\omega}\ge 1\ge 
\alpha^* (\Omega_4^{\pm}(\varphi )) =\frac{2\omega}{3-\omega}.
\end{equation}
So if we write $$\beta(\Omega_4^{\pm}(\varphi ))=  \frac{1}{\alpha (\Omega_4^{\pm}(\varphi )) -1}-1 , \quad\quad\beta^*(\Omega_4^{\pm}(\varphi ))=\frac{1}{\alpha^* (\Omega_4^{\pm}(\varphi ))}-1,$$
then we get the identity
\begin{equation}\label{lat200}
\beta(\Omega_4^{\pm}(\varphi ))\cdot\beta^*(\Omega_4^{\pm}(\varphi )) = 3.
\end{equation}

 We do not put here a special figure devoted to the critical lattices in the case $p=2$. 
 We would like to mention that for $\varphi =0$ the situation is similar to the one from the left picture in Fig. \ref{fff1}, and that for 
$\varphi =\frac{\pi}{6}$, the situation is similar to the one from the right picture in Fig. \ref{fff1}
 (of course, the values of coordinates of the points will differ from  $\frac{1}{2}$), while for $0< \varphi < \frac{\pi}{6}$ we can refer to Fig. \ref{fff2} (the picture will have the same structure, although the angles will be different).

        \vskip+0.3cm

     \section{  Proof of Theorem \ref{thm:charachterization}}\label{ffiii}

         \vskip+0.3cm

          First, we show that in all statements ({\bf a}) - ({\bf d}) of Theorem \ref{thm:charachterization}, if $\alpha$ is $F$-Dirichlet non-improvable, then its continued fraction expansion contains corresponding patterns.
          For this purpose, we  use Proposition \ref{proposition:1} and the following property of lattices 
 $\Lambda_\alpha (t)$, which is an important tool for our consideration.
 
   \vskip+0.3cm
\begin{lemma}\label{lemma:1:points}
Suppose that two points $\frak{z}'= (x', y') , \frak{z}''=(x'', y'')\in \Lambda_\alpha (t),\, t>1$ satisfy the conditions
$$
0< x'< x'',\,\,\,\,\,\,\,\,
|y''|<|y'|
$$
and in addition, in the rectangle
$$
R(x'', y') =
\{ (x,y)\in \mathbb{R}^2:\,\,\, 0\le x\le x'', |y|\le |y'|
\}
$$
there are no points of  $\Lambda_\alpha (t)$  different from
$\pmb{0},  \frak{z}' , \frak{z}''$. Then there exists $\nu$ such that 
$$
\frak{z}'= \frak{z}_\nu,\,\,\,\, \frak{z}''= \frak{z}_{\nu+1}.
$$
\end{lemma}
 \vskip+0.3cm
 
\begin{proof}
     This lemma is an immediate corollary of Lagrange's theorem that convergents of a continued fraction to $\alpha$ correspond to best approximations to $\alpha$.
     \end{proof}
 
  \vskip+0.3cm

 We can assume that $\alpha$ is irrational, as for any norm $F$, we have $\Q\subset\mathbf{DI}_F$.
 
 Let us prove the necessity condition in the statement ({\bf a}). From Proposition \ref{proposition:1} we see that either there is a subsequence of lattices $ \Lambda_\alpha (t)$ which converges to $ \overline{\Lambda}_1$
  or a subsequence of lattices $ \Lambda_\alpha (t)$ which converges to $ \overline{\Lambda}_1'
  .$
  
  Suppose that the lattice $ \Lambda_\alpha (t)$ is close enough to $ \overline{\Lambda}_1$ or to $ \overline{\Lambda}_1'$.
  It means that we have fixed some large $R$ and taken a small $\varepsilon$ such that 
  close to any point $ \pmb{z}\in  \overline{\Lambda}_1$ or to $ \overline{\Lambda}_1'$ with $| \pmb{z}|\le R$  
  there exist just one point $\frak{z}\in \Lambda_\alpha (t)$ with 
  $| \pmb{z}- \frak{z}|<\varepsilon$. In particular, there are points 
$\frak{z}' =(x',y'), \frak{z}''=(x'',y'') \in \Lambda_\alpha (t)$  which are close to the points $\pmb{z}',\pmb{z}''$ from Section \ref{sec:lpnorms:param} which correspond to $\Lambda_1$ or $\Lambda_1'$.
  We may assume that 
\begin{equation}\label{dopp}
\left|
\frac{x'}{x''} - \alpha^*(\Omega_1)
\right| =
\left|
\frac{x'}{x''} - 1
\right| <\varepsilon_1,\,\,\,\,\,
\left|
\frac{|y'|}{|y''|} - \alpha(\Omega_1)
\right| =
\left|
\frac{|y'|}{|y''|} - 1
\right| <\varepsilon_1
\end{equation}
{or}
$$
\left|
\frac{x'}{x''} - \alpha^*(\Omega_1')
\right| =
\left|
\frac{x'}{x''} - 1
\right| <\varepsilon_1,\,\,\,\,\,
\left|
\frac{|y'|}{|y''|} - \alpha(\Omega_1')
\right| =
\left|
\frac{|y'|}{|y''|} - 1
\right| <\varepsilon_1
$$
for small positive $\varepsilon_1$.
However, we need more detailed information to deal with the statement ({\bf a}).
  
  Suppose that  lattice $ \Lambda_\alpha (t)$ is close enough  just to $ \overline{\Lambda}_1$.
Consider points $ \frak{z}' =(x',y'),\frak{z}'' =(x'',y'')\in \Lambda_\alpha (t)$  which are close to  points 
$\pmb{z}' =\left(x_0, \frac{1}{2}\right),\pmb{z}'' =\left(x_0, -\frac{1}{2}\right)\in \Lambda_1$ correspondingly.
 As $x'\neq x''$ (otherwise $|y'-y''|\ge t$ which is not possible for large $t$), either $ x'<x''$ or $ x''<x'$. Without loss of generality, we may assume that  $ x'<x''$. As $\alpha$ is irrational, we have $|y'|\neq|y''|$, so we need to consider two cases: either $ |y''|<|y'|$ ({\bf case 1$^0$}) or 
$ |y'|<|y''|$ ({\bf case 2$^0$}).
We should note that in the strip $\mathcal{S}=\{ (x,y): 0< x <x_0\}$ there are no points of  $\Lambda_1$.

In {\bf case 1$^0$} we have $ |y''|<|y'|$. As all the points of $\Lambda_\alpha (t)$ are located in the neighbourhoods of points of $\Lambda_1$  and the strip $\mathcal{S}$ is free of points of $\Lambda_1$, we conclude that the rectangle $R(x'', y')$ contains only three points of  lattice 
$
\Lambda_\alpha (t)$  - points $\pmb{0},  \frak{z}'$ and  $\frak{z}''$. By Lemma \ref{lemma:1:points} there 
exists $\nu$ such that 
\begin{equation}\label{nuod}
\frak{z}'= \frak{z}_\nu,\,\,\,\, \frak{z}''= \frak{z}_{\nu+1}.
\end{equation}
Now
\begin{equation}\label{fdof}
\frac{x'}{x''} = \frac{q_\nu}{q_{\nu+1}}  \in  (1-\varepsilon_1, 1)\,\,\,\,
\text{and}\,\,\,\,
\frac{y'}{y''} =\frac{\xi_\nu}{\xi_{\nu+1}} \in (1, 1+\varepsilon_1)
,
\end{equation}
where $\varepsilon_1>0$ is small when $\varepsilon$ is small.
Formulae (\ref{fromu}) show that $ a_{\nu+1} = a_{\nu+2} = 1$, while  partial quotients 
$a_\nu$  and $a_{\nu+3}$ are large and tend to infinity as $\varepsilon$ tends to zero.

In {\bf case 2$^0$} we have $ |y''|>|y'|$. Here, we should consider the point 
$ \frak{z}''-\frak{z}' = (x''-x', y''-y') \in  \Lambda_\alpha (t)$ which is close either to the point
$(0,1)\in \Lambda_1$ or to $(0,-1)\in \Lambda_1$. Note that $ x''-x'>0$ and 
$y'', y'$ have different signs. Now, from the property of the empty strip $\mathcal{S}$, we see that the
rectangle 
$
R(x', |y''|+|y'|)
$ contains only three points $\pmb{0}, \frak{z}''-\frak{z}' ,\frak{z}' $ of the lattice 
$\Lambda_\alpha (t)$ which lie on its boundary. Hence by Lemma \ref{lemma:1:points}
$$
\exists \, \nu :\,\,\,\,\,\frak{z}''
-  \frak{z}' =\frak{z}_\nu,\,\,\,\,\,\,\frak{z}' = \frak{z}_{\nu+1}.
$$
Now
$$
\frac{x''-x'}{x'} = \frac{q_\nu}{q_{\nu+1}}  \in  (0,\varepsilon_1)\,\,\,\,
\text{and}\,\,\,\,
\frac{|y''-y'|}{|y''|} =\frac{\xi_\nu}{\xi_{\nu+1}} \in (2, 2+\varepsilon_1)
$$
with small positive $\varepsilon_1$ and by formulae (\ref{fromu})  we see that 
$a_{\nu+2} = 2$ while both $a_{\nu+1}$ and $a_{\nu+3}$ are large.

  \vskip+0.3cm

  Suppose that  lattice $ \Lambda_\alpha (t)$ is close enough to $ \overline{\Lambda}_1'$.
  Again we consider 
  points $ \frak{z}' =(x',y'),\frak{z}'' =(x'',y'')\in \Lambda_\alpha (t)$  which are close to  points 
$\pmb{z}' =\left( \frac{1}{2}, y_0\right),\pmb{z}'' =\left(\frac{1}{2}, - y_0\right)\in \Lambda_1, y_0>0$ correspondingly. We should note that the only points of $ \overline{\Lambda}_1'$ which lie in rectangle
$\mathcal{R} = R(1,y_0)$ are points 
\begin{equation}\label{four}
\pmb{0}, \pmb{z}', \pmb{z}''\,\,\,\text{ and}\,\,\, \pmb{z}'+\pmb{z}''.
\end{equation}
By discreteness of $ \overline{\Lambda}_1'$ we see that for small $\varepsilon$ in the 
$\varepsilon$-neighbourhood of $\mathcal{R}$  the only points of 
   $ \overline{\Lambda}_1'$ in this neighbourhood will be the same four points (\ref{four}).
   Suppose that $x'<x''$ and consider two cases:
   $ |y''|<|y'|$ ({\bf case 1$^0$}) or 
$ |y'|<|y''|$ ({\bf case 2$^0$}).

In {\bf case} 1$^0$ we have that the rectangle $R(x'', y')$ contains only three points of  lattice 
$
\Lambda_\alpha (t)$  - points $\pmb{0},  \frak{z}'$ and  $\frak{z}''$. So Lemma \ref{lemma:1:points} gives $\nu$ with (\ref{nuod}) and
 formulae 
  (\ref{fromu}) analogously to (\ref{fdof})  show 
   that $ a_{\nu+1} = a_{\nu+2} = 1$, while  partial quotients 
$a_\nu$  and $a_{\nu+3}$ are large.
  
  In {\bf case} 2$^0$ we see that the rectangle
  $R(x'+x'', |y'|)\cap \Lambda_\alpha (t) = \{ \pmb{0},\frak{z}',\frak{z}'+\frak{z}''\}$,
  and the last three points lie on the boundary of $R(x'+x'', |y'|)$. By Lemma \ref{lemma:1:points} this means that
 $$
\exists \, \nu :\,\,\,\,\,\frak{z}'=\frak{z}_\nu,\,\,\,\,\,\,\frak{z}'+\frak{z}'' = \frak{z}_{\nu+1}.
$$
Now 
$\frac{|y'|}{|y'+y''|}=
\frac{\xi_\nu}{\xi_{\nu+1}}
$
is very large, while 
$
\frac{x'+x''}{x'}=
\frac{q_{\nu+1}}{q_\nu} \in (2, 2+\varepsilon_1)$.
By formulae  (\ref{fromu}) this means that $a_{\nu+1} = 2$, while 
$a_\nu$ and $a_{\nu+2}$ are very large.

In all cases, we got patterns (\ref{trip0}). Hence, the necessity for statement ({\bf a})  is proven.

  \vskip+0.3cm

  Now we prove the necessity condition in the statement ({\bf b}). The situation here is simpler than in the statement  
   ({\bf a}). Let a subsequence  of lattices  $\Lambda_\alpha(t)$ from Proposition \ref{proposition:1} converges to $\overline{\Lambda}_2^+$ or $\overline{\Lambda}_2^-$. We take points $\frak{z}' =(x',y'),\, \frak{z}'' =(x'',y'')\in \Lambda_\alpha(t)$
   which are close to $ \pmb{z}', \pmb{z}''$ in  $\overline{\Lambda}_2^+$ or $\overline{\Lambda}_2^-$. Then
   analogously to (\ref{dopp})  by (\ref{lat12}) we have
   $$
   \left|
\frac{x'}{x''} - \alpha^*(\Omega_2^\pm)
\right| =
\left|
\frac{x'}{x''} - 
   \varsigma_p
\right| <\varepsilon_1,\,\,\,\,\,
\left|
\frac{|y'|}{|y''|} - \alpha(\Omega_2^\pm)
\right| =
\left|
\frac{|y'|}{|y''|} - (1+\varsigma_p)
\right| <\varepsilon_1.
$$ 
But in this case $1+\varsigma_p >\varsigma_p$.
So for $\Lambda_2^\pm$ we have $x''>x', |y''|< |y'|$ and   for $\varepsilon_1$ small enough in the rectangle
$ R(x'', |y'|)$ there are no other points of $\Lambda_\alpha (t)$ but $\pmb{0},\pmb{z}',\pmb{z}''$.
This means that by Lemma \ref{lemma:1:points} we get the existence of $\nu$ with 
$$
   \frak{z}' =\frak{z}_\nu,\,\,\,\,\,\,\frak{z}'' = \frak{z}_{\nu+1} .
   $$
   We take into account that irrational numbers have a unique continued fraction expansion, meanwhile, for rational  numbers, there exist exactly two expansions
   (with last digit equal to $1$ and with last digit equal at least $2$).
   Now formulae (\ref{fromu})  immediately give the desired patterns. Thus, the necessity in the statement ({\bf b}) is proven.
   
   \vskip+0.3cm
   
   As for the necessity in the statement ({\bf c}), first, consider the cases when $\Lambda_\alpha (t)$ converges to a lattice of the form $\overline{\Lambda}_3^\pm (a)$ with $a \neq 0$. Again, we consider points 
   $ \frak{z}' = (x',y'), \frak{z}''=(x'',y'')\in \Lambda_\alpha (t)$ close to $\pmb{z}',\pmb{z}''$.
   Then  again $R(x'',|y'|)$  does not contain points of $\Lambda_\alpha (t)$ different from $ \pmb{0},\pmb{z}',\pmb{z}''$ and by Lemma \ref{lemma:1:points}
   $\frak{z}' = \frak{z}_\nu, \frak{z}''= \frak{z}_{\nu+1}$ with some $\nu$. But  in this case 
   we also have $ R(x'+x'', |y''|) \cap \Lambda_\alpha (t) = \{ 0,\frak{z}_{\nu+1}, \frak{z}_\nu + \frak{z}_{\nu+1}\}$.
So  $ \frak{z}_{\nu+2} = \frak{z}_{\nu}+\frak{z}_{\nu+1}$
   and $ a_{\nu+2} = 1$.

Now in case ({\bf c})  by (\ref{lat1920}) and formulae (\ref{fromu})  we see that $\alpha_{\nu+1}^*+\alpha_{\nu+2}$ is close to $2$. 
   If $a_{\nu +3}  >1 $, then
   $\alpha_{\nu+2} = [1;a_{\nu+3}, a_{\nu+4},\ldots]$ and 
  $ 2-\alpha_{\nu +2} = [0;1, a_{\nu+3}-1, a_{\nu+4},\ldots]$.
  Assume that the limits are irrational.
  Then by the uniqueness of continued fractions' expansion, we get the condition with the first patterns from (\ref{asym}) with $ b_1= a_{\nu+2}-1, b_2 = a_{\nu+3},\ldots$.
    If $a_{\nu+2}  =1 $, then we get the second pattern from (\ref{asym}).
    If there is no limit lattice   $\overline{\Lambda}_3^\pm (a)$ with  irrational value of $\alpha(\Omega_3^\pm)$, then there exists a particular limit lattice 
    $\overline{\Lambda}_3^\pm (a)$ with rational value of $\alpha(\Omega_3^\pm)$.
    Then the limits are (both) rational and we get the other eight patterns from the formulation of the statement ({\bf c}).

   \vskip+0.3cm
   In the statement ({\bf c}), it remains to consider the case when a subsequence of lattices $\Lambda_\alpha (t)$ converges to $\overline{\Lambda}_3^+(0)$.
   The situation here is similar to the one in statement  ({\bf a}). We may suppose that $\frak{z}' = (x',y'), \frak{z}''=(x'',y'')\in \Lambda_\alpha (t)$ are close to 
   points $ \pmb{z}' +\pmb{z}'', \pmb{z}'' \in \overline{\Lambda}_3^+(0)$.
   Without loss of generality, we assume that $ x'<x''$. Then if $ |y'|>|y''|$ we come to the pattern $x,1,1,y$, and if 
   $ |y'|<|y''|$ we get the pattern $x,2,y$ with large $x$ and $y$.
   Thus, we get the necessity in the statement  ({\bf c}).

     \vskip+0.3cm
     Finally, we deal with the necessity in the statement ({\bf d}). Here we have two different situations. When there is a subsequence of lattices $\Lambda_\alpha (t)$ which converges to $\Lambda_4^\pm(\varphi)$ with $\varphi = 0 $ or $\varphi = \frac{\pi}{6}$ the situation is similar to the one in statement ({\bf a}). We consider two cases and conclude that in continued fraction expansion for $\alpha$, there should be patterns of the type $ x,1,1,y$ or $ x,2,y$ with large $x $ and $y$.
     In the case when $ \varphi  \in\left(0,\frac{\pi}{6}\right)$ we conclude that $\Lambda_\alpha (t) \cap  R(x'',|y'|) = \{ \pmb{0},\frak{z}',\frak{z}''\}$
     (here again $ \frak{z}',\frak{z}''\in \Lambda_\alpha(t)$ are points close to $\pmb{z}',\pmb{z}''\in \Lambda_4^\pm(\varphi)$). By Lemma \ref{lemma:1:points} and formulae (\ref{fromu}) we see that $\alpha_{\nu+1}^* $ and $\alpha_{\nu+2}$ should be close to 
     $\alpha^*(\Omega_4^\pm(\varphi))$ and  $\alpha(\Omega_4^\pm(\varphi))$  respectively. Equation (\ref{lat200}) leads to the patterns \eqref{asym1}, satisfying \eqref{rararar} from  Theorem \ref{thm:charachterization}.
      \vskip+0.3cm
      
      The sufficiency in all the cases of Theorem \ref{thm:charachterization} is clear, because if we have $\alpha$ with desired patterns in continued fraction expansion, each pattern
      gives a value of $t$ such that the corresponding points $ \frak{z}_\nu, \frak{z}_{\nu+1} \in \Lambda_\alpha (t) $ are close to the points $\pmb{z}',\pmb{z}''$ of
      critical lattice and so  the trajectory 
      $\{ \Lambda_\alpha (t) \}_{t\ge 1}$ has a limit point in $\overline{\frak{L}}_{p}$.$\hfill\qed$

\section{ More basic facts from continued fraction theory}\label{sec:basic:cf}
In this section, we provide the necessary basic facts from the theory of continued fractions, which we use in the upcoming sections. 

Recall that the $n$th convergent of $x$ is defined as
\begin{align*}
\frac{p_{n}\left(x\right)}{q_{n}\left(x\right)}=\left[a_{1}\left(x\right),\ldots,a_{n}\left(x\right)\right].
\end{align*}
For simplicity, we will denote $p_n=p_{n}\left(a_{1},\ldots,a_{n}\right)$, $q_n=q_{n}\left(a_{1},\ldots,a_{n}\right)$ when there is no ambiguity. Both $p_n$ and $q_n$ satisfy the iterative formula, for any $k\ge 2$ we have
\begin{equation}\label{rule}
\begin{split}
p_{k}&=a_{k}p_{k-1}+p_{k-2},\\
q_{k}&=a_{k}q_{k-1}+q_{k-2},
\end{split}
\end{equation}
where $p_{0}=0,p_{1}=1$ and $q_{0}=1,q_{1}=a_1$.

We will sometimes use a notation $\langle a_1,\ldots,a_n\rangle := q_n(a_1,\ldots,a_n)$. For different values of $n\in\N$, this function $\langle a_1,\ldots,a_n\rangle : \N^n \to \N$ depends on $n$ variables $a_1,\ldots,a_n$ and is called {\it a continuant}. This is a strictly growing function with respect to any of its variables. Moreover, it is a strictly decreasing function upon deleting one of its variables. To be precise, for any $n$ and any $(a_1,\ldots,a_n)\in\N^n$, we have
$$
q_n(a_1,\ldots,a_{i-1},a_i,a_{i+1},\ldots,a_n) \ge q_{n-1}(a_1,\ldots,a_{i-1},a_{i+1},\ldots,a_n) \text{ for any } i\in\{1,\ldots,n\}.
$$

For $n\ge1$ and any $\left(a_{1},\ldots,a_{n}\right)\in\N^n$, where $a_{i}\in\mathbb{N}, 1\le i\le n$, we define $\textit{a cylinder of order }n$, denoted by $I_{n}\left(a_{1},\ldots,a_{n}\right)$ as
\begin{align*}
I_{n}\left(a_{1},\ldots,a_{n}\right):=\left\{x\in\left[0,1\right):a_{1}\left(x\right)=a_{1},\ldots,a_{n}\left(x\right)=a_{n}\right\}.
\end{align*}
We have the following standard facts about cylinders.
\begin{proposition}[{\protect\cite{Khinchin_book}}]\label{range}
For any $n\ge 1$ and $\left(a_{1},\ldots,a_{n}\right)\in\mathbb{N}^{n}$, $p_{k},q_{k}$ are defined recursively by $\left(\ref{rule}\right)$ where $0\le k\le n$, we have
\begin{equation}
I_{n}\left(a_{1},\ldots,a_{n}\right)=\left\{
  \begin{array}{ll}
   \left[\dfrac{p_{n}}{q_{n}},\dfrac{p_{n}+p_{n-1}}{q_{n}+q_{n-1}}\right) &{\rm if  }
   \,n\textrm{ is even},\\
   \left(\dfrac{p_{n}+p_{n-1}}{q_{n}+q_{n-1}},\dfrac{p_{n}}{q_{n}}\right] &{\rm if    }\,n\textrm{ is odd}.\\
  \end{array}
\right.
\end{equation} 
Therefore, the length of a cylinder is given by
\begin{equation}\label{length}
\left|I_{n}\left(a_{1},\ldots,a_{n}\right)\right|=\dfrac{1}{q_{n}\left(q_{n}+q_{n-1}\right)}.
\end{equation}
\end{proposition}
For a better estimation, we need the following results. For any $n\ge 1$ and $\left(a_{1},\ldots,a_{n}\right)\in\mathbb{N}^{n}$, we have
\begin{lemma}[{\protect\cite{Khinchin_book}}] For any $1\le k\le n$, we have
\begin{equation}\label{c1}
\dfrac{a_{k}+1}{2}\le\dfrac{q_{n}\left(a_{1},\ldots,a_{n}\right)}{q_{n-1}\left(a_{1},\ldots,a_{k-1},a_{k+1},\ldots,a_{n}\right)}\le a_{k}+1.
\end{equation}
\end{lemma}
\begin{lemma}[{\protect\cite{Khinchin_book}}] For any $k\ge 1$, we have
\begin{equation}\label{c2}
\begin{split}
q_{n+k}\left(a_{1},\ldots,a_{n},a_{n+1},\ldots,a_{n+k}\right)&\ge q_{n}\left(a_{1},\ldots,a_{n}\right)\cdot q_{k}\left(a_{n+1},\ldots,a_{n+k}\right),\\
q_{n+k}\left(a_{1},\ldots,a_{n},a_{n+1},\ldots,a_{n+k}\right)&\le 2q_{n}\left(a_{1},\ldots,a_{n}\right)\cdot q_{k}\left(a_{n+1},\ldots,a_{n+k}\right).
\end{split}
\end{equation}
\end{lemma}
For any $n\ge 1$, since $p_{n-1}q_{n}-p_{n}q_{n-1}=\left(-1\right)^{n}$, combined with a simple calculation, we get the bound
\begin{equation}\label{qE}
q_{n}\ge 2^{\left(n-1\right)/2}.
\end{equation}

We will also use a trivial bound 
\begin{equation}\label{trivialbound}
a_n\cdot q_{n-1}\le q_{n} \le (a_n+1)\cdot q_{n-1}.
\end{equation}

An easy corollary of the results listed above is the inequality
    \begin{equation}\label{req1}
      |I_{n+1}(a_1,\ldots,a_n,b)|> \frac{1}{(b+1)^{2}}        |I_{n}(a_1,\ldots,a_n)|
      .\end{equation}





   \section{ Proof of Theorem \ref{thm:bad:dirichlet:fulldim}}\label{sec:proof:bad:dirichlet:fulldim}
     \vskip+0.3cm
     
    We prove an auxiliary statement from which Theorem \ref{thm:bad:dirichlet:fulldim} follows immediately.
    
                   \vskip+0.3cm
  \begin{lemma}\label{lemma:good}
     Consider an infinite sequence of finite patterns
    \begin{equation}\label{patteR}
    W:\,\,\,\,\,
    b_{1,1},\ldots, b_{1,r_1}; \,\,\, \ldots\,\,\, ;      b_{t,1},\ldots, b_{t,r_t};\,\,\,
        b_{t+1,1},\ldots, b_{t+1,r_{t+1}}; \,\,\, \ldots \,\,\,   , \,\,\ b_{i,j} \in \mathbb{Z}_+.
        \end{equation}
        Then the set 
        $$ 
        \mathbf{BA}_W =
        \{ \alpha = [a_0;a_1,a_2,\ldots,a_\nu,\ldots]\in \mathbf{BA}: 
        \text{ for any}\,\, t \,\,  \text{there exists}\,\, \nu
                $$
        $$
       \hskip+2.7cm \text{such that }\,\, a_\nu=b_{t,1}, a_{\nu+1} = b_{t,2},\ldots,a_{\nu+r_t-1} = b_{t,r_t}\}
        $$
        has full Hausdorff dimension.
        \end{lemma}
     \vskip+0.3cm
    To deduce Theorem \ref{thm:bad:dirichlet:fulldim} from this lemma, in the case $ p\in  (1,2)\cup (p_0,\infty)$ with $ \varsigma_p \in \mathbf{BA}$ we simply insert patterns \eqref{pat} coming from $\varsigma_p$ of growing lengths. In the case $p=1$, we insert almost symmetrical patterns \eqref{asym} of bounded numbers. In the case $p=2$, we take a number $\beta\in\mathbf{BA}$. Easy to see that in this case $\frac{3}{\beta}=\beta^*\in\mathbf{BA}$. Then we insert patterns \eqref{asym1}.

    The proof of Lemma \ref{lemma:good} relies on standard techniques from the metrical theory of continued fractions. We deduce it from a well-known result by Good \cite{G}.
     
         \vskip+0.3cm
    \begin{proof}[Proof of Lemma \ref{lemma:good}]  
      For any $0<\varepsilon <1$, we give an explicit construction of the set
      $         \mathbf{BA}_W (\varepsilon) \subset         \mathbf{BA}_W $  whose Hausdorff dimension is at least $1-\varepsilon$. This will be enough to conclude that the Hausdorff dimension of $\mathbf{BA}_W$ is full. We choose the following parameters. Take $0<\varepsilon<1$ and put
      \begin{equation}\label{papara}
       M=   \left\lceil \frac{8}{\varepsilon\log 2}\right\rceil,\,\,\,\,\,
      Q_t = \prod_{j=1}^{r_t} (b_{t,j}+1)^2,\,\,\,\,\,
      \nu_t = \left\lceil \frac{2\log Q_t}{\varepsilon\log 2}\right\rceil
      \end{equation}
      and define a sequence $n_t$ by
      $$
      n_1 = 1,\,\,\,\, n_{t+1} = n_t +r_t + \nu_{t+1} +1.
      $$
      Note that from our choice of parameters (\ref{papara}) it follows that 
        \begin{equation}\label{papara1}
        2^\varepsilon\left(1-\frac{2}{M}\right)\ge 2^{\frac{\varepsilon}{2}}
\ge
      Q_t^{\frac{1}{\nu_t}}.
            \end{equation}
                  Define the set 
      $$
      \mathbf{BA}_W (\varepsilon) =
      \{
      \alpha = [a_0; a_{1,1},\ldots,a_{1,\nu_1},  b_{1,1},\ldots, b_{1,r_1},
      \ldots, 
      a_{t,1},\ldots,a_{t,\nu_t},  b_{t,1},\ldots, b_{t,r_t},
    \ldots]\} \subset        \mathbf{BA}_W.
      $$
      Here we mean that $\mathbf{BA}_W (\varepsilon)$ consists of those $\alpha$ for which in its continued fraction expansion     
  there are patterns (\ref{patteR}) under our consideration with partial quotients $b_{i,j}$, 
  and $ a_{i,j} $  are arbitrary positive integers satisfying the condition   $ a_{i,j} \le M$.

      We take the following inequality from the famous paper by Jarnik
      (see p. 104, lines 9-10 from 
      \cite{JA}), where he proved that  the set of badly approximable numbers $\mathbf{BA}$  has full Hausdorff dimension:
        \begin{equation}\label{req2}
    \sum_{a=1}^M  |I_{n+1}(a_1,\ldots,a_n,a)|^s> 2^\varepsilon\left(1-\frac{2}{M}\right)    |I_{n}(a_1,\ldots,a_n)|^s > Q_t^{\frac{1}{\nu_t}}      |I_{n}(a_1,\ldots,a_n)|^s,\,\,\,\,\,
    s = 1 -\varepsilon, 
      \end{equation}
      where in the last inequality we used  \eqref{papara1}. 
      
      Inequalities 
      (\ref{req1}) and (\ref{req2}) are valid for arbitrary continued fractions with partial quotients
      $a_1,\ldots,a_n$. Now we use this fact for convergents of numbers from  the set 
      $
      \mathbf{BA}_W (\varepsilon) $.
      
      We define two functions $ f(n),\, g(n),\, n = 1,2,3,\ldots $. 
      Let  $a_n $ be the $n$-th  partial quotient  for a number $ \alpha \in 
      \mathbf{BA}_W (\varepsilon) $.        
      Then either $ a_n = a_{i,j} $ and in this case we put
      $$
      f(n) = 1,\,\,\,\ g(n) = M,\,\,\,\,\, \text{so}\,\,\,\, f(n) \le a_n \le g(n),
      $$
      or 
        $ a_n = b_{i,j} $ and in this case we put
      $$
      f(n) = g(n) = b_{i,j},\,\,\,\,\, \text{so again}\,\,\,\, f(n) \le a_n \le g(n).
      $$
      We should note that this definition of $f(n), g(n)$ does not depend on a particular choice of  
$ \alpha \in 
      \mathbf{BA}_W (\varepsilon) $.  Now inequalities     \eqref{req1}, \eqref{req2} together with definition of
      $Q_t$  (\ref{papara})  give us    
      $$
          \sum_{a: \, f(n+1)\le a\le g(n+1)}  |I_{n+1}(a_1,\ldots,a_n,a)|^s>  F(n) \, |I_{n}(a_1,\ldots,a_n)|^s  ,
          \,\,\,\,\, s = 1-\varepsilon
          $$
          with 
          $$
          F(n)
          =\begin{cases}
\frac{1}{(b_{t,j}+1)^2}\,\, \text{if}\,\,  \sum_{i=1}^{t-1}(\nu_i+r_i)+ \nu_t< n \le  \sum_{i=1}^{t}(\nu_i+r_i),
\cr
        Q_{t+1}^{\frac{1}{\nu_{t+1}}} \,\, \text{if}\,\,  \sum_{i=1}^{t}(\nu_i+r_i)< n \le  \sum_{i=1}^{t}(\nu_i+r_i)+\nu_{t+1}.
          \end{cases}
          $$ 
    
         This function, by construction, satisfies
          $$
          \prod_{k=1}^n F(k) \ge 1\,\,\,\,\,\text{for all}\,\, n  \text{ sufficiently large}.
          $$
         So we are just in the conditions of Lemma  5 from  \cite{G}. The conclusion of this lemma is that the Hausdorff dimension of the set 
               $\mathbf{BA}_W (\varepsilon) $ is at least $  s= 1-\varepsilon$.
     \end{proof}

              \vskip+0.3cm

\section{ The Hausdorff dimension of $\mathbf{DI}_p\setminus \mathbf{BA}$ and $\mathbf{DI}_{p_1}\setminus\mathbf{DI}_{p_2}$}\label{sec:di:minus:bad}

In this section, we prove Theorem \ref{theorem:diminusbad:dimension} and Theorem \ref{thm:di2minusd1:di1minusd2}. To do this, we will utilize some symbolic spaces and facts from the theory of continued fractions.


Fix a fast-growing sparse sequence $\{n_i\}_i$ of natural numbers. Next, fix a sequence of words of natural numbers $\AA= (A_1,A_2,\ldots)$, so that for each $i\in\N$, a word $A_i$ has length \begin{equation}\label{pattern:length}
\ell_i=|A_i| < \Delta_{i+1}=n_{i+1}-n_i
\end{equation} 
and each $A_i=(a_{i,1},\ldots,a_{i,\ell_i})$. 

Define a map $S' =S'(\{n_i\}_i,\AA) \colon \N^\N \to \N^\N$ acting by
\begin{equation}\label{definition:map}
(a_1,\ldots,a_{n_1-1},a_{n_1},\ldots) \to (a_1,\ldots,a_{n_1-1},A_{1},a_{n_1},\ldots,a_{n_2-\ell_1-1},A_{2},a_{n_2-\ell_1},\ldots)=(b_1,b_2,\ldots).
\end{equation}
Basically, $S'$ "inserts" a word $A_i$ starting at the place with index $n_i$, while moving apart $a_k$'s accordingly. In particular, for any $i\in\N$, we have $b_{n_i}= a_{i,1}$. For a fixed map (that is for fixed $\{n_i\}_i$ and $\AA$), we define a counting function
$$
\omega(n) :=  \# \{k  \le  n\colon \, b_k = a_{i,j} \text{ for some } (i,j) \}.
$$
 This function counts how many 'new' elements appeared under $S$ among the first $n$ terms.


Consider an inverse map $P':= S'^{-1}$, so that $P'$ takes a sequence from $\N^\N$ and for all $i$ removes elements in the places with indices from $n_i$ to $n_i+\ell_i-1$ and moves all remaining parts together.

It is well-known that for any real irrational number $x\in [0,1)$, the continued fraction expansion is uniquely determined, so we can define a bijection $\theta: \N^\N \to [0,1)\setminus \Q$ from a space of sequences of positive integers to real numbers by
$$
\theta(a_1,a_2,\ldots) = [0;a_1,a_2,\ldots].
$$

For any fixed set $B'\subseteq \N^\N$, we will denote the image of this set under $\theta$ by the same letter without the prime symbol, that is $B:= \theta(B')$.

Also, define a map $P: [0,1] \to [0,1]$ by $P(x) = \theta \circ  P' \circ \theta^{-1} (x)$. Similarly, we define $S(x)= \theta \circ  S' \circ \theta^{-1} (x)$.
Finally, define
$$E'_N = \{x = (x_1,x_2,\ldots)\in \N^\N: 1\le x_{n}\le N \text{ for all } n\in\N \},$$
$$
E_N =\theta(E'_N) = \{x\in[0,1)\setminus\Q : 1\le a_{n}(x) \le N \text{ for all } n\in\N \},
$$
and for $c>0$ we let
$$
\Omega(c) = \{ x\in[0,1)\setminus\Q :\, 1\le a_n(x) \le n^c \text{ for all } n\in\N \}.
$$

The main instrument is the following statement.

\begin{lemma}\label{lemma:holder}

Fix a subset $B \subseteq \mathbf{BA}$ and write $B_N = B\cap E_N$. 

Fix a sequence $\{n_i\}_i$ and a collection $\AA = (A_1,A_2,\ldots)$.
Assume that the following three conditions are satisfied.
\begin{enumerate}
\item The sequence $\{n_i\}_i$ and $\AA = (A_1,A_2,\ldots)$ satisfy \eqref{pattern:length};

\item For each $N\in\N$, there exists $c=c(N)>0$ such that
$$
S(B_N) \subseteq \Omega(c);
$$

\item The function $\omega(n)$ satisfies $\omega(n)=\overline{o} (n^{1-\epsilon_1})$ for some $\epsilon_1>0$.
\end{enumerate}

Then for any $N\in\N$ and any $\varepsilon>0$, the map $P|_{S(B_N)}:=P(\{n_i\}_i,\AA)|_{S(B_N)}$ satisfies the local $\frac{1}{1+\varepsilon}$-H{\"o}lder condition, the map $S|_{B_N}:=S(\{n_i\}_i,\AA)|_{B_N}$ is locally Lipschitz, and hence 
$$
\hdim S(B)= \hdim B.
$$
\end{lemma}




We believe that Lemma \ref{lemma:holder} is of independent interest and is reusable in other settings of sets of continued fractions, for example, in the questions about sets with restrictions on the growth rate of functions of partial quotients, see \cite{HuWuXu,hls,HussainShulgaFiniteProgressions,WaWu08}. We also note that in all of the applications of this lemma in the present paper, we take sets $B$ with $\hdim B=1$. As one can see from the statement, this is not a necessary requirement. Also, this lemma can be seen as a generalization of Lemma \ref{lemma:good} with a deeper and more precise result.

First, let us derive Theorem \ref{theorem:diminusbad:dimension} from the comments after Theorem E, Lemma \ref{lemma:holder}, and Theorem \ref{thm:charachterization}. We prove Lemma \ref{lemma:holder} in a separate Section \ref{aadddd}.

\begin{proof}[Proof of Theorem \ref{theorem:diminusbad:dimension}]
For each $p\in [1,\infty)$, let $B=B(p):=\mathbf{DI}_p \cap \mathbf{BA}$.
By the results of Kleinbock and Rao, we know that
$$
\hdim B = 1.
$$

Now for $p\in\{1\}\cup\{2\}\cup(2,p_0)\cup\{p_0\}$ and for $p\in(1,2)\cup (p_0,\infty)$ with $\varsigma_p\notin\Q$, we apply Lemma \ref{lemma:holder} with this choice of $B$, and choose the parameters to be $n_i=2^i$ and $A_i = i$ for any $i\in\N$, that is we insert a single digit in the places with indices being powers of $2$. Note that for this choice of $n_i$, the quantity  $\Delta_i = n_i  - n_{i-1}$ is an even number.

To highlight diverse types of applications of the lemma, in the case $p\in(1,2)\cup (p_0,\infty)$ with $\varsigma_p\in\Q$, we take the parameters to be $n_i = 2^{i+100}$ and $A_i = (2,i,4)$ for all $i\in\N$, so that starting at places with indices being powers of $2$, we insert three digits, where the middle one is growing, and two on the sides are bounded. We note that in this case, we can also come to a contradiction by inserting the same pattern as in all previous cases.

By Lemma \ref{lemma:holder}, for the set $G=G(p):=S(B)$ we get the equality
$$
\hdim G = \hdim B = 1.
$$

The last thing to show is that $G\subset \mathbf{DI}_p \setminus \mathbf{BA}.$

Clearly, $G$ avoids the set of badly approximable numbers, as it has an infinite subsequence of growing partial quotients coming from the sequence $\AA$. We need to show that any element $x\in G$ satisfies $x \in \mathbf{DI}_p$. 

By {\it distance} between two partial quotients $a$ and $b$ in the pattern, we mean the number of partial quotients lying strictly between $a$ and $b$. We also refer to the partial quotients from $A_i$ as {\it large} partial quotients (as they are growing to infinity).

Now, we distinguish cases of different values of $p$.

\begin{enumerate}
    \item Case $p=1$.  By definition of the map $S$, we need to make sure that the procedure \eqref{definition:map} will not generate restricted patterns in a number $$x=(a_1,\ldots,a_{n_1-1},A_{1},a_{n_1},\ldots,a_{n_2-2},A_{2},a_{n_2-1}\ldots),$$ provided that there were no restricted patterns in $(a_1,\ldots,a_{n_1-1},a_{n_1},\ldots) $.  As we have three different types of patterns, let us consider three possibilities.
    \begin{enumerate}
        \item Patterns $b_k,\ldots,b_1,1,1,b_1+1,b_2,\ldots,b_k$ or $b_k,\ldots,b_2,b_1+1,1,1,b_1,b_2,\ldots,b_k$ for arbitrarily long $k$. 

        Clearly, those will not appear in $x$ if they were not originally present in the preimage. This is because $x$ has the same blocks of digits as the preimage except for a strictly growing sequence $(A_1,A_2,\ldots)$. As the restricted pattern of this case is (almost) symmetrical, the addition of strictly growing numbers cannot generate a restricted pattern. 

        \item Patterns of finite length $x,b_k,\ldots,b_1,1,1,b_1+1,b_2,\ldots,b_k,y$ and others with $\min(x,y)\to\infty$.

        All of these eight patterns in this case have two unbounded partial quotients on the sides. As all partial quotients in $(a_1,\ldots,a_{n_1-1},a_{n_1},\ldots) $ are bounded, the only candidates for unbounded are $A_i$'s from $\AA$. But one sees that the distance between large partial quotients $A_i$ and $A_{i+1}$ grows with $i$, so we cannot have infinitely many patterns of the same finite length.

        \item Patterns
        $x,2,y$ or $x,1,1,y$ with $\min(x,y)\to\infty$.

        Here, we are required to have two unbounded partial quotients with only one or two partial quotients between them. The only unbounded partial quotients in our case are $A_i$'s, which are located at a growing distance between two consecutive ones, thus this pattern cannot occur infinitely many times.

    \end{enumerate}

Therefore, we get that $G\subset \mathbf{DI}_1 \setminus \mathbf{BA}$.

\item Case $2<p<p_0$. The only two patterns to be avoided in this case are the patterns with two growing partial quotients being at a distance 1 or 2 from each other infinitely many times, that is pattern $x,1,1,y$ or $x,2,y$ with $\min(x,y)\to\infty$. As the only growing partial quotients are located in the places with indices $2^i$, the distance between two consecutive large partial quotients is growing, hence this procedure will not generate infinitely many restricted patterns.

Therefore, we get that $G\subset \mathbf{DI}_p \setminus \mathbf{BA}$ for $2<p<p_0$.
\item Case $p\in(1,2)\cup (p_0,\infty)$ with $\varsigma_p\in\Q$. All restricted patterns in this case have two growing partial quotients on the sides of it, while the rest of the partial quotients in between two large partial quotients form a symmetrical or almost symmetrical pattern. Once again, the growing partial quotients can only come from $A_i$'s, but for each restricted pattern, the partial quotient right after the first large partial quotient $x$ and the last partial quotient before the last large partial quotient $y$ are either equal to each other, or at least one of them is equal to $1$. However, we inserted patterns of the form $(2,i,4)$, so neither of the two options can happen. Thus, restricted patterns will never occur.

Alternatively, one can see that we also get a contradiction with Mahler's compactness, as we need to have a subsequence of patterns with 'tails' converging to some fixed numbers. This cannot happen as the distance between large partial quotients from different $A_i$'s grows to infinity.

Therefore, we get that $G\subset \mathbf{DI}_p \setminus \mathbf{BA}$ for $p\in(1,2)\cup (p_0,\infty)$ with $\varsigma_p\in\Q$.

\item Case $p\in(1,2)\cup (p_0,\infty)$ with $\varsigma_p\notin\Q$. We only need to avoid a pattern $s_\nu,\ldots,s_1,1,s_1,\ldots,s_\nu$. As this pattern is symmetric, and we inserted a strictly growing sequence of partial quotients, that restricted pattern cannot occur, as for any fixed number with bounded partial quotients, this new strictly growing sequence will eventually outgrow the bound on 'small' partial quotients.

Therefore, we get that $G\subset \mathbf{DI}_p \setminus \mathbf{BA}$ for $p\in(1,2)\cup (p_0,\infty)$ with $\varsigma_p\notin\Q$.

\item Case $p=p_0$. Recall the Remark \ref{remark:p0}, stating that in the case $p=p_0$ we need to avoid patterns $x,1,1,y$ or $x,2,y$ or the same patterns as in the case $p\in(1,2)\cup (p_0,\infty)$. It is not known whether $\varsigma_{p_0}$ is rational or not, so we cannot choose between cases ${\bf (b1)}$ or ${\bf (b2)}$ of Theorem \ref{thm:charachterization}. However, we can still prove the result for $p_0$ by showing that no matter if the value of $\varsigma_{p_0}$ is rational or irrational, we will avoid restricted patterns of that case.

Once again, as we are inserting large partial quotients with a growing distance between two consecutive large ones, no patterns of finite length can occur infinitely often. This eliminates patterns $x,1,1,y$ or $x,2,y$ or the ones from ${\bf (b1)}$. The pattern $s_\nu,\ldots,s_1,1,s_1,\ldots,s_\nu$ with growing $\nu$ is eliminated in the exact same way as before, because we cannot get a symmetrical pattern from bounded partial quotients with inserted strictly growing numbers.

Therefore, we get $G\subset \mathbf{DI}_{p_0} \setminus \mathbf{BA}$.

\item  Case $p=2$. First, let us note that any patterns of finite lengths, that is $x,1,1,y$ or $x,2,y$ with $\min(x,y)\to\infty$, as well as eight patterns coming from \eqref{asym1} in the case when $\beta^*$ and $\beta$ are rational, cannot occur infinitely many times by the same logic as before. Indeed, the only large partial quotients will be located in the places with indices $2^n$, so the distance between two consecutive large partial quotients grows. Thus, no finite pattern can occur.

Now we need to deal with infinite patterns \eqref{asym1}. 



Both of the patterns \eqref{asym1} have a partial quotient $1$ in the middle. We will call it a {\it central} partial quotient. 

We will prove that every $x\in G$ will avoid the patterns of this form by contradiction. Assume that one of them will appear infinitely many times. Clearly, those patterns should contain at least one partial quotient coming from $A_i$, because otherwise, we get a contradiction with the fact that the original number is an element of $\mathbf{DI}_2$. Indeed, if a restricted pattern appears and it doesn't contain any partial quotients from $A_i$, then it consists only of bounded partial quotients of the original number, which doesn't have any restricted pattern as an element of $\mathbf{DI}_2$.

Next, we claim that there exists a constant $D\in\N$, such that in each appearance of the pattern from \eqref{asym1}, the distance between the central partial quotient to the closest large partial quotient from $A_i$ is at most $D$. Indeed, assume the opposite, that there exists an integer sequence $\{D_i\}_i$ with $D_i\to\infty$ with $i\to\infty$, and a corresponding sequence of restricted patterns $\{P_i\}_i$ with $|P_i|\to\infty$ when $i\to\infty$, for which the distance between central partial quotient are the closest large partial quotient in the pattern $P_i$ is equal to $D_i$. But this means that for any such pattern $P_i$, we can consider a subpattern $P_i'$, which consists of a central partial quotient $1$, with $D_i-1$ partial quotients on one side, and $D_i-2$ partial quotients on the other side of it. This subpattern has total length $|P_i'|=2D_i-2$, which tends to infinity as $i\to\infty$. Moreover, this pattern has no large partial quotients in it. So we obtained a sequence of patterns $P_i'$ of arbitrarily large length with only bounded partial quotients in it, thus, we get a contradiction with the fact that the original number was an element of $\mathbf{DI}_2$ and therefore avoided all restricted patterns of this case. 

Now we can safely assume that we have a sequence of arbitrarily long patterns, for which the distance between their central partial quotient and the closest large partial quotient is bounded by some fixed constant $D$. By the structure of the map $S$, the distance between large partial quotients grows to infinity the further we go in the continued fraction expansion. Thus, eventually, we will have patterns of one of the two forms \eqref{asym1}, such that on one side of the central partial quotient, we have a sequence of bounded partial quotients of growing length, whereas on the other side, we always have a growing partial quotient on a finite distance from the central partial quotient. This contradicts the compactness argument. Indeed, the compactness argument in the case of $p=2$ states that there exist $\beta$ and $\beta^*$ for which there exists a sequence of patterns such that on one side of the central partial quotient, the subpattern tends to $\beta$ and on the other side, the subpattern tends to $\beta^*$. In our case, as shown, one side tends to a rational number, while the other one tends to an irrational one. As in the product, they should give an integer number, we get a contradiction.  Thus, the assumption that restricted patterns occurred was incorrect.

Therefore, we get that $G\subset \mathbf{DI}_2 \setminus \mathbf{BA}$.

\end{enumerate}

\end{proof}


Similarly, using Theorem \ref{thm:charachterization}, Lemma \ref{lemma:holder}, Theorem E, and the compactness argument, we can derive Theorem \ref{thm:di2minusd1:di1minusd2}.

\begin{proof}[Proof of Theorem \ref{thm:di2minusd1:di1minusd2}]

First, let us deal with $\mathbf{DI}_{1}\setminus\mathbf{DI}_{2}$. There are multiple ways in which we can apply  Lemma \ref{lemma:holder} to this setting. The strategy is to start with the set $\mathbf{DI}_1\cap\mathbf{BA}$ and then insert some patterns from \eqref{rararar} using the map $S$ to ensure we are in $\mathbf{DI}_2^c$, and then check that we are still in $\mathbf{DI}_1$. The easiest case would be to include finite patterns. 

Consider $\beta^* = \frac43=[1;3]$ and associated with it $\beta=\frac94=[2;4]$, so that $\beta^*\cdot\beta=3$.

The corresponding restricted pattern for this choice of $\beta^*, \beta$ is $x,3,2,1,3,4,y$ with $\min(x,y)\to\infty$. To secure the existence of those patterns, we let $n_i=2^{i+100}$ and $A_i = (n_i,3,2,1,3,4,n_i+1)$ for all $i\in\N$ and apply Lemma \ref{lemma:holder} with those parameters and $B=\mathbf{DI}_1\cap \mathbf{BA}$. Clearly, for $G = S(B)$ we have $\hdim G = \hdim B = 1$ and $G\cap \mathbf{DI}_2 =\emptyset$. So in order to get $\hdim \mathbf{DI}_{1}\setminus\mathbf{DI}_{2} =1$, we need to show that $G\subset \mathbf{DI}_1$. In other words, we need to show that no restricted patterns from the case ${\bf (c) }$ from Theorem \ref{thm:charachterization} will appear in elements of $G$ infinitely many times.

First, we can deal with all patterns of finite lengths, that is, $x,2,y$ or $x,1,1,y$ or patterns of eight forms from the case ${\bf (c) }$ from Theorem \ref{thm:charachterization}. As in the proof of Theorem \ref{theorem:diminusbad:dimension} they clearly cannot appear infinitely many times, as all of them require large partial quotients to be on a fixed distance between each other, while in our case the distance between large partial quotients from two consecutive blocks $A_i$'s grows, and the distance between two large partial quotients in one block $A_i = (n_i,3,2,1,3,4,n_i+1)$ is fixed, but this pattern is not of any of the required forms.

To deal with infinite patterns, we notice that if they appear, they should include large partial quotients from $A_i$'s, as otherwise there will be a contradiction with the fact that the original number is an element  $B=\mathbf{DI}_1\cap \mathbf{BA}$ and hence avoids all restricted patterns. However, as infinite patterns are (almost) symmetrical, they should include at least two large partial quotients, which should be equal to each other, which is impossible because the inserted large partial quotients form a strictly growing sequence of numbers.

Therefore, $\hdim \mathbf{DI}_{1}\setminus\mathbf{DI}_{2} =1$.

Now, let us deal with $\mathbf{DI}_{2}\setminus\mathbf{DI}_{1}$. Once again, we start by recalling that $\hdim \mathbf{DI}_2\cap\mathbf{BA}=1$. So we let $B=\mathbf{DI}_2\cap\mathbf{BA}$ and we will insert some restricted patterns from the characterization of $\mathbf{DI}_1^c$, while also checking that this insertion will not get us out of $\mathbf{DI}_2$.
For the pattern, we choose $n_i=2^{i+100}$ and $A_i = (n_i,1,1,1,2,n_i+1)$, which is a pattern of the first type from the list of eight finite patterns in case ${\bf (c) }$ from Theorem \ref{thm:charachterization}. 

We need to check that no restricted patterns from $\mathbf{DI}_2$ will appear. Clearly, no finite patterns can occur based on the same argument as before. The only new thing to check is that the inserted pattern $(x,1,1,1,2,x+1)$ with $x\to\infty$ does not belong to the list of eight finite restricted patterns in that case. This is indeed the case as this pattern would correspond to rational numbers $\beta^*=\beta=1$, which do not satisfy $\beta^*\cdot\beta=3$.

Lastly, we need to check patterns of growing length. As in case (6) from the proof of Theorem \ref{theorem:diminusbad:dimension}, we will analyze the location of a central partial quotient together with the closest large partial quotient to it. Assume that one of two restricted patterns appeared. We distinguish two cases. 

1) The central partial quotient is located inside one of the $A_i$'s. 
This contradicts the compactness argument, as in this case, the patterns would stabilize for rationals $\beta=\beta^*=1$, which do not satisfy $\beta\cdot\beta^*=3$.

2) The central partial quotient is located outside any $A_i$'s. In this case, by the same argument as in case (6) from the proof of Theorem \ref{theorem:diminusbad:dimension}, we conclude that the closest growing partial quotient to a central partial quotient is located at a finite distance from it, whereas on the other side of the pattern, the closest growing partial quotient is located on a growing distance from the central partial quotient. Thus, once again we get a contradiction with the compactness argument, as the subpattern on one side will tend to a rational number, while the subpattern on the other side will tend to an irrational one.

Therefore, $\hdim \mathbf{DI}_{2}\setminus\mathbf{DI}_{1} =1$.
\end{proof}

\section{Proof of Lemma \ref{lemma:holder}}\label{aadddd}

In order to prove Lemma \ref{lemma:holder}, let us introduce some notation. First, define the set of prefixes from $S'(B'_N)$ for a fixed $N$ as
$$
A_n = \{ (c_1,\ldots,c_n)\in \N^n :\, (c_1,\ldots,c_n) = (x_1,\ldots,x_n) \text{ for some } (x_1,\ldots,x_n,\ldots)\in S'(B'_N) \}.
$$
Next, for any $(c_1,\ldots,c_n)\in A_n$, we let $(\overline{c_1,\ldots,c_n})$ be a finite word obtained by removing all symbols coming from $A_i$ from $(c_1,\ldots,c_n)$. Thus,
$$
(\overline{c_1,\ldots,c_n}) \in \N^{n-\omega(n)}.
$$
Similarly, we set
$$
\overline{q_n}(c_1,\ldots,c_n) = q_{n-\omega(n)}(\overline{c_1,\ldots,c_n}) \text{ and } \overline{I_n}(c_1\ldots,c_n) = I_{n-\omega(n)}(\overline{c_1,\ldots,c_n}). 
$$

\begin{proof}[Proof of Lemma \ref{lemma:holder}.]
We need to show that for any $b \in S(B_N)$, there exist real numbers $\delta,L>0$, such that for any $x,y \in S(B_N)$ with $x\neq y,\, |x-b|<\delta, \, |y-b|<\delta$, one has
\begin{equation}\label{definition:holder}
|P(x)-P(y)| \le L |x-y|^\frac{1}{1+\varepsilon}.
\end{equation}

Fix $b= [b_1,b_2,\ldots]=[a_1,\ldots,a_{n_1-1},A_{1},a_{n_1},\ldots,a_{n_2-\ell_1-1},A_{2},a_{n_2-\ell_1},\ldots] \in S(B_N)$. By \eqref{qE}, we know that
$$
\overline{q_n}^2(a_1,\ldots,a_n) \ge 2^{n-\omega(n)-1}.
$$
For any $\varepsilon>0$, there exists $M>0$, such that for all $n\ge M$, we have
\begin{equation}\label{ineq:notfast}
2^{(n-\omega(n)-1)\varepsilon}\ge 2 ( n^c+1)^{2\omega(n) }.
\end{equation}

Indeed, this easily follows by recalling that $\omega(n) = \overline{o} (n^{1-\epsilon_1})$ for some $\epsilon_1>0$ and using \eqref{c1}.
So using \eqref{length} and \eqref{trivialbound}, we can bound lengths of fundamental cylinders as 
$$
|I_n(a_1,\ldots,a_n)| \ge \frac{1}{2q_n^2(a_1,\ldots,a_n)} \ge \frac{1}{2 {q_{n-\omega(n)}}^2(\overline{a_1,\ldots,a_n}) (n^c+1)^{2\omega(n)}}\ge \frac{1}{q_{n-\omega(n)}^{2(1+\varepsilon)}(\overline{a_1,\ldots,a_n})}\ge |\overline{I_n}(a_1\ldots,a_n)|^{1+\varepsilon}.
$$

Let $J_M(b)=I_M(b_1,\ldots,b_M)$ and let
$$
\rho_b=\operatorname{dist}\bigl(b,\partial J_M(b)\bigr).
$$
Since $b$ is irrational, it is not an endpoint of the cylinder $J_M(b)$, and hence $\rho_b>0$. Choose $0<\delta<\rho_b/2$. Then for any $x=[x_1,x_2,\ldots],y=[y_1,y_2,\ldots] \in (b-\delta,b+\delta) \cap S(B_N)$ with $x\ne y$, the points $x$ and $y$ both lie in $J_M(b)$. Therefore, if $n=n(x,y)$ is the length of their common prefix, then $n\ge M$ and
$$
x_i = y_i \text{ for } i=1,\ldots,n, \text{ but } x_{n+1} \neq y_{n+1}.
$$
Note that all partial quotients coming from $A_i$'s are the same for all elements of $S(B_N)$ and they are located in the same places, thus we always have $x_{n+1},y_{n+1}\le N$.

We have
\begin{equation}\label{almost:holder}
|P(x) - P(y)| \le |\overline{I_n}(x_1\ldots,x_n)| \le |I_n(x_1\ldots,x_n)|^{\frac{1}{1+\varepsilon}}.
\end{equation}
On the other hand, we have
$$
|x-y| \ge C |I_n(x_1\ldots,x_n)| .
$$
To show this, we distinguish two cases.

1) Case $x_{n+2},y_{n+2} \le N$. This can happen if the partial quotients with index $n+2$ are the ones coming from the original number with bounded partial quotients, or, if these partial quotients are coming from inserted $A_i$'s, but they are smaller than $N$. 

Then we can find a cylinder of order $n+2$ between $x$ and $y$. To be precise,
\begin{align*}
|x-y| &\ge \min\{ |I_{n+2}(x_1\ldots,x_n,x_{n+1},x_{n+2}+1)|,|I_{n+2}(x_1\ldots,x_n,y_{n+1},y_{n+2}+1)| \} \\
& \ge |I_{n+2}(x_1\ldots,x_n,N,N+1)| = \frac{1}{\langle x_1,\ldots,x_n,N,N+1 \rangle ( \langle x_1,\ldots,x_n,N,N+1 \rangle + \langle x_1,\ldots,x_n,N\rangle)}\\
& \ge \frac{1}{(N+2)(N+3)\langle x_1,\ldots,x_n,N\rangle^2} \ge \frac{1}{(N+3)^4 \langle x_1,\ldots,x_n\rangle^2} \ge C |I_n(x_1\ldots,x_n)|.
\end{align*}

It turns out that the only alternative to Case 1) is the following case.

2) Case $x_{n+2}=y_{n+2} >N$. This can only happen if partial quotients with index $n+2$ are the ones from $A_i$'s, such that they are greater than $N$.

Without loss of generality, assume that $x<y$. If $n$ is even, then $x_{n+1}> y_{n+1}$. We get
\begin{align*}
|x-y| & >y-[y_1,\ldots,y_n,y_{n+1}+1] > [y_1,\ldots,y_{n+2}]-[y_1,\ldots,y_n,y_{n+1}+1] \\
& = \frac{y_{n+2}-1}{\langle y_1,\ldots,y_{n+2} \rangle( \langle y_1,\ldots,y_{n+1} \rangle+\langle y_1,\ldots,y_{n} \rangle)} \ge \frac{y_{n+2}-1}{2(y_{n+2}+1)\langle y_1,\ldots,y_{n+1} \rangle^2}\\
& \ge \frac{1}{8} \frac{1}{\langle y_1,\ldots,y_{n+1} \rangle^2} \ge \frac{1}{8} \frac{1}{(N+1)^2\langle y_1,\ldots,y_{n} \rangle^2} \ge C |I_n(x_1\ldots,x_n)|.
\end{align*}
If $n$ is odd, then $x_{n+1}<y_{n+1}$ and we proceed as
\begin{align*}
|x-y| & >[x_1,\ldots,x_n,x_{n+1}+1] - x> [x_1,\ldots,x_n,x_{n+1}+1]-[x_1,\ldots,x_{n+2}] \\
& = \frac{x_{n+2}-1}{\langle x_1,\ldots,x_{n+2} \rangle( \langle x_1,\ldots,x_{n+1} \rangle+\langle x_1,\ldots,x_{n} \rangle)} \ge \frac{x_{n+2}-1}{2(x_{n+2}+1)\langle x_1,\ldots,x_{n+1} \rangle^2}\\
& \ge \frac{1}{8} \frac{1}{\langle x_1,\ldots,x_{n+1} \rangle^2} \ge \frac{1}{8} \frac{1}{(N+1)^2\langle x_1,\ldots,x_{n} \rangle^2} \ge C |I_n(x_1\ldots,x_n)|.
\end{align*}

So in any case, we get
$$
|x-y| \ge C |I_n(x_1\ldots,x_n)|.
$$
Combining with \eqref{almost:holder} and letting $L=C^{-\frac{1}{1+\varepsilon}}$, we get
$$
|P(x) - P(y)| \le L|x-y|^{\frac{1}{1+\varepsilon}},
$$
which is exactly \eqref{definition:holder}.

Thus we proved the $\frac{1}{1+\varepsilon}$-H{\"o}lder condition for the map $P$ on $S(B_N)$.

It is well known that if a map $f: A\to B$ is locally $\alpha$-H{\"o}lder, then 
$$
\hdim B =\hdim f(A) \le \frac{1}{\alpha} \hdim A.
$$
We have 
$$
\hdim B_N =\hdim P(S(B_N)) \le (1+\varepsilon)\hdim S(B_N).
$$
Hence 
$$
\hdim S(B_N) \ge \frac{1}{1+\varepsilon} \hdim B_N .
$$
As $S(B) =  \bigcup_{N=1}^\infty S(B_N)$, $B = \bigcup_{N=1}^\infty B_N$ and due to the arbitrariness of $\varepsilon$, we get
$$
 \hdim S(B) \ge \hdim B .
$$

For the opposite inequality, we want to show that the map $S$ is locally Lipschitz.
This means that for any $b \in B_N$, there exist real numbers $\delta,L>0$, such that for any $x,y \in B_N$ with $x\neq y,\, |x-b|<\delta, \, |y-b|<\delta$, one has
\begin{equation}\label{definition:lipschitz}
|S(x)-S(y)| \le L |x-y|.
\end{equation}

Fix $b=[a_1,a_2,a_3,\ldots] \in B_N$. Let $M\ge 1$ and set $J_M(b)=I_M(a_1,\ldots,a_M)$. Let
$$
\rho_b=\operatorname{dist}\bigl(b,\partial J_M(b)\bigr)>0.
$$
Choose $0<\delta<\rho_b/2$ and put $C=\frac{1}{2(N+3)^4}$. Then for any $x=[x_1,x_2,\ldots],y=[y_1,y_2,\ldots] \in (b-\delta,b+\delta) \cap B_N$ with $x\ne y$, the points $x$ and $y$ both lie in $J_M(b)$. Therefore, if $n=n(x,y)$ is the length of their common prefix, then $n\ge M$ and
$$
x_i = y_i \text{ for } i=1,\ldots,n, \text{ but } x_{n+1} \neq y_{n+1}.
$$

As before, we can find a cylinder of order $n+2$ between $x$ and $y$. To be precise,
\begin{align*}
|x-y| &\ge \min\{ |I_{n+2}(x_1\ldots,x_n,x_{n+1},x_{n+2}+1)|,|I_{n+2}(x_1\ldots,x_n,y_{n+1},y_{n+2}+1)| \} \\
& \ge |I_{n+2}(x_1\ldots,x_n,N,N+1)| = \frac{1}{\langle x_1,\ldots,x_n,N,N+1 \rangle ( \langle x_1,\ldots,x_n,N,N+1 \rangle + \langle x_1,\ldots,x_n,N\rangle)}\\
& \ge \frac{1}{(N+2)(N+3)\langle x_1,\ldots,x_n,N\rangle^2} \ge \frac{1}{(N+3)^4 \langle x_1,\ldots,x_n\rangle^2} \ge C |I_n(x_1\ldots,x_n)|.
\end{align*}

On the other hand, as $\omega(n)\ge0$, the number of coinciding first partial quotients of number $S(x)=[\tilde{x}_1,\tilde{x}_2,\ldots]$ and $S(y)=[\tilde{y}_1,\tilde{y}_2,\ldots]$ is at least $n+\omega(n)\ge n$. Thus,
\begin{equation}\label{almost:lipschitz}
|S(x) - S(y)| \le |I_{n+\omega(n)} (\tilde{x}_1,\tilde{x}_2,\ldots,\tilde{x}_{n+\omega(n)})| \le |I_n(x_1\ldots,x_n)|.
\end{equation}
Therefore, 
$$
|S(x) - S(y)| \le C^{-1} |x-y|.
$$
Thus, for every $N$, we have
$$
\hdim S(B_N) \le \hdim B_N.
$$
Taking the supremum over $N$, this implies
$$
 \hdim S(B) \le \hdim B .
$$
Combining the lower and upper bounds together, we get 
$$
 \hdim S(B) = \hdim B .
$$
\end{proof}









    

\vskip+0.3cm
 \section {Proof of Corollary \ref{coroll:setofparameters}}\label{sec:proof:setofparameters}

    In this section, we prove a corollary of Theorem \ref{thm:charachterization} about the set of values of $p$ for which there exist $p$-Dirichlet non-improvable badly approximable numbers.
 \begin{proof}[Proof of Corollary \ref{coroll:setofparameters}]
     
  By Corollary \ref{coroll:bad}, for $p\in (1,2)\cup (p_0,\infty)$ we know that $p$-Dirichlet non-improvable $\alpha$ can be badly approximable only in the case when $ \varsigma_p \in \mathbf{BA}$ is badly approximable itself.
    Recall that $\varsigma_p$ is the unique root of equation (\ref{lat3}) in the interval $(0,1)$.
    Let
    $$
    \mathfrak{g} (p,\varsigma ) = \varsigma^p+(1+\varsigma)^p.
    $$
    So  (\ref{lat3})  can be written as 
        \begin{equation}\label{implicit}
    \mathfrak{g} (p,\varsigma ) = 2.
    \end{equation}
    We want to show that  (\ref{implicit}) gives an explicit smooth decreasing function $ p =   \mathfrak{p}(\varsigma)$ with continuous derivative $\mathfrak{p}'(\varsigma)$ bounded from zero on each finite interval. This will imply the result of  Corollary \ref{coroll:setofparameters}. Indeed, for any segment
    $[a,b] \subset  (1,2)\cup (p_0 ,\infty)$
    the set $\frak{P}\cap [a,b]$ is the image of the set  $ \mathbf{BA}\cap \left( \varsigma_b, \varsigma_a\right)$
    which is dense,  has zero measure, and is absolutely winning. So by standard properties of winning sets (see \cite{mcm}), $\frak{P}\cap [a,b]$ is also a dense absolutely winning null set.
    
     \vskip+0.3cm
Now we prove the statement about the explicit function  $ p =   \mathfrak{p}(\varsigma)$.
First, we consider the segment $[1,2]$. It is clear that $ \varsigma_1 = \frac{1}{2}, \varsigma_2 = \frac{\sqrt{3}-1}{2}$. On the rectangle 
    $1\le p \le 2, \,  \varsigma_2=\frac{\sqrt{3}-1}{2}\le \varsigma\le \varsigma_1 =\frac{1}{2}$ both derivatives $\partial \mathfrak{g}/\partial p, \partial \mathfrak{g}/ \partial \varsigma$ are positive and  are bounded from zero.
    For $\partial \mathfrak{g}/ \partial \varsigma$ this is clear and on the rectangle under consideration for $  \partial \mathfrak{g}/\partial p$ we have 
    $$
    \partial \mathfrak{g}/\partial p = \varsigma^p\log \varsigma + (\varsigma+1)^p\log (\varsigma+1) \ge  \varsigma\log \varsigma + (\varsigma+1)\log (\varsigma+1) \ge   \varsigma_2\log \varsigma_2 + (\varsigma_2+1)\log (\varsigma_2+1)>0.
    $$

        Then we consider any $p^*>2$  and the interval  $2\le p \le p^*$. It is clear that
    $ \varsigma_p \le \frac{1}{2}$ and  $\inf_{2\le p\le p^*} \varsigma_p= \varsigma^*>0$.
    Thus,
       $$
    \partial \mathfrak{g}/\partial p = \varsigma^p\log \varsigma + (\varsigma+1)^p\log (\varsigma+1) \ge 
     \varsigma^2\log \varsigma + \log (\varsigma+1) \ge 
     -\frac{\varsigma}{e} +
     \varsigma \log \frac{3}{2} \ge 0.03\, \varsigma^*
       $$
       and everything is proven.
       
 \end{proof}

\newpage

\appendix
\section{Verification of Minkowski conjecture}\label{appendinx:1}

\subsection{Introductory comments}
The statement and the proof of Theorem \ref{thm:charachterization}, the main characterization result, use the solution of the Minkowski conjecture. The proof of it is split across 50 years in a dozen manuscripts, with the final solution claimed in 1986 in \cite{maly} by Glazunov, Golovanov and Malyshev. The solution highly relies on computer-assisted computations, and as these were performed at a time when personal computers did not exist, the underlying code was not provided in the papers, not to mention that some computations were carried out on machines manufactured in the 1960s. In this appendix, we present a preliminary report on our confirmation of the proof of the Minkowski conjecture. In the process of doing so, we realized that the original proof of the Minkowski conjecture contained several mistakes, making some parts of the existing proof invalid. The second author is preparing a more detailed paper \cite{MinkowskiUpcoming} on the full verification process, including the details of the interval calculations, more information on the problems in the original sources, and several technical tricks needed for the computation. We independently made all computations ourselves and claim that the Minkowski conjecture is indeed true, though the original strategy from \cite{maly1985, maly} and many previous works contain non-trivial flaws.

We attach the link to the GitHub project \cite{github} containing Julia code that can be run on any modern device. This code corresponds to interval calculations performed to verify the conjecture.

Nevertheless, we start by recording three significant mistakes (using the notation of the original papers) that we managed to find. 
\begin{itemize}
\item In \cite{maly1985}, Theorem 2 asserts, in particular, that
$$
g'_p(p,\sigma)=\frac{\partial}{\partial p}g(p,\sigma)>0
$$
on the rectangles described in Tables 3 and 4 of that paper. The same sign claim is also used in \cite{maly}, Proposition 3, where it is stated in the form
$$
g'_p(p,\sigma)>0,\qquad 1.99\le p\le 2.05,\qquad
\Sigma_1<\sigma<\Sigma_2,
$$
with $\Sigma_1=1.001802$ and $\Sigma_2=1.7281$ (equation (33) in \cite{maly}). This inequality is false as stated. For example, one can check that
$$
g'_p(1.992,1.728)\approx -1.20569\cdot10^{-3}<0.
$$

\item In \cite{maly}, Proposition 5, the proof near $p=1$ uses the sign condition
$$
\frac{\partial^2 l^{(1)}(p,\tau)}{\partial p^2}<0
$$
on the region denoted there by $\Pi_1$. This inequality fails near the maximum value of $\tau$ on $\Pi_1$. To be precise,
$$
\frac{\partial^2 l^{(1)}}{\partial p^2}(1.01,0.33275100865)
\approx 1.87\cdot 10^{-5}>0.
$$

\item In \cite{maly}, Proposition 6, the authors consider the remaining neighbourhood of the point
$$
p=2,\qquad \sigma=\sigma_p=\sqrt 3.
$$
on a rectangle
$$
1.99995\le p\le 2.00019,\qquad
1.732\le \sigma\le \sigma_{2.00019},
$$
Here are two problems. First, there is a quite important typo in Proposition 6. In the statement, the authors write that the conjecture (MA$_1$) is true, which in their notation corresponds to treating functions as functions of $(p,\tau)$, however, later, in their equation (57) they write that the function depends on $(p,\sigma)$. Those nuances can change the value of the derivative, and it is unclear what exactly was calculated, since no formulae or code were provided. Nevertheless, in either interpretation, the inequality (57) fails for some points of the region $\tilde{P}.$ They claim that for any $(p,\sigma)\in\tilde{P}$ one has
$$
\frac{\partial^2 l^{(0)}(p,\sigma)}{\partial p^2}>0,
$$
 but at $p_0=1.9999500006,\ \sigma_0=\sigma_p(p_0)-5\cdot10^{-6},\ \tau_0=\tau(p_0,\sigma_0)$, numerical evaluation gives
$$
\frac{\partial^2 l^{(0)}}{\partial p^2}(p_0,\sigma_0)
<0 \quad \text{and} \quad \frac{\partial^2 l^{(0)}}{\partial p^2}(p_0,\tau_0)<0.
$$
Also, for the critical point $(p,\sigma)=(2,\sqrt3)$ itself we have
$$
\frac{\partial^2 l^{(0)}}{\partial p^2}(2,\sqrt3)=0.
$$

\end{itemize}

\subsection{Minkowski conjecture: calculations and inequalities.}

Let us first state the Minkowski conjecture and its analytical equivalent reformulation. Here, we partially follow \cite{Malyshev1975}. Consider the two-dimensional region $D_p\subset \R^2$ defined by
$$
|x|^p + |y|^p <1,
$$
dependent on the real parameter $p>1$, and let $\Delta(D_p)$ be the critical determinant of $D_p$. Consider two $D_p$-admissible lattices $\Lambda_p^{(0)}, \Lambda_p^{(1)}$. They have six points on the boundary of $D_p$ and $(1,0)\in \Lambda_p^{(0)},\ (-2^{-1/p},2^{1/p})\in \Lambda_p^{(1)}$. Note that lattices $\Lambda_p^{(0)}, \Lambda_p^{(1)}$ are uniquely defined under those conditions.  The Minkowski conjecture  on critical lattices claims that 
$$
\Delta(D_p) = \min\left( \det(\Lambda_p^{(0)}), \det(\Lambda_p^{(1)})\right),
$$
and that all critical lattices of $D_p$ are contained among  $\Lambda_p^{(0)}, \Lambda_p^{(1)}$ and among lattices symmetrical to $\Lambda_p^{(0)}, \Lambda_p^{(1)}$ with respect to lines $x=0,\ y=0,\ x=y,\ x=-y$.

Cohn \cite{MR35314} has introduced a parametrization of the problem, which can be used to reformulate Minkowski's conjecture in analytical form. Consider a function $$\Delta ( p, \sigma) = ( \tau + \sigma) ( 1+ \tau^p)^{-1/p} ( 1 + \sigma^p)^{-1/p}$$ defined in the region $$ p>1, 1 \le \sigma \le \sigma_p, \sigma_p = (2^p-1)^{1/p}.$$ 

Here $\tau := \tau( p, \sigma)$ is the function uniquely defined by the conditions $A^p + B^p = 1, 0 \le \tau \le t_p$, where

\begin{align*}
    A &= A(p,\sigma, \tau) = (1+\tau^p)^{-1/p} - (1+\sigma^p)^{-1/p}, \\
    B &= B(p,\sigma, \tau) = \tau (1+\tau^p)^{-1/p} + \sigma (1+\sigma^p)^{-1/p},
    \end{align*}
    and $t_p$ is defined by the equation 
    
\begin{equation}\label{mal:taup}
    2(1-t_p)^p = 1+t_p^p, 0 < t_p <1.
\end{equation}

The Minkowski conjecture is equivalent to the following statement, first formulated by Davis in \cite{D}.

\begin{conjecture}\label{conj:minkowski}
    For any real $p,\sigma$ such that $p>1, p \neq 2, 1< \sigma < \sigma_p$, we have 
\begin{equation}\label{mal:mink:conj}
    \Delta (p,\sigma) > \min\{ \Delta(p,1), \Delta(p,\sigma_p) \}.
    \end{equation}

\end{conjecture} 
Note the three special cases $p=1,2,\infty$ for which we have $\Delta(1,\sigma)\equiv \frac12, \, \Delta(2,\sigma)\equiv \frac12 \sqrt3, \, \Delta(\infty,\sigma)\equiv 1$ as noted by Cohn \cite{MR35314}.

In Figure \ref{fig:delta} below, we present the way the function $\Delta(p,\sigma)$ behaves as a function of $\sigma$ for different fixed values of $p$. By $p^{(1)}>2$ we denote a root of the equation $\Delta''_{\sigma^2}|_{\sigma=\sigma_p}=0$, and  $p^{(2)}>2$ a root of the equation $\Delta''_{\sigma^2}|_{\sigma=1}=0$. While it is not yet formally proven that the function $\sigma \longrightarrow \Delta(p,\sigma)$ is monotonic exactly on the intervals shown in Figure \ref{fig:delta} for given $p$, the current numerical calculations and other information on $\Delta(p,\sigma)$ strongly suggest that it is indeed the case. Note that for the purposes of Conjecture \ref{conj:minkowski}, one only claims that the global minimum of the function $\Delta(p,\sigma)$ is located at one of the endpoints in $\sigma$, and the monotonicity questions of $\Delta(p,\sigma)$ are not directly addressed.

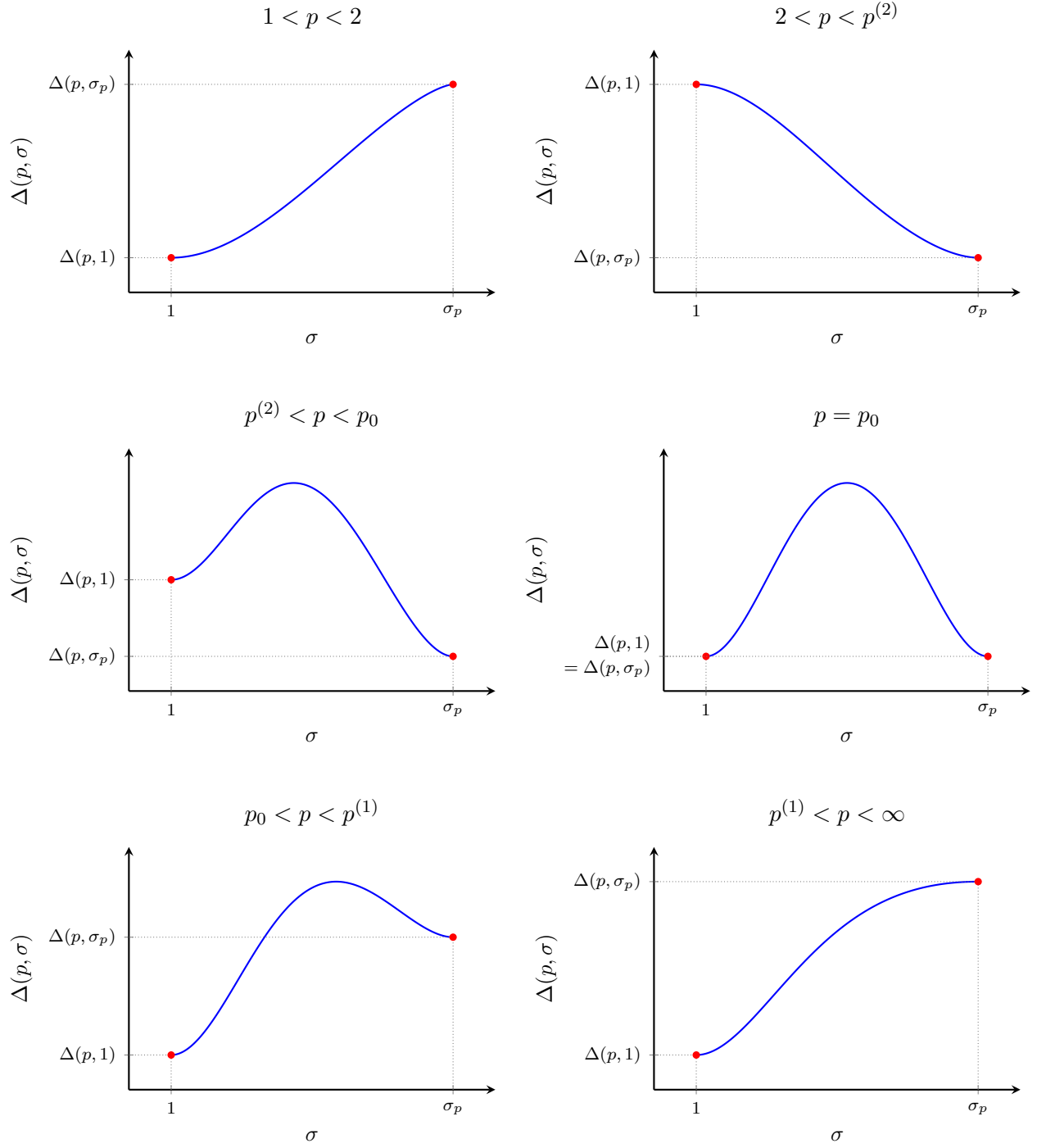
\begin{figure}[h!]
\centering
\begin{tabular}{c c}
\begin{tikzpicture}
    \begin{axis}[
        axis lines = left,
        axis line style = {-stealth, thick},
        width=7.5cm, height=5.5cm,
        clip=false,
        title = {$1 < p <2$},
        xlabel = {$\sigma$}, 
        ylabel = {$\Delta(p, \sigma)$},
        xmin = 0.9257101110, xmax = 1.5695558157,
        ymin = 0.7389837372, ymax = 0.7490745010,
        xtick = {1, 1.4952659267},
        xticklabels = {$1$, $\sigma_p$},
        ytick = {0.7404252749, 0.7476329633},
        yticklabels = {{$\Delta(p, 1)$}, {$\Delta(p, \sigma_p)$}},
        tick label style={font=\scriptsize},
        label style={font=\small},
        title style={font=\normalsize},
        scaled y ticks=false,
        y tick label style={/pgf/number format/fixed}
    ]

    \addplot[blue, thick, smooth] coordinates {
(1.0000001000, 0.7404252749)
(1.0050027851, 0.7404269840)
(1.0100054702, 0.7404320764)
(1.0150081554, 0.7404404992)
(1.0200108405, 0.7404521986)
(1.0250135256, 0.7404671200)
(1.0300162107, 0.7404852081)
(1.0350188958, 0.7405064069)
(1.0400215809, 0.7405306601)
(1.0450242661, 0.7405579105)
(1.0500269512, 0.7405881004)
(1.0550296363, 0.7406211721)
(1.0600323214, 0.7406570669)
(1.0650350065, 0.7406957262)
(1.0700376917, 0.7407370909)
(1.0750403768, 0.7407811017)
(1.0800430619, 0.7408276989)
(1.0850457470, 0.7408768228)
(1.0900484321, 0.7409284135)
(1.0950511172, 0.7409824107)
(1.1000538024, 0.7410387544)
(1.1050564875, 0.7410973843)
(1.1100591726, 0.7411582399)
(1.1150618577, 0.7412212609)
(1.1200645428, 0.7412863869)
(1.1250672279, 0.7413535575)
(1.1300699131, 0.7414227123)
(1.1350725982, 0.7414937911)
(1.1400752833, 0.7415667335)
(1.1450779684, 0.7416414794)
(1.1500806535, 0.7417179685)
(1.1550833387, 0.7417961410)
(1.1600860238, 0.7418759369)
(1.1650887089, 0.7419572964)
(1.1700913940, 0.7420401598)
(1.1750940791, 0.7421244675)
(1.1800967642, 0.7422101602)
(1.1850994494, 0.7422971784)
(1.1901021345, 0.7423854631)
(1.1951048196, 0.7424749552)
(1.2001075047, 0.7425655958)
(1.2051101898, 0.7426573262)
(1.2101128750, 0.7427500876)
(1.2151155601, 0.7428438215)
(1.2201182452, 0.7429384695)
(1.2251209303, 0.7430339733)
(1.2301236154, 0.7431302746)
(1.2351263005, 0.7432273153)
(1.2401289857, 0.7433250372)
(1.2451316708, 0.7434233824)
(1.2501343559, 0.7435222927)
(1.2551370410, 0.7436217101)
(1.2601397261, 0.7437215766)
(1.2651424113, 0.7438218342)
(1.2701450964, 0.7439224246)
(1.2751477815, 0.7440232896)
(1.2801504666, 0.7441243709)
(1.2851531517, 0.7442256098)
(1.2901558368, 0.7443269477)
(1.2951585220, 0.7444283255)
(1.3001612071, 0.7445296839)
(1.3051638922, 0.7446309632)
(1.3101665773, 0.7447321033)
(1.3151692624, 0.7448330438)
(1.3201719475, 0.7449337234)
(1.3251746327, 0.7450340804)
(1.3301773178, 0.7451340524)
(1.3351800029, 0.7452335762)
(1.3401826880, 0.7453325874)
(1.3451853731, 0.7454310211)
(1.3501880583, 0.7455288106)
(1.3551907434, 0.7456258884)
(1.3601934285, 0.7457221854)
(1.3651961136, 0.7458176306)
(1.3701987987, 0.7459121513)
(1.3752014838, 0.7460056727)
(1.3802041690, 0.7460981176)
(1.3852068541, 0.7461894059)
(1.3902095392, 0.7462794546)
(1.3952122243, 0.7463681771)
(1.4002149094, 0.7464554828)
(1.4052175946, 0.7465412764)
(1.4102202797, 0.7466254573)
(1.4152229648, 0.7467079187)
(1.4202256499, 0.7467885464)
(1.4252283350, 0.7468672178)
(1.4302310201, 0.7469438002)
(1.4352337053, 0.7470181485)
(1.4402363904, 0.7470901028)
(1.4452390755, 0.7471594852)
(1.4502417606, 0.7472260945)
(1.4552444457, 0.7472897004)
(1.4602471309, 0.7473500340)
(1.4652498160, 0.7474067740)
(1.4702525011, 0.7474595257)
(1.4752551862, 0.7475077837)
(1.4802578713, 0.7475508654)
(1.4852605564, 0.7475877645)
(1.4902632416, 0.7476167472)
(1.4952659267, 0.7476329633)

    };

    \draw[densely dotted, gray] (axis cs:1, 0.7389837372) -- (axis cs:1, 0.7404252749);
    \draw[densely dotted, gray] (axis cs:0.9257101110, 0.7404252749) -- (axis cs:1, 0.7404252749);
    \node[circle, fill=red, inner sep=1.2pt] at (axis cs:1, 0.7404252749) {};

    \draw[densely dotted, gray] (axis cs:1.4952659267, 0.7389837372) -- (axis cs:1.4952659267, 0.7476329633);
    \draw[densely dotted, gray] (axis cs:0.9257101110, 0.7476329633) -- (axis cs:1.4952659267, 0.7476329633);
    \node[circle, fill=red, inner sep=1.2pt] at (axis cs:1.4952659267, 0.7476329633) {};

    \end{axis}
\end{tikzpicture}
&
\begin{tikzpicture}
    \begin{axis}[
        axis lines = left,
        axis line style = {-stealth, thick},
        width=7.5cm, height=5.5cm,
        clip=false,
        title = {$2<p < p^{(2)}$},
        xlabel = {$\sigma$}, 
        ylabel = {$\Delta(p, \sigma)$},
        xmin = 0.8856421197, xmax = 1.8767437489,
        ymin = 0.8810539342, ymax = 0.8820269346,
        xtick = {1, 1.7623858686},
        xticklabels = {$1$, $\sigma_p$},
        ytick = {0.8818879346, 0.8811929343},
        yticklabels = {{$\Delta(p, 1)$}, {$\Delta(p, \sigma_p)$}},
        tick label style={font=\scriptsize},
        label style={font=\small},
        title style={font=\normalsize},
        scaled y ticks=false,
        y tick label style={/pgf/number format/fixed}
    ]

    \addplot[blue, thick, smooth] coordinates {
(1.0000001000, 0.8818879346)
(1.0077009663, 0.8818876998)
(1.0154018327, 0.8818870029)
(1.0231026990, 0.8818858549)
(1.0308035654, 0.8818842672)
(1.0385044317, 0.8818822512)
(1.0462052981, 0.8818798183)
(1.0539061644, 0.8818769799)
(1.0616070308, 0.8818737478)
(1.0693078971, 0.8818701333)
(1.0770087635, 0.8818661483)
(1.0847096298, 0.8818618042)
(1.0924104962, 0.8818571125)
(1.1001113625, 0.8818520850)
(1.1078122289, 0.8818467330)
(1.1155130952, 0.8818410680)
(1.1232139616, 0.8818351015)
(1.1309148279, 0.8818288447)
(1.1386156943, 0.8818223089)
(1.1463165606, 0.8818155052)
(1.1540174270, 0.8818084448)
(1.1617182933, 0.8818011387)
(1.1694191597, 0.8817935976)
(1.1771200260, 0.8817858324)
(1.1848208924, 0.8817778538)
(1.1925217587, 0.8817696723)
(1.2002226251, 0.8817612982)
(1.2079234914, 0.8817527421)
(1.2156243578, 0.8817440139)
(1.2233252241, 0.8817351238)
(1.2310260905, 0.8817260817)
(1.2387269568, 0.8817168974)
(1.2464278232, 0.8817075806)
(1.2541286895, 0.8816981409)
(1.2618295559, 0.8816885875)
(1.2695304222, 0.8816789299)
(1.2772312886, 0.8816691771)
(1.2849321549, 0.8816593382)
(1.2926330213, 0.8816494221)
(1.3003338876, 0.8816394374)
(1.3080347540, 0.8816293929)
(1.3157356203, 0.8816192970)
(1.3234364867, 0.8816091581)
(1.3311373530, 0.8815989843)
(1.3388382194, 0.8815887840)
(1.3465390857, 0.8815785650)
(1.3542399521, 0.8815683352)
(1.3619408184, 0.8815581023)
(1.3696416848, 0.8815478741)
(1.3773425511, 0.8815376580)
(1.3850434175, 0.8815274615)
(1.3927442838, 0.8815172919)
(1.4004451502, 0.8815071563)
(1.4081460165, 0.8814970620)
(1.4158468829, 0.8814870159)
(1.4235477492, 0.8814770249)
(1.4312486156, 0.8814670959)
(1.4389494819, 0.8814572356)
(1.4466503483, 0.8814474506)
(1.4543512146, 0.8814377477)
(1.4620520810, 0.8814281332)
(1.4697529473, 0.8814186136)
(1.4774538136, 0.8814091953)
(1.4851546800, 0.8813998847)
(1.4928555463, 0.8813906880)
(1.5005564127, 0.8813816114)
(1.5082572790, 0.8813726612)
(1.5159581454, 0.8813638435)
(1.5236590117, 0.8813551645)
(1.5313598781, 0.8813466302)
(1.5390607444, 0.8813382469)
(1.5467616108, 0.8813300206)
(1.5544624771, 0.8813219574)
(1.5621633435, 0.8813140636)
(1.5698642098, 0.8813063454)
(1.5775650762, 0.8812988089)
(1.5852659425, 0.8812914606)
(1.5929668089, 0.8812843067)
(1.6006676752, 0.8812773539)
(1.6083685416, 0.8812706087)
(1.6160694079, 0.8812640780)
(1.6237702743, 0.8812577686)
(1.6314711406, 0.8812516878)
(1.6391720070, 0.8812458429)
(1.6468728733, 0.8812402416)
(1.6545737397, 0.8812348920)
(1.6622746060, 0.8812298025)
(1.6699754724, 0.8812249820)
(1.6776763387, 0.8812204400)
(1.6853772051, 0.8812161868)
(1.6930780714, 0.8812122333)
(1.7007789378, 0.8812085916)
(1.7084798041, 0.8812052753)
(1.7161806705, 0.8812022993)
(1.7238815368, 0.8811996812)
(1.7315824032, 0.8811974414)
(1.7392832695, 0.8811956050)
(1.7469841359, 0.8811942046)
(1.7546850022, 0.8811932862)
(1.7623858686, 0.8811929343)

    };

    \draw[densely dotted, gray] (axis cs:1, 0.8810539342) -- (axis cs:1, 0.8818879346);
    \draw[densely dotted, gray] (axis cs:0.8856421197, 0.8818879346) -- (axis cs:1, 0.8818879346);
    \node[circle, fill=red, inner sep=1.2pt] at (axis cs:1, 0.8818879346) {};

    \draw[densely dotted, gray] (axis cs:1.7623858686, 0.8810539342) -- (axis cs:1.7623858686, 0.8811929343);
    \draw[densely dotted, gray] (axis cs:0.8856421197, 0.8811929343) -- (axis cs:1.7623858686, 0.8811929343);
    \node[circle, fill=red, inner sep=1.2pt] at (axis cs:1.7623858686, 0.8811929343) {};

    \end{axis}
\end{tikzpicture}
 \\ [0.5cm]

\begin{tikzpicture}
    \begin{axis}[
        axis lines = left,
        axis line style = {-stealth, thick},
        width=7.5cm, height=5.5cm,
        clip=false,
        title = {$p^{(2)} < p< p_0 $},
        xlabel = {$\sigma$}, 
        ylabel = {$\Delta(p, \sigma)$},
        xmin = 0.8708877514, xmax = 1.9898605727,
        ymin = 0.9303521819, ymax = 0.9305060431,
        xtick = {1, 1.8607483241},
        xticklabels = {$1$, $\sigma_p$},
        ytick = {0.9304227117, 0.9303741620},
        yticklabels = {{$\Delta(p, 1)$}, {$\Delta(p, \sigma_p)$}},
        tick label style={font=\scriptsize},
        label style={font=\small},
        title style={font=\normalsize},
        scaled y ticks=false,
        y tick label style={/pgf/number format/fixed}
    ]

    \addplot[blue, thick, smooth] coordinates {
(1.0000001000, 0.9304227117)
(1.0086945265, 0.9304228153)
(1.0173889530, 0.9304231222)
(1.0260833795, 0.9304236251)
(1.0347778060, 0.9304243162)
(1.0434722325, 0.9304251867)
(1.0521666590, 0.9304262271)
(1.0608610855, 0.9304274273)
(1.0695555120, 0.9304287764)
(1.0782499386, 0.9304302636)
(1.0869443651, 0.9304318771)
(1.0956387916, 0.9304336052)
(1.1043332181, 0.9304354360)
(1.1130276446, 0.9304373571)
(1.1217220711, 0.9304393565)
(1.1304164976, 0.9304414218)
(1.1391109241, 0.9304435408)
(1.1478053506, 0.9304457014)
(1.1564997771, 0.9304478916)
(1.1651942036, 0.9304500996)
(1.1738886301, 0.9304523139)
(1.1825830566, 0.9304545231)
(1.1912774831, 0.9304567164)
(1.1999719096, 0.9304588830)
(1.2086663361, 0.9304610126)
(1.2173607626, 0.9304630955)
(1.2260551892, 0.9304651222)
(1.2347496157, 0.9304670836)
(1.2434440422, 0.9304689710)
(1.2521384687, 0.9304707766)
(1.2608328952, 0.9304724924)
(1.2695273217, 0.9304741115)
(1.2782217482, 0.9304756272)
(1.2869161747, 0.9304770332)
(1.2956106012, 0.9304783239)
(1.3043050277, 0.9304794942)
(1.3129994542, 0.9304805393)
(1.3216938807, 0.9304814551)
(1.3303883072, 0.9304822379)
(1.3390827337, 0.9304828846)
(1.3477771602, 0.9304833924)
(1.3564715867, 0.9304837592)
(1.3651660132, 0.9304839831)
(1.3738604398, 0.9304840629)
(1.3825548663, 0.9304839978)
(1.3912492928, 0.9304837874)
(1.3999437193, 0.9304834318)
(1.4086381458, 0.9304829315)
(1.4173325723, 0.9304822872)
(1.4260269988, 0.9304815005)
(1.4347214253, 0.9304805728)
(1.4434158518, 0.9304795064)
(1.4521102783, 0.9304783037)
(1.4608047048, 0.9304769674)
(1.4694991313, 0.9304755008)
(1.4781935578, 0.9304739072)
(1.4868879843, 0.9304721906)
(1.4955824108, 0.9304703549)
(1.5042768373, 0.9304684046)
(1.5129712638, 0.9304663444)
(1.5216656903, 0.9304641792)
(1.5303601169, 0.9304619144)
(1.5390545434, 0.9304595552)
(1.5477489699, 0.9304571076)
(1.5564433964, 0.9304545774)
(1.5651378229, 0.9304519708)
(1.5738322494, 0.9304492942)
(1.5825266759, 0.9304465542)
(1.5912211024, 0.9304437576)
(1.5999155289, 0.9304409113)
(1.6086099554, 0.9304380226)
(1.6173043819, 0.9304350988)
(1.6259988084, 0.9304321473)
(1.6346932349, 0.9304291758)
(1.6433876614, 0.9304261923)
(1.6520820879, 0.9304232046)
(1.6607765144, 0.9304202209)
(1.6694709409, 0.9304172494)
(1.6781653675, 0.9304142987)
(1.6868597940, 0.9304113773)
(1.6955542205, 0.9304084939)
(1.7042486470, 0.9304056575)
(1.7129430735, 0.9304028771)
(1.7216375000, 0.9304001619)
(1.7303319265, 0.9303975213)
(1.7390263530, 0.9303949649)
(1.7477207795, 0.9303925025)
(1.7564152060, 0.9303901442)
(1.7651096325, 0.9303879002)
(1.7738040590, 0.9303857811)
(1.7824984855, 0.9303837979)
(1.7911929120, 0.9303819617)
(1.7998873385, 0.9303802843)
(1.8085817650, 0.9303787782)
(1.8172761915, 0.9303774563)
(1.8259706181, 0.9303763327)
(1.8346650446, 0.9303754225)
(1.8433594711, 0.9303747429)
(1.8520538976, 0.9303743139)
(1.8607483241, 0.9303741620)

    };

    \draw[densely dotted, gray] (axis cs:1, 0.9303521819) -- (axis cs:1, 0.9304227117);
    \draw[densely dotted, gray] (axis cs:0.8708877514, 0.9304227117) -- (axis cs:1, 0.9304227117);
    \node[circle, fill=red, inner sep=1.2pt] at (axis cs:1, 0.9304227117) {};

    \draw[densely dotted, gray] (axis cs:1.8607483241, 0.9303521819) -- (axis cs:1.8607483241, 0.9303741620);
    \draw[densely dotted, gray] (axis cs:0.8708877514, 0.9303741620) -- (axis cs:1.8607483241, 0.9303741620);
    \node[circle, fill=red, inner sep=1.2pt] at (axis cs:1.8607483241, 0.9303741620) {};

    \end{axis}
\end{tikzpicture}
 & 
\begin{tikzpicture}
    \begin{axis}[
       axis lines = left,
        axis line style = {-stealth, thick},
        width=7.5cm, height=5.5cm,
        clip=false,
        title = {$p = p_0$ },
        xlabel = {$\sigma$}, 
        ylabel = {$\Delta(p, \sigma)$},
        xmin = 0.8707323752, xmax = 1.9910517902,
        ymin = 0.9308770775, ymax = 0.9309996316,
        xtick = {1, 1.8617841654},
        xticklabels = {$1$, $\sigma_p$},
        ytick = {0.9308945852},
        yticklabels = {{\shortstack[r]{$\Delta(p, 1)$ \\ $= \Delta(p, \sigma_p)$}}},
        tick label style={font=\scriptsize},
        label style={font=\small},
        title style={font=\normalsize},
        scaled y ticks=false,
        y tick label style={/pgf/number format/fixed}
    ]

    \addplot[blue, thick, smooth] coordinates {
(1.0000001000, 0.9308945852)
(1.0087049895, 0.9308947193)
(1.0174098791, 0.9308951062)
(1.0261147686, 0.9308957290)
(1.0348196582, 0.9308965733)
(1.0435245477, 0.9308976256)
(1.0522294373, 0.9308988730)
(1.0609343268, 0.9309003028)
(1.0696392164, 0.9309019026)
(1.0783441059, 0.9309036600)
(1.0870489955, 0.9309055627)
(1.0957538850, 0.9309075985)
(1.1044587746, 0.9309097553)
(1.1131636641, 0.9309120210)
(1.1218685537, 0.9309143837)
(1.1305734432, 0.9309168315)
(1.1392783328, 0.9309193529)
(1.1479832223, 0.9309219362)
(1.1566881119, 0.9309245700)
(1.1653930014, 0.9309272434)
(1.1740978910, 0.9309299452)
(1.1828027805, 0.9309326648)
(1.1915076701, 0.9309353917)
(1.2002125596, 0.9309381157)
(1.2089174492, 0.9309408269)
(1.2176223387, 0.9309435155)
(1.2263272283, 0.9309461723)
(1.2350321178, 0.9309487882)
(1.2437370074, 0.9309513545)
(1.2524418969, 0.9309538627)
(1.2611467865, 0.9309563049)
(1.2698516760, 0.9309586734)
(1.2785565656, 0.9309609607)
(1.2872614551, 0.9309631600)
(1.2959663447, 0.9309652645)
(1.3046712342, 0.9309672680)
(1.3133761238, 0.9309691647)
(1.3220810133, 0.9309709490)
(1.3307859029, 0.9309726159)
(1.3394907924, 0.9309741606)
(1.3481956820, 0.9309755787)
(1.3569005715, 0.9309768663)
(1.3656054611, 0.9309780198)
(1.3743103506, 0.9309790361)
(1.3830152402, 0.9309799124)
(1.3917201297, 0.9309806462)
(1.4004250193, 0.9309812357)
(1.4091299088, 0.9309816792)
(1.4178347984, 0.9309819755)
(1.4265396879, 0.9309821239)
(1.4352445775, 0.9309821239)
(1.4439494670, 0.9309819755)
(1.4526543566, 0.9309816792)
(1.4613592461, 0.9309812357)
(1.4700641357, 0.9309806462)
(1.4787690252, 0.9309799124)
(1.4874739148, 0.9309790361)
(1.4961788043, 0.9309780198)
(1.5048836939, 0.9309768663)
(1.5135885834, 0.9309755787)
(1.5222934730, 0.9309741606)
(1.5309983625, 0.9309726159)
(1.5397032520, 0.9309709490)
(1.5484081416, 0.9309691647)
(1.5571130311, 0.9309672680)
(1.5658179207, 0.9309652645)
(1.5745228102, 0.9309631600)
(1.5832276998, 0.9309609607)
(1.5919325893, 0.9309586734)
(1.6006374789, 0.9309563049)
(1.6093423684, 0.9309538627)
(1.6180472580, 0.9309513545)
(1.6267521475, 0.9309487882)
(1.6354570371, 0.9309461723)
(1.6441619266, 0.9309435155)
(1.6528668162, 0.9309408269)
(1.6615717057, 0.9309381157)
(1.6702765953, 0.9309353917)
(1.6789814848, 0.9309326648)
(1.6876863744, 0.9309299452)
(1.6963912639, 0.9309272434)
(1.7050961535, 0.9309245700)
(1.7138010430, 0.9309219362)
(1.7225059326, 0.9309193529)
(1.7312108221, 0.9309168315)
(1.7399157117, 0.9309143837)
(1.7486206012, 0.9309120210)
(1.7573254908, 0.9309097553)
(1.7660303803, 0.9309075985)
(1.7747352699, 0.9309055627)
(1.7834401594, 0.9309036600)
(1.7921450490, 0.9309019026)
(1.8008499385, 0.9309003028)
(1.8095548281, 0.9308988730)
(1.8182597176, 0.9308976256)
(1.8269646072, 0.9308965733)
(1.8356694967, 0.9308957290)
(1.8443743863, 0.9308951062)
(1.8530792758, 0.9308947193)
(1.8617841654, 0.9308945852)

    };

    \draw[densely dotted, gray] (axis cs:1, 0.9308770775) -- (axis cs:1, 0.9308945852);
    \draw[densely dotted, gray] (axis cs:0.8707323752, 0.9308945852) -- (axis cs:1, 0.9308945852);
    \node[circle, fill=red, inner sep=1.2pt] at (axis cs:1, 0.9308945852) {};

    \draw[densely dotted, gray] (axis cs:1.8617841654, 0.9308770775) -- (axis cs:1.8617841654, 0.9308945852);
    \draw[densely dotted, gray] (axis cs:0.8707323752, 0.9308945852) -- (axis cs:1.8617841654, 0.9308945852);
    \node[circle, fill=red, inner sep=1.2pt] at (axis cs:1.8617841654, 0.9308945852) {};

    \end{axis}
\end{tikzpicture}
 \\ [0.5cm]

\begin{tikzpicture}
    \begin{axis}[
        axis lines = left,
        axis line style = {-stealth, thick},
        width=7.5cm, height=5.5cm,
        clip=false,
        title = {$p_0 < p < p^{(1)}$},
        xlabel = {$\sigma$}, 
        ylabel = {$\Delta(p, \sigma)$},
        xmin = 0.8703394101, xmax = 1.9940645226,
        ymin = 0.9320601002, ymax = 0.9322859176,
        xtick = {1, 1.8644039327},
        xticklabels = {$1$, $\sigma_p$},
        ytick = {0.9320923599, 0.9322019663},
        yticklabels = {{$\Delta(p, 1)$}, {$\Delta(p, \sigma_p)$}},
        tick label style={font=\scriptsize},
        label style={font=\small},
        title style={font=\normalsize},
        scaled y ticks=false,
        y tick label style={/pgf/number format/fixed}
    ]

    \addplot[blue, thick, smooth] coordinates {
(1.0000001000, 0.9320923599)
(1.0087314518, 0.9320925407)
(1.0174628037, 0.9320930763)
(1.0261941555, 0.9320939552)
(1.0349255074, 0.9320951648)
(1.0436568592, 0.9320966918)
(1.0523882111, 0.9320985218)
(1.0611195629, 0.9321006397)
(1.0698509148, 0.9321030299)
(1.0785822666, 0.9321056763)
(1.0873136185, 0.9321085621)
(1.0960449703, 0.9321116706)
(1.1047763221, 0.9321149845)
(1.1135076740, 0.9321184866)
(1.1222390258, 0.9321221595)
(1.1309703777, 0.9321259859)
(1.1397017295, 0.9321299486)
(1.1484330814, 0.9321340304)
(1.1571644332, 0.9321382144)
(1.1658957851, 0.9321424840)
(1.1746271369, 0.9321468228)
(1.1833584887, 0.9321512148)
(1.1920898406, 0.9321556445)
(1.2008211924, 0.9321600967)
(1.2095525443, 0.9321645567)
(1.2182838961, 0.9321690103)
(1.2270152480, 0.9321734438)
(1.2357465998, 0.9321778441)
(1.2444779517, 0.9321821987)
(1.2532093035, 0.9321864955)
(1.2619406554, 0.9321907231)
(1.2706720072, 0.9321948708)
(1.2794033590, 0.9321989284)
(1.2881347109, 0.9322028863)
(1.2968660627, 0.9322067356)
(1.3055974146, 0.9322104678)
(1.3143287664, 0.9322140754)
(1.3230601183, 0.9322175512)
(1.3317914701, 0.9322208887)
(1.3405228220, 0.9322240820)
(1.3492541738, 0.9322271259)
(1.3579855257, 0.9322300158)
(1.3667168775, 0.9322327473)
(1.3754482293, 0.9322353172)
(1.3841795812, 0.9322377223)
(1.3929109330, 0.9322399602)
(1.4016422849, 0.9322420291)
(1.4103736367, 0.9322439277)
(1.4191049886, 0.9322456549)
(1.4278363404, 0.9322472106)
(1.4365676923, 0.9322485948)
(1.4452990441, 0.9322498081)
(1.4540303959, 0.9322508516)
(1.4627617478, 0.9322517267)
(1.4714930996, 0.9322524354)
(1.4802244515, 0.9322529799)
(1.4889558033, 0.9322533631)
(1.4976871552, 0.9322535879)
(1.5064185070, 0.9322536579)
(1.5151498589, 0.9322535769)
(1.5238812107, 0.9322533489)
(1.5326125626, 0.9322529786)
(1.5413439144, 0.9322524705)
(1.5500752662, 0.9322518299)
(1.5588066181, 0.9322510622)
(1.5675379699, 0.9322501728)
(1.5762693218, 0.9322491678)
(1.5850006736, 0.9322480532)
(1.5937320255, 0.9322468355)
(1.6024633773, 0.9322455213)
(1.6111947292, 0.9322441173)
(1.6199260810, 0.9322426307)
(1.6286574329, 0.9322410687)
(1.6373887847, 0.9322394387)
(1.6461201365, 0.9322377484)
(1.6548514884, 0.9322360054)
(1.6635828402, 0.9322342179)
(1.6723141921, 0.9322323939)
(1.6810455439, 0.9322305418)
(1.6897768958, 0.9322286701)
(1.6985082476, 0.9322267873)
(1.7072395995, 0.9322249024)
(1.7159709513, 0.9322230242)
(1.7247023031, 0.9322211619)
(1.7334336550, 0.9322193248)
(1.7421650068, 0.9322175226)
(1.7508963587, 0.9322157647)
(1.7596277105, 0.9322140613)
(1.7683590624, 0.9322124223)
(1.7770904142, 0.9322108584)
(1.7858217661, 0.9322093801)
(1.7945531179, 0.9322079986)
(1.8032844698, 0.9322067253)
(1.8120158216, 0.9322055722)
(1.8207471734, 0.9322045520)
(1.8294785253, 0.9322036781)
(1.8382098771, 0.9322029649)
(1.8469412290, 0.9322024286)
(1.8556725808, 0.9322020877)
(1.8644039327, 0.9322019663)

    };

    \draw[densely dotted, gray] (axis cs:1, 0.9320601002) -- (axis cs:1, 0.9320923599);
    \draw[densely dotted, gray] (axis cs:0.8703394101, 0.9320923599) -- (axis cs:1, 0.9320923599);
    \node[circle, fill=red, inner sep=1.2pt] at (axis cs:1, 0.9320923599) {};

    \draw[densely dotted, gray] (axis cs:1.8644039327, 0.9320601002) -- (axis cs:1.8644039327, 0.9322019663);
    \draw[densely dotted, gray] (axis cs:0.8703394101, 0.9322019663) -- (axis cs:1.8644039327, 0.9322019663);
    \node[circle, fill=red, inner sep=1.2pt] at (axis cs:1.8644039327, 0.9322019663) {};

    \end{axis}
\end{tikzpicture}
 & 
\begin{tikzpicture}
    \begin{axis}[
        axis lines = left,
        axis line style = {-stealth, thick},
        width=7.5cm, height=5.5cm,
        clip=false,
        title = {$p^{(1)} < p < \infty$},
        xlabel = {$\sigma$}, 
        ylabel = {$\Delta(p, \sigma)$},
        xmin = 0.8518988868, xmax = 2.1354418679,
        ymin = 0.9811421513, ymax = 0.9957584150,
        xtick = {1, 1.9873407547},
        xticklabels = {$1$, $\sigma_p$},
        ytick = {0.9832301890, 0.9936703773},
        yticklabels = {{$\Delta(p, 1)$}, {$\Delta(p, \sigma_p)$}},
        tick label style={font=\scriptsize},
        label style={font=\small},
        title style={font=\normalsize},
        scaled y ticks=false,
        y tick label style={/pgf/number format/fixed}
    ]

    \addplot[blue, thick, smooth] coordinates {
(1.0000001000, 0.9832301890)
(1.0099732379, 0.9832375965)
(1.0199463759, 0.9832594992)
(1.0299195138, 0.9832953900)
(1.0398926517, 0.9833447281)
(1.0498657896, 0.9834069442)
(1.0598389276, 0.9834814456)
(1.0698120655, 0.9835676208)
(1.0797852034, 0.9836648437)
(1.0897583413, 0.9837724786)
(1.0997314793, 0.9838898835)
(1.1097046172, 0.9840164143)
(1.1196777551, 0.9841514281)
(1.1296508930, 0.9842942863)
(1.1396240310, 0.9844443575)
(1.1495971689, 0.9846010203)
(1.1595703068, 0.9847636651)
(1.1695434447, 0.9849316967)
(1.1795165827, 0.9851045358)
(1.1894897206, 0.9852816202)
(1.1994628585, 0.9854624067)
(1.2094359964, 0.9856463715)
(1.2194091344, 0.9858330114)
(1.2293822723, 0.9860218445)
(1.2393554102, 0.9862124101)
(1.2493285481, 0.9864042696)
(1.2593016861, 0.9865970065)
(1.2692748240, 0.9867902259)
(1.2792479619, 0.9869835549)
(1.2892210999, 0.9871766424)
(1.2991942378, 0.9873691584)
(1.3091673757, 0.9875607937)
(1.3191405136, 0.9877512597)
(1.3291136516, 0.9879402874)
(1.3390867895, 0.9881276274)
(1.3490599274, 0.9883130488)
(1.3590330653, 0.9884963386)
(1.3690062033, 0.9886773012)
(1.3789793412, 0.9888557578)
(1.3889524791, 0.9890315453)
(1.3989256170, 0.9892045158)
(1.4088987550, 0.9893745362)
(1.4188718929, 0.9895414869)
(1.4288450308, 0.9897052616)
(1.4388181687, 0.9898657662)
(1.4487913067, 0.9900229184)
(1.4587644446, 0.9901766473)
(1.4687375825, 0.9903268919)
(1.4787107204, 0.9904736014)
(1.4886838584, 0.9906167341)
(1.4986569963, 0.9907562569)
(1.5086301342, 0.9908921447)
(1.5186032721, 0.9910243801)
(1.5285764101, 0.9911529526)
(1.5385495480, 0.9912778582)
(1.5485226859, 0.9913990989)
(1.5584958239, 0.9915166825)
(1.5684689618, 0.9916306217)
(1.5784420997, 0.9917409342)
(1.5884152376, 0.9918476419)
(1.5983883756, 0.9919507707)
(1.6083615135, 0.9920503503)
(1.6183346514, 0.9921464137)
(1.6283077893, 0.9922389969)
(1.6382809273, 0.9923281385)
(1.6482540652, 0.9924138799)
(1.6582272031, 0.9924962644)
(1.6682003410, 0.9925753373)
(1.6781734790, 0.9926511459)
(1.6881466169, 0.9927237389)
(1.6981197548, 0.9927931662)
(1.7080928927, 0.9928594789)
(1.7180660307, 0.9929227294)
(1.7280391686, 0.9929829705)
(1.7380123065, 0.9930402559)
(1.7479854444, 0.9930946399)
(1.7579585824, 0.9931461771)
(1.7679317203, 0.9931949223)
(1.7779048582, 0.9932409307)
(1.7878779961, 0.9932842576)
(1.7978511341, 0.9933249580)
(1.8078242720, 0.9933630871)
(1.8177974099, 0.9933986999)
(1.8277705479, 0.9934318510)
(1.8377436858, 0.9934625949)
(1.8477168237, 0.9934909855)
(1.8576899616, 0.9935170765)
(1.8676630996, 0.9935409211)
(1.8776362375, 0.9935625718)
(1.8876093754, 0.9935820807)
(1.8975825133, 0.9935994994)
(1.9075556513, 0.9936148787)
(1.9175287892, 0.9936282688)
(1.9275019271, 0.9936397194)
(1.9374750650, 0.9936492792)
(1.9474482030, 0.9936569965)
(1.9574213409, 0.9936629187)
(1.9673944788, 0.9936670924)
(1.9773676167, 0.9936695635)
(1.9873407547, 0.9936703773)

    };

    \draw[densely dotted, gray] (axis cs:1, 0.9811421513) -- (axis cs:1, 0.9832301890);
    \draw[densely dotted, gray] (axis cs:0.8518988868, 0.9832301890) -- (axis cs:1, 0.9832301890);
    \node[circle, fill=red, inner sep=1.2pt] at (axis cs:1, 0.9832301890) {};

    \draw[densely dotted, gray] (axis cs:1.9873407547, 0.9811421513) -- (axis cs:1.9873407547, 0.9936703773);
    \draw[densely dotted, gray] (axis cs:0.8518988868, 0.9936703773) -- (axis cs:1.9873407547, 0.9936703773);
    \node[circle, fill=red, inner sep=1.2pt] at (axis cs:1.9873407547, 0.9936703773) {};

    \end{axis}
\end{tikzpicture}
 \\ [0.5cm]

\end{tabular}\caption{The behaviour of $\Delta(p,\sigma)$ for fixed $p$ from a given interval}\label{fig:delta}
\end{figure}

For $p\ge6$, the conjecture was proven in \cite{Malyshev1975} by Malyshev and Voronetsky without relying on computers. To be precise, the authors proved the following result
\begin{theorem}
    If $p\ge6, \ 1< \sigma\le \sigma_p$, then 
    $$
    \Delta(p,\sigma) > \Delta(p,1) =: \Delta_p^{(1)}.
    $$
\end{theorem}

For this reason, it is enough to provide computer-assisted proof only for the region 
\begin{equation}\label{mal:region}
1 < p < 6, \quad 1 < \sigma < \sigma_p.
\end{equation}

The general idea is to first split the $p$-region from \eqref{mal:region} into parts where the general behaviour of $\Delta(p,\sigma)$ is somewhat similar inside of each $p$-subregion as per Figure \ref{fig:delta}. 
These would be 
\begin{enumerate}
    \item $1.01 \leq p \leq 1.99$ (the function is supposedly strictly increasing); \\ 

    \item $2.03407 \leq p \leq 2.56809$ (the function is either supposedly strictly decreasing or at least the minimum is attained for $\sigma=\sigma_p$);\\

    \item $2.56809 \leq p \leq 2.642$ (the function has a "bell" shape); \\

    \item $2.642 \leq p \leq 6$ (the function is supposedly strictly increasing or at least the minimum is attained for $\sigma=1$).
    
\end{enumerate}

The remaining small neighborhoods of $p=1$ and $p=2$ are not covered by the cases (1)--(4) above and are dealt with separately. These are actually the hardest cases to deal with.

For each of the cases (1)--(4) (and in fact in the neigbourhoods of $p=1$ and $p=2$, but with extra tricks to be introduced) the idea is to split the $(p,\sigma)$ region into tiny rectangles, and in each one of them show the main inequality \eqref{mal:mink:conj} by numerically bounding $\Delta(p,\sigma)$ together with its derivatives on each of the rectangles. We did so using a technique called Interval Arithmetic, see the book \cite{MR733988} or the Wikipedia page \cite{website:Wiki}. We used the Julia programming language \cite{website:Julia} (version 1.12.4), where we were able to find a pre-existing Interval Arithmetic library \cite{IntervalArithmetic.jl} (version v1.0.2), which has been updated until now. 
All calculations were performed on a personal laptop MacBook Air, Apple M4 chip, 16 GB Memory. The theoretical strategy was fully developed by the authors, and the code for some programs was partially generated with the help of Gemini 3 Pro and ChatGPT 5.5.

In this appendix, we provide four tables with the values of parameters for which we ran our programs. Table \ref{table:1} covers the four main $p$-ranges away from $p=1$ and $p=2$. Tables \ref{table:2} and \ref{table:3} cover the neighbourhoods of $p=2$ and $p=1$, respectively. Table \ref{table:4} records the remaining endpoint check in the one-sided right neighbourhood of $p=2$. The full programs, as well as some other useful information, can be found in the GitHub repository \cite{github}. 

The purpose of the tables and pseudocode below is to make the verification independent of any particular implementation. The attached programs \cite{github} provide our version of interval arithmetic verification of the stated inequalities, but it is possible to verify the inequalities from the tables by other means. Thus, even if the code was unavailable, the proof could be reproduced by verifying the same finite collection of inequalities by interval arithmetic or any other rigorous method.

First, let us introduce the functions used in the verification. Write
$$
C_p(u)=(1+u^p)^{-1/p}.
$$
Then
$$
A=C_p(\tau)-C_p(\sigma),\qquad
B=\tau C_p(\tau)+\sigma C_p(\sigma),
$$
and the critical branch is given by
$$
F(p,\sigma,\tau)=A^p+B^p-1=0.
$$
The derivatives of $F$ with respect to $\sigma$ and $\tau$ are
\begin{align*}
F_\sigma
&=
p(1+\sigma^p)^{-1-1/p}
\left(B^{p-1}+\sigma^{p-1}A^{p-1}\right),\\
F_\tau
&=
p(1+\tau^p)^{-1-1/p}
\left(B^{p-1}-\tau^{p-1}A^{p-1}\right).
\end{align*}

We need to analyze the derivative of $\Delta(p,\sigma)$ with respect to $\sigma$. Direct differentiation gives
\begin{equation}
\frac{\partial \Delta(p,\sigma)}{\partial \sigma}
=
-\frac{g(p,\sigma)}
{(1+\sigma^p)^{1+1/p}
\left(B^{p-1}-\tau^{p-1}A^{p-1}\right)},
\end{equation}
where
\begin{equation}
g(p,\sigma)
=
(1+\sigma^p)^{-1/p}(1-\sigma\tau^{p-1})
\left(B^{p-1}+\sigma^{p-1}A^{p-1}\right)
-
(1+\tau^p)^{-1/p}(1-\tau\sigma^{p-1})
\left(B^{p-1}-\tau^{p-1}A^{p-1}\right).
\end{equation}
Therefore,
$$
\operatorname{sgn} g(p,\sigma)
=
-\operatorname{sgn}\frac{\partial\Delta(p,\sigma)}{\partial\sigma}.
$$
In particular, $g(p,\sigma)<0$ implies that $\Delta(p,\sigma)$ is increasing in $\sigma$.

Following \cite{maly}, we define
\begin{equation}\label{mal:func:h}
h(p,\sigma)
:=
\frac{\partial g(p,\sigma)}{\partial \sigma},
\end{equation}
where, as above, $\tau=\tau(p,\sigma)$ is determined by the critical-branch equation
$$
F(p,\sigma,\tau)=0.
$$

For computer verification purposes in the neighbourhoods of $p=1$ and $p=2$, we also use the following functions:
\begin{align*}
h'_p(p,\sigma)
&:=\frac{\partial h(p,\sigma)}{\partial p},\\
l^{(1)}(p,\sigma)
&:=\Delta(p,\sigma)-\Delta(p,1),\\
(l^{(1)})'_p(p,\sigma)
&:=\frac{\partial l^{(1)}(p,\sigma)}{\partial p},\\
(l^{(1)})''_{pp}(p,\tau)
&:=\frac{\partial^2 l^{(1)}(p,\tau)}{\partial p^2},\\
g'_p(p,\sigma)
&:=\frac{\partial g(p,\sigma)}{\partial p},\\
g'_p(p,\tau)
&:=\frac{\partial g(p,\tau)}{\partial p},\\
g''_{p\tau}(p,\tau)
&:=\frac{\partial^2 g(p,\tau)}{\partial p\,\partial\tau},\\
l^{(0)}(p,\sigma)
&:=\Delta(p,\sigma)-\Delta(p,\sigma_p),\\
(l^{(0)})'_p(p,\sigma)
&:=\frac{\partial l^{(0)}(p,\sigma)}{\partial p},\\
h(p,\sigma)
:=
&\frac{\partial g(p,\tau)}{\partial \tau}.
\end{align*}
When one of the functions above is written in variables $(p,\tau)$, the value of $\sigma=\sigma(p,\tau)$ is determined by the equation $F(p,\sigma,\tau)=0$ on the relevant branch.

It remains to introduce the extra notation used only in Table \ref{table:4}. That table deals with the right-hand neighbourhood of $p=2$, more precisely with
$$
2\le p\le 2.0002,\qquad 1.72\le\sigma\le\sigma_p.
$$
This check is written in the variables $(p,\tau)$, as in Table \ref{table:3}. The new difficulty is the endpoint $\tau=0$. When differentiating with respect to $p$, the factors
$$
\tau^{p-2},\qquad \tau^{p-2}\log\tau
$$
appear. To control these terms on intervals touching $\tau=0$, the program introduces
$$
\delta=p-2,\qquad x=\tau^\delta=\tau^{p-2}.
$$
The variable $x$ is auxiliary and used in the interval estimates near $\tau=0$. In the expression for $h$, we replace
$$
\tau^p=\tau^2x,\qquad
\tau^{p-1}=\tau x,\qquad
\tau^{p-2}=x.
$$
Let $\widehat{h}=\widehat{h}(p,\sigma,\tau,x)$ be the expression for $h$ after these substitutions. In the partial derivatives
$\widehat{h}_p,\
\widehat{h}_\sigma,\
\widehat{h}_x$, the variables $p,\sigma,x$ are treated as independent variables and $\tau$ is kept fixed.

The quantity checked in Table \ref{table:4} is
$$
H_\delta
=
\widehat{h}_p+\widehat{h}_\sigma\sigma_\delta+\widehat{h}_x x_\delta,
$$
where $x_\delta=x\log\tau$.
The value of $\sigma_\delta$ is obtained from the equation $F(p,\sigma,\tau)=0$ by differentiating with respect to $p$, keeping $\tau$ fixed, so that $\sigma_\delta=-\frac{F_p}{F_\sigma}$.
Hence,
$$
H_\delta
=
\widehat{h}_p
-\widehat{h}_\sigma\frac{F_p}{F_\sigma}
+\widehat{h}_x x\log\tau.
$$

\begin{remark}
    Note that for the majority of the functions above, we do not provide a full analytic expression. The reason for that is the length of these analytic formulae, as some of them can take more than a couple of full pages. These expressions can be retrieved either directly from the code from the GitHub project \cite{github} or in the upcoming work \cite{MinkowskiUpcoming} we mentioned above.
\end{remark}

\begin{remark}\label{remark:tau}
Some functions are defined in variables $(p,\tau)$ instead of $(p,\sigma)$. If we treat $\Delta$ as a function of $(p,\tau)$ instead of $(p,\sigma)$, then the equation $A^p+B^p=1$ uniquely defines $\sigma:=\sigma(p,\tau)$ on the relevant branch. Treating $\tau$ as an independent variable is useful in the neighbourhood of $p=1$ and also near the endpoint $\tau=0$ in the neighbourhood of $p=2$. This was also noted (and used) in \cite{maly}. In Table \ref{table:4}, the only additional device is the auxiliary variable $x=\tau^{p-2}$.
\end{remark}


The cases (1)--(4) are covered by Table \ref{table:1}. We covered all values of $p$ from the range $1.01 \leq p \leq 1.99$, $2.03407 \leq p \leq 6$ using three programs: minLeft.jl, minRight.jl, and BellShape.jl.

Below, in Appendix \ref{appendinx:2}, we provide pseudocode for all three programs. 

For the neighbourhoods of $p=1$ and $p=2$ we had to come up with 12 new programs. The main distinctive feature of them, compared to the three programs defined above, is that each of them is designed to verify only one inequality (except for the program covering the right neighbourhood of $p=2)$, instead of automatically cycling through different phases corresponding to different inequalities being verified. The choice of parameters and indication of what exact inequality was checked can be found in Table \ref{table:2}, Table \ref{table:3} and Table \ref{table:4}. We note once again (see Remark \ref{remark:tau}) that in Tables \ref{table:3} and \ref{table:4}, all involved functions are regarded as functions of variables $(p,\tau)$, not $(p,\sigma)$.

\begin{table}[H]
    \centering
    \caption{Choice of parameters and runtime for cases (1)--(4), in $(p,\sigma)$}
    \label{table:1}
    \begin{tabular}{cccccc}
        \toprule
        $p_{\text{start}}$ & $p_{\text{end}}$ & Program & $p_{\text{step}}$ & $\sigma_{\text{step}}$ & time (sec) \\
        \midrule
         1.01 & 1.04 & minLeft.jl & 0.00008 & 0.00005 & 173 \\
        1.04 & 1.88 & minLeft.jl & 0.001 & 0.001 & 76 \\
        1.88 & 1.92399 & minLeft.jl & 0.001 & 0.0001 & 54 \\
        1.92399 & 1.9647 & minLeft.jl & 0.0004 & 0.0001 & 105 \\
        1.9647 & 1.98 & minLeft.jl & 0.0001 & 0.00005 & 231 \\
        1.98 & 1.99 & minLeft.jl & 0.00005 & 0.00001 & 1553 \\
        2.03407 & 2.03767 & minRight.jl & 0.0001 & 0.00001 & 282 \\
        2.03767 & 2.07481 & minRight.jl & 0.0003 & 0.00002 &  497 \\
        2.07481 & 2.48439 & minRight.jl & 0.0008 & 0.00008 & 715 \\
        2.48439 & 2.56 & minRight.jl & 0.0002 & 0.00005 & 616 \\
        2.56 & 2.56809 & minRight.jl & 0.00009 & 0.00002 & 362 \\
        2.56809 & 2.58 & BellShape.jl & 0.00009 & 0.00007 & 157 \\
        2.58 & 2.642 & BellShape.jl & 0.0001 & 0.00007 & 720 \\
        2.642 & 3.06 & minLeft.jl & 0.001 & 0.0001 & 365 \\
        3.06 & 6.0 & minLeft.jl & 0.001 & 0.001 & 260 \\
        \bottomrule
    \end{tabular}
\end{table}

\begin{table}[h]
    \centering
    \caption{Choice of parameters and runtime near $p=2$, in $(p,\sigma)$}
    \label{table:2}
    \begin{tabular}{ccccccccc}
        \toprule
        \text{Fact checked} & $p_{\text{start}}$ & $p_{\text{end}}$ & $\sigma_{\text{start}}$ & $\sigma_{\text{end}}$ & Program & $p_{\text{step}}$ & $\sigma_{\text{step}}$ & time (sec) \\
        \midrule
        $h<-10^{-9}$ & 1.99 & 1.99946 & 1 & 1.0016 & ok2h.jl & 0.00001 & 0.000008  & 23 \\ 
        $h<-10^{-9}$ & 1.99946 & 1.99954 & 1 & 1.0016 & ok2h.jl & 0.000001 & 0.000001  & 17 \\ 
        $h'_p>10^{-9}$ & 1.99954 & 2 & 1 & 1.0016 & ok2hpLEFT.jl & 0.00001 & 0.00001  & 2 \\ 
        $l^{(1)}>10^{-9}$ & 1.99 & 1.99916 & 1.72 & $\sigma_p$ & ok2l1.jl & 0.00001 & 0.000001 & 865 \\
        $l^{(1)}>10^{-9}$ & 1.99916 & 1.9995 & 1.72 & $\sigma_p$ & ok2l1.jl & 0.000006 & 0.0000005 & 123 \\
        $l^{(1)}>10^{-9}$ & 1.9995 & 1.99954 & 1.72 & $\sigma_p$ & ok2l1.jl & 0.000004 & 0.0000002 & 56 \\
        $(l^{(1)})'_p<-10^{-9}$ & 1.99954 & 2 & 1.72 & $\sigma_p$ & ok2l1p.jl & 0.00001 & 0.000001 & 53 \\        
        $g'_p>10^{-9}$ & 1.99 & 1.9965 & 1.0016 & 1.02 & ok2gp.jl & 0.00001 & 0.00001  & 141 \\
        $g'_p>10^{-9}$ & 1.9965 & 2.0076 & 1.0016 & 1.02 & ok2gp.jl & 0.0001 & 0.00001 & 25 \\
        $g'_p>10^{-9}$ & 1.99 & 2.0076 & 1.02 & 1.72 & ok2gp.jl & 0.0001 & 0.0001  & 139 \\ 
        $g'_p>10^{-9}$ & 2.0076 & 2.05 & 1.0016 & 1.0025 & ok2gp.jl & 0.00001 & 0.00001 & 47 \\
        $g'_p>10^{-9}$ & 2.0076 & 2.05 & 1.0025 & 1.72 & ok2gp.jl & 0.0001 & 0.0001 & 341 \\
        $(l^{(0)})'_p>10^{-9}$ & 2 & 2.0002 & 1 & 1.0016 & ok2l0p.jl & 0.00001 & 0.0001  & 1 \\ 
         $l^{(0)}>10^{-9}$ & 2.0002 & 2.00033 & 1 & 1.0016 & ok2l0.jl & 0.000001 & 0.00000005  & 473 \\ 
        $l^{(0)}>10^{-9}$ & 2.00033 & 2.00073 & 1 & 1.0016 & ok2l0.jl & 0.000004 & 0.0000001  & 177 \\ 
        $l^{(0)}>10^{-9}$ & 2.00073 & 2.00741 & 1 & 1.0016 & ok2l0.jl & 0.000008 & 0.0000008  & 188 \\ 
        $l^{(0)}>10^{-9}$ & 2.00741 & 2.03407 & 1 & 1.0016 & ok2l0.jl & 0.0001 & 0.00001  & 6 \\ 
        $h<-10^{-9}$ & 2.0002 & 2.0006 & 1.72 & $\sigma_p$ & ok2h2.jl & 0.00001 & 0.000001  & 46 \\ 
        $h<-10^{-9}$ & 2.0006 & 2.0052 & 1.72 & $\sigma_p$ & ok2h2.jl & 0.00001 & 0.00001  & 56 \\ 
        $h<-10^{-9}$ & 2.0052 & 2.03407 & 1.72 & $\sigma_p$ & ok2h2.jl & 0.0001 & 0.0001  & 6 \\         
        \bottomrule
    \end{tabular}
\end{table}

\begin{table}[h]
    \centering
    \caption{Choice of parameters and runtime near $p=1$, in $(p,\tau)$}
    \label{table:3}
    \begin{tabular}{cccccccc}
    \toprule
\multicolumn{7}{l}{
Common range:
$1\le p\le 1.01,\quad 1\le \sigma\le 1.01377243$.
}
\\
        \midrule
        \text{Fact checked} &  $\tau_{\text{start}}$ & $\tau_{\text{end}}$ & Program & $p_{\text{step}}$& $\sigma_{\text{step}}$ & $\tau_{\text{step}}$ & time (sec) \\
        \midrule
         $(l^{(1)})''_{pp}<-10^{-9}$ &  0.0 & 0.002 & ok1l1pp.jl & 0.0007 & 0.0002 & 0.0002 & 2 \\ 
        $g'_p<-10^{-9}$ &  0.002 & 0.2 & ok1gp.jl & - & - & 0.00001  & 2 \\ 
        $g''_{p\tau}>10^{-9}$ &  0.2 & 1/3 & ok1gptau.jl & 0.0005 & - & 0.0005 & 2 \\ 
        \bottomrule
    \end{tabular}
\end{table}

\begin{table}[h]
\centering
\caption{Choice of parameters and runtime near $p=2$, in $(p,\tau)$.}
\label{table:4}
\small
\begin{tabular}{lcccccc}
\toprule
\multicolumn{7}{l}{
Common range:
$2\le p\le 2.0002,\quad 1.72\le \sigma\le\sigma_p$. Here $p-$ and $\sigma-$ranges are not subdivided.
}
\\
\multicolumn{7}{l}{
Definitions:
$\delta=p-2,\quad x=\tau^\delta$; 
$[\tau^\delta]_{\rm box}$ denotes the interval enclosure on the current $\tau$-box. 
}
\\
\midrule
Fact checked
& $\tau_{\rm start}$ & $\tau_{\rm end}$ & $x$-enclosure & \# of $\tau$-boxes & Program & time (sec)
\\
\midrule

$h<0$ & $0$ & $10^{-12}$ & $0\le x\le 1/2$ & 1 & \multirow{4}{*}{\texttt{ok2gptau.jl}} & \multirow{4}{*}{4}
\\

$H_\delta<0$ & $0$ & $10^{-12}$ & $1/2\le x\le 1$ & 1 & &
\\

$H_\delta<0$ & $10^{-12}$ & $0.004$ & $x\in[\tau^\delta]_{\rm box}$ & 40 & &
\\

$F(p,1.72,0.004)>0,\ F_\sigma>0,\ F_\tau>0$ & $0$ & $0.004$ & -- & 1 & &
\\

\bottomrule
\end{tabular}
\end{table}


\clearpage
\section{Pseudocode for three programs}\label{appendinx:2}

\begin{algorithm}
\caption{Description of \texttt{minLeft.jl}}
\begin{algorithmic}[1]
    \State \textbf{Objective:} Verification of case (2), where $\Delta(p,\sigma)$ is minimized at $\sigma=1$.
    \State \textbf{Input:} Range $[p_{\text{start}}, p_{\text{end}}]$ and step sizes $p_{\text{step}}, \sigma_{\text{step}}$.
   \State \textbf{Discretization:} Construct a descending grid $\mathcal{P} = \{p_{\text{end}},p_{\text{end}}- p_{\text{step}}, p_{\text{end}}- 2p_{\text{step}} ,\dots,p_{\text{start}}\}$.
    \State \textbf{Process:}
    \For{each interval $[P_1,P_2],\ P_1<P_2$ defined by adjacent points in $\mathcal{P}$ }
        \State \textbf{Initialize:} Set verification phase $P \leftarrow 1$. \Comment{Forces remaining $\sigma$-intervals to start from {\bf Check 1}}
        
        \State \textbf{Define:} Construct a descending grid $\Sigma = \{\sigma_{P_2}, \sigma_{P_2}- \sigma_{\text{step}}, \dots, 1.0 \}$.
        \For{each interval $[\varsigma_1,\varsigma_2],\ \varsigma_1<\varsigma_2$ defined by adjacent points in $\Sigma$}
        
            \State 
            Find $\tau$-interval for given $p$-interval $[P_1,P_2]$ and $\sigma$-interval $[\varsigma_1,\varsigma_2]$.
            \State 
            \If{$P = 1$} 
                \State \textbf{Check 1:} Verify inequality $\Delta (p,\sigma) > \Delta(p,1)$ on a $(p,\sigma,\tau)$ box defined by two {\bf for} loops above.
                \If{Check 1 fails for some $(p,\sigma,\tau)$ box}
                    \State $P \leftarrow 2$ \Comment{Forces remaining $\sigma$-intervals to start from {\bf Check 2}}
                \EndIf
            \EndIf

            \State 
            \If{$P = 2$}
                \State \textbf{Check 2:} Verify $g(p,\sigma) < -10^{-9}$ on a $(p,\sigma,\tau)$ box defined by two {\bf for} loops above.
                \If{Check 2 fails for some $(p,\sigma,\tau)$ box}
                    \State $P \leftarrow 3$ \Comment{Forces remaining $\sigma$-intervals to start from {\bf Check 3}}
                \EndIf
            \EndIf

            \State 
            \If{$P = 3$}
                \State \textbf{Check 3:} Verify $h(p,\sigma) := g'_\sigma(p,\sigma) < -10^{-9}$ on a $(p,\sigma,\tau)$ box defined by two {\bf for} loops above.
                \If{Check 3 fails for some $(p,\sigma,\tau)$ box}
                    \State \textbf{Failure:} Record current $p$-interval as inconclusive.
                    \State \textbf{Break:} Exit $\sigma$-loop and proceed to next $p$-interval.
                \EndIf
            \EndIf
            
        \EndFor
    \EndFor
\end{algorithmic}
\end{algorithm}

\begin{algorithm}
\caption{Description of \texttt{minRight.jl}}
\begin{algorithmic}[1]
    \State \textbf{Objective:} Verification of case (2), where $\Delta(p,\sigma)$ is minimized at $\sigma=\sigma_p$.
    \State \textbf{Input:} Range $[p_{\text{start}}, p_{\text{end}}]$ and step sizes $p_{\text{step}}, \sigma_{\text{step}}$.
     \State \textbf{Discretization:} Construct an ascending grid $\mathcal{P} = \{p_{\text{start}}, p_{\text{start}} + p_{\text{step}}, \dots, p_{\text{end}}\}$.
    \State \textbf{Process:}
    \For{each interval $[P_1,P_2],\ P_1<P_2$ defined by adjacent points in $\mathcal{P}$ }
        \State \textbf{Initialize:} Set verification phase $P \leftarrow 1$. \Comment{Forces remaining $\sigma$-intervals to start from {\bf Check 1}}
        
        \State \textbf{Define:} Construct an ascending grid  $\Sigma = \{1.0, 1.0 + \sigma_{\text{step}}, \dots, \sigma_{\max}\}$.
        \For{each interval $[\varsigma_1,\varsigma_2],\ \varsigma_1<\varsigma_2$ defined by adjacent points in $\Sigma$}
            
            \State 
            Find $\tau$-interval for given $p$-interval $[P_1,P_2]$ and $\sigma$-interval $[\varsigma_1,\varsigma_2]$.
            \State 
            \If{$P = 1$} 
                \State \textbf{Check 1:} Verify inequality $\Delta (p,\sigma) > \Delta(p,\sigma_p)$ on a $(p,\sigma,\tau)$ box defined by two {\bf for} loops above.
                \If{Check 1 fails for some $(p,\sigma,\tau)$ box}
                    \State $P \leftarrow 2$ \Comment{Forces remaining $\sigma$-intervals to start from {\bf Check 2}}
                \EndIf
            \EndIf

            \State 
            \If{$P = 2$}
                \State \textbf{Check 2:} Verify $g(p,\sigma) > 10^{-9}$ on a $(p,\sigma,\tau)$ box defined by two {\bf for} loops above.
                \If{Check 2 fails for some $(p,\sigma,\tau)$ box}
                    \State $P \leftarrow 3$ \Comment{Forces remaining $\sigma$-intervals to start from {\bf Check 3}}
                \EndIf
            \EndIf

            \State 
            \If{$P = 3$}
                \State \textbf{Check 3:} Verify $h(p,\sigma) := g'_\sigma(p,\sigma) < -10^{-9}$ on a $(p,\sigma,\tau)$ box defined by two {\bf for} loops above.
                \If{Check 3 fails for some $(p,\sigma,\tau)$ box}
                    \State \textbf{Failure:} Record current $p$-interval as inconclusive.
                    \State \textbf{Break:} Exit $\sigma$-loop and proceed to next $p$-interval.
                \EndIf
            \EndIf
            
        \EndFor
    \EndFor
\end{algorithmic}
\end{algorithm}

\begin{algorithm}
\caption{Description of \texttt{BellShape.jl}}
\begin{algorithmic}[1]
    \State \textbf{Objective:} Verification of case (3), where the graph $\Delta(p,\sigma)$ has a "bell" shape (neighbourhood of $p=p_0$).
    \State \textbf{Input:} Range $[p_{\text{start}}, p_{\text{end}}]$ and step sizes $p_{\text{step}}, \sigma_{\text{step}}$.
    \State \textbf{Discretization:} Construct a descending grid $\mathcal{P} = \{p_{\text{end}},p_{\text{end}}- p_{\text{step}}, p_{\text{end}}- 2p_{\text{step}} ,\dots,p_{\text{start}}\}$.
    \State \textbf{Process:}
    \For{each interval $[P_1,P_2],\ P_1<P_2$ defined by adjacent points in $\mathcal{P}$ }
        \State \textbf{Initialize:} Set verification phase $P \leftarrow 1$. \Comment{Forces remaining $\sigma$-intervals to start from {\bf Check 1}}
        
        \State \textbf{Define:} Construct a descending grid $\Sigma = \{\sigma_{P_2}, \sigma_{P_2}- \sigma_{\text{step}}, \dots, 1.0 \}$.
        \For{each interval $[\varsigma_1,\varsigma_2],\ \varsigma_1<\varsigma_2$ defined by adjacent points in $\Sigma$}
         
            \State 
            Find $\tau$-interval for given $p$-interval $[P_1,P_2]$ and $\sigma$-interval $[\varsigma_1,\varsigma_2]$.
            \State 
            \If{$P = 1$} 
                \State \textbf{Check 1:} Verify $h(p,\sigma) := g'_\sigma(p,\sigma) < -10^{-9}$ on a $(p,\sigma,\tau)$ box defined by two {\bf for} loops above.
                \If{Check 1 fails for some $(p,\sigma,\tau)$ box}
                    \State $P \leftarrow 2$ \Comment{Forces remaining $\sigma$-intervals to start from {\bf Check 2}}
                \EndIf
            \EndIf

            \State 
            \If{$P = 2$}
                \State \textbf{Check 2:} Verify $g(p,\sigma) > 10^{-9}$ on a $(p,\sigma,\tau)$ box defined by two {\bf for} loops above.
                \If{Check 2 fails for some $(p,\sigma,\tau)$ box}
                    \State $P \leftarrow 3$ \Comment{Forces remaining $\sigma$-intervals to start from {\bf Check 3}}
                \EndIf
            \EndIf

            \State 
            \If{$P = 3$}
                \State \textbf{Check 3:} Verify $\Delta (p,\sigma) > \min\{ \Delta(p,1), \Delta(p,\sigma_p) \}$ on a $(p,\sigma,\tau)$ box defined by two {\bf for} loops above.
                \If{Check 3 fails for some $(p,\sigma,\tau)$ box}
                    \State $P \leftarrow 4$ \Comment{Forces remaining $\sigma$-intervals to start from {\bf Check 4}}
                \EndIf
            \EndIf

            \State 
            \If{$P = 4$}
                \State \textbf{Check 4:} Verify $g(p,\sigma) <- 10^{-9}$ on a $(p,\sigma,\tau)$ box defined by two {\bf for} loops above.
                \If{Check 4 fails for some $(p,\sigma,\tau)$ box}
                    \State $P \leftarrow 5$ \Comment{Forces remaining $\sigma$-intervals to start from {\bf Check 5}}
                \EndIf
            \EndIf

            \State 
            \If{$P = 5$}
                \State \textbf{Check 5:} Verify $h(p,\sigma) := g'_\sigma(p,\sigma) < -10^{-9}$ on a $(p,\sigma,\tau)$ box defined by two {\bf for} loops above.
                \If{Check 5 fails for some $(p,\sigma,\tau)$ box}
                    \State \textbf{Failure:} Record current $p$-interval as inconclusive.
                    \State \textbf{Break:} Exit $\sigma$-loop and proceed to next $p$-interval.
                \EndIf
            \EndIf
            
        \EndFor
    \EndFor
\end{algorithmic}
\end{algorithm}

\clearpage

\bibliographystyle{abbrv}
\bibliography{bibliog}

  \end{document}